\documentclass[cleveref,a4paper,UKenglish,numberwithinsect]{lipics-v2021}
\usepackage{cite} % sort cite numbers
\usepackage{amsmath}
\usepackage{amssymb}  
\usepackage{xspace}

\usepackage{tikz}
\usepackage{todonotes}
\tikzstyle{every picture} = [>=latex]
\usetikzlibrary{calc,arrows,decorations.markings,quotes}
\usetikzlibrary{backgrounds}

\def\ca#1{{\cal#1}}

\def\levll#1#2{\lambda[#1](#2)}

\nolinenumbers
\hideLIPIcs

\title{Twin-width of Planar Graphs is at most~8, and some Related Bounds}
\author{Petr Hlin{\v e}n\'y}{Masaryk University, Brno,
  Czech republic}{hlineny@fi.muni.cz}{https://orcid.org/0000-0003-2125-1514}{}
\author{Jan Jedelsk\'y}{Masaryk University, Brno,
  Czech republic}{484988@mail.muni.cz}{https://orcid.org/0000-0001-9585-2553}{}
  
\authorrunning{P.\ Hlin\v{e}n\'y and J.~Jedelsk\'y}
\Copyright{Petr Hlin\v{e}n\'y and Jan Jedelsk\'y}



\begin{document}
\maketitle

\begin{abstract}
Twin-width is a structural width parameter introduced by Bonnet, Kim, Thomass\'e and Watrigant [FOCS 2020],
and has interesting applications in the areas of logic on graphs and in parameterized algorithmics.
Very briefly, the essence of twin-width is in a gradual reduction (a contraction sequence) of the given graph down to a single vertex while maintaining
limited difference of neighbourhoods of the vertices, and it can be seen as widely generalizing several other traditional structural parameters.
While for many natural graph classes it is known that their twin-width is bounded,
published upper bounds on the twin-width in non-trivial cases are very often ``astronomically large''.

We focus on planar graphs, which are known to have bounded twin-width already since the introduction of it, 
but it took some time for the first explicit ``non-astronomical'' upper bounds to come.
Namely, in the order of preprint appearance, 
it was the bound of at most $183$ by Jacob and Pilipczuk [arXiv, January 2022], and $583$ by Bonnet, Kwon and Wood~[arXiv, February 2022].
Subsequent arXiv manuscripts in 2022 improved the bound down to $37$ (Bekos et al.), $11$ and $9$ (both by Hlin\v{e}n\'y).
We further elaborate on the approach used in the latter manuscripts, proving that the twin-width of every planar graph is at most~$8$, 
and construct a witnessing contraction sequence in linear time.
Note that the currently best lower-bound planar example is of twin-width $7$, by Kr\'al' and Lamaison [arXiv, September 2022].
We also prove small explicit upper bounds on the twin-width of bipartite planar and $1$-planar graphs ($6$ and $16$), and of map graphs ($38$).
The common denominator of all these results is the use of a novel specially crafted recursive decomposition of planar graphs,
which may be found useful also in other areas.
\end{abstract}

\section{Introduction}
%%%%%%%%%%%%%%%%%%%%%%%%%%%%%%%%%%%%%%%%%%%%%%%%%%%%%%%%%%%%%%%%%%%%%%%

Twin-width is a relatively new structural width measure of graphs and relational structures introduced in 2020 by Bonnet, Kim, Thomass\'e and Watrigant~\cite{DBLP:conf/focs/Bonnet0TW20}.
Informally, twin-width of a graph measures how diverse the neighbourhoods of the graph vertices are.
E.g., {\em cographs}\,---the graphs which can be built from singleton vertices by repeated operations of a disjoint union and taking the complement,
have the lowest possible value of twin-width,~$0$, which means that the graph can be brought down to a single vertex by successively identifying twin vertices.
(Two vertices $x$ and $y$ are called {\em twins} in a graph $G$ if they have the same neighbours in $V(G) \setminus \{x,y\}$.)
Hence the name, {\em twin-width}, for the parameter.

Importance of this new concept is clearly witnessed by numerous recent papers on the topic, such as the follow-up series
\cite{DBLP:conf/soda/BonnetGKTW21,DBLP:conf/icalp/BonnetG0TW21,DBLP:conf/stoc/BonnetGMSTT22,DBLP:conf/soda/BonnetKRT22,DBLP:conf/iwpec/BonnetC0K0T22,DBLP:journals/jacm/BonnetKTW22}
and more related research papers represented by, e.g.,
\cite{DBLP:journals/siamdm/AhnHKO22,DBLP:conf/iwpec/BalabanH21,DBLP:journals/corr/abs-2102-06880,DBLP:conf/icalp/BergeBD22,DBLP:conf/icalp/GajarskyPPT22,DBLP:conf/stacs/PilipczukSZ22}.

\vspace*{-2ex}\subparagraph{Twin-width definition. }
In general, the concept of twin-width can be considered over arbitrary binary relational structures of a finite signature,
but here we will define it and deal with it for finite {\em simple graphs}, i.e., graphs without loops and multiple edges.
A \emph{trigraph} is a simple graph $G$ in which some edges are marked as {\em red}, and with respect to the red edges only, 
we naturally speak about \emph{red neighbours} and \emph{red degree} in~$G$. 
However, when speaking about edges, neighbours and/or subgraphs without further specification, we count both ordinary and red edges together as one edge set denoted by~$E(G)$.
The edges of $G$ which are not red are sometimes called (and depicted) black for distinction.
For a pair of (possibly not adjacent) vertices $x_1,x_2\in V(G)$, we define a \emph{contraction} of the pair $x_1,x_2$ as the operation
creating a trigraph $G'$ which is the same as $G$ except that $x_1,x_2$ are replaced with a new vertex $x_0$ (said to {\em stem from $x_1,x_2$}) such that:
\begin{itemize}
\item 
the (full) neighbourhood of $x_0$ in $G'$ (i.e., including the red neighbours), denoted by $N_{G'}(x_0)$,
equals the union of the neighbourhoods $N_G(x_1)$ of $x_1$ and $N_G(x_2)$ of $x_2$ in $G$ except $x_1,x_2$ themselves, that is,
$N_{G'}(x_0)=(N_G(x_1)\cup N_G(x_2))\setminus\{x_1,x_2\}$, and
\item 
the red neighbours of $x_0$, denoted here by $N_{G'}^r(x_0)$, inherit all red neighbours of $x_1$ and of $x_2$ and add those in $N_G(x_1)\Delta N_G(x_2)$,
that is, $N_{G'}^r(x_0)=\big(N_{G}^r(x_1)\cup N_G^r(x_2)\cup(N_G(x_1)\Delta N_G(x_2))\big)\setminus\{x_1,x_2\}$, where $\Delta$ denotes the~symmetric set difference.
\end{itemize}
A \emph{contraction sequence} of a trigraph $G$ is a sequence of successive contractions turning~$G$  into a single vertex,
and its \emph{width} $d$ is the maximum red degree of any vertex in any trigraph of the sequence.
We also then say that it is a $d$-contraction sequence of~$G$.
The \emph{twin-width} of a trigraph $G$ is the minimum width over all possible contraction sequences of~$G$.
In other words, a graph has twin-width at most $d$, if and only if it admits a $d$-contraction sequence.

To define the twin-width of an ordinary (simple) graph $G$, we consider $G$ as a trigraph with no red edges.
In a broader sense, if $G$ possibly contains parallel edges, but no loops, we
may define the twin-width of $G$ as the twin-width of the simplification of~$G$.

\vspace*{-2ex}\subparagraph{Algorithmic aspects. }
Twin-width, as a structural width parameter, has naturally many algorithmic applications in the FPT area.
Among the most important ones we mention that the first order (FO) model checking problem --
that is, deciding whether a fixed first-order sentence holds in an input graph -- can be solved in linear FPT-time \cite{DBLP:journals/jacm/BonnetKTW22}.
This and other algorithmic applications assume that a contraction sequence of bounded width is given alongside with the input graph.
Deciding the exact value of twin-width (in particular, twin-width~$4$) is in general NP-hard \cite{DBLP:conf/icalp/BergeBD22},
but for many natural graph classes we know that they are of bounded twin-width.
However, published upper bounds on the twin-width in non-trivial cases are often non-explicit or ``astronomically large'',
and it is not usual that we could, alongside such a bound, compute a contraction sequence of provably
``reasonably small'' width efficiently and practically.
We pay attention to this particular aspect; and we will accompany our fine mathematical upper bounds on the twin-width
with rather simple linear-time algorithms for computing contraction sequences of the claimed widths.

\vspace*{-2ex}\subparagraph{Twin-width of planar graphs. }
The fact that the class of planar graphs is of bounded twin-width was mentioned already in the pioneering paper~\cite{DBLP:conf/focs/Bonnet0TW20}, 
but without giving any explicit upper bound on the twin-width.
The first explicit (numeric) upper bounds on the twin-width of planar graphs have been published only later;
chronologically on arXiv, the bound of $183$ by Jacob and Pilipczuk~\cite{DBLP:conf/wg/JacobP22},
of $583$ by Bonnet, Kwon and Wood~\cite{DBLP:journals/corr/abs-2202-11858} 
(this paper more generally bounds the twin-width of $k$-planar graphs by asymptotic $2^{\ca O(k)}$),
and of $37$ by Bekos, Da Lozzo, Hlin\v{e}n{\'{y}}, and Kaufmann~\cite{DBLP:conf/isaac/BekosLH022}
(this paper bounds the twin-width of so-called $h$-framed graphs by~$\ca O(h)$, 
and the bound of $37$ for planar graphs is explicitly stated only in the preprint version of it).

It is worth to mention that all three papers~\cite{DBLP:conf/wg/JacobP22,DBLP:journals/corr/abs-2202-11858,DBLP:conf/isaac/BekosLH022}, 
more or less explicitly, use the product structure machinery of planar graphs as in~\cite{DBLP:journals/jacm/DujmovicJMMUW20}.
We have then developed, in a similar style, an alternative decomposition-based approach, leading to a single-digit upper bound of~$9$ for all planar graphs in 
\cite{DBLP:journals/corr/abs-2205.05378}, followed by an upper bound of $6$ on the twin-width of {\em bipartite} planar graphs thereafter.
However, the exact approach of \cite{DBLP:journals/corr/abs-2205.05378} seemed to be stuck right at $9$, and new ideas were needed to obtain further improvements.

In this paper we present an improved and more systematic approach (see \Cref{sec:tools}) to the task of proving fine upper bound 
on the twin-width of planar and certain beyond-planar graph classes, which, in particular, further tightens 
the results of \cite{DBLP:journals/corr/abs-2205.05378} and also clarifies some cumbersome technical details of the former paper:

\begin{theorem}\label{thm:twwplanar}
The twin-width of any simple planar graph is at most~$8$, and a corresponding contraction sequence can be found in linear time.
\end{theorem}

Regarding lower bounds, we note that, roughly at the same time as the first version of this work,
Kr\'al' and Lamaison \cite{DBLP:journals/corr/abs-2209-11537} have found a construction and a proof of a planar graph with twin-width~$7$.
Our very recent contribution, not yet fully written up, confirms $7$ to be the right maximum value of the twin-width over all simple planar graphs,
with a very refined, long and complex case analysis targeting the planar case differently.

In this paper, on the other hand, we provide a more general systematic approach providing fine upper bounds on the twin-width
in other graph classes related to planarity.
In this direction, we also bound the twin-width of bipartite planar graphs, for which the upper bound of $6$ stays the same 
as in the previous preprint \cite{DBLP:journals/corr/abs-2205.05378}, but the proof is significantly simpler now:
\begin{theorem}\label{thm:twwbiplanar}
The twin-width of any simple {\em bipartite} planar graph is at most~$6$, and a corresponding contraction sequence can be found in linear time.
\end{theorem}

Both \Cref{thm:twwplanar} and \Cref{thm:twwbiplanar} have appeared with shortened proof sketches in the conference paper \cite{DBLP:conf/icalp/HlinenyJ23}.
In order to demonstrate flexible usability of our general approach to proving fine twin-width bounds, we use it
to approach two other traditional beyond-planar graph classes for which twin-width bounds have already been considered in the literature.

First, we consider $1$-planar graphs.
A graph $G$ is {\em$1$-planar} if $G$ has a drawing in the plane such that every edge is crossed by at most one other edge.
Second to consider, {\em map graphs} are defined as the (simple) intersection graphs of simply connected and internally disjoint regions of the plane.
We refrain from giving more details on this topological definition since we will deal with map graphs using an alternative characterization
via bipartite graph squares~\cite{10.1145/506147.506148} given in \Cref{sec:prooftwwmapplanar}.
For the formulation of our result, we say that a graph $H$ is a {\em square} of a graph $G$ if $V(H)=V(G)$ and 
$u,v\in V(G)$ are adjacent in $H$ if and only if $u\not=v$ are at distance at most two in~$G$.
We now state:

\begin{theorem}\label{thm:tww1planar}
The twin-width of any simple {$1$-planar} graph is at most~$16$.
\end{theorem}
\begin{theorem}\label{thm:mapplanar}
The twin-width of the square of any bipartite planar graph is at most~$63$, and the twin-width of any map graph is at most~$38$.
\end{theorem}

Regarding comparison of \Cref{thm:tww1planar} to previous knowledge,
the improvement is from $80$ of \cite{DBLP:conf/isaac/BekosLH022} down to $16$ now.
In the case of \Cref{thm:mapplanar}, the aforementioned paper of Bonnet, Kwon and Wood~\cite{DBLP:journals/corr/abs-2202-11858}, 
in addition to other results, bounds the twin-width of the squares of graphs of genus $\gamma$ by $\ca O(\gamma^4)$, 
and specifically of the squares of planar graphs by at most~$3614782$, which is now down to $38$ specifically for map graphs.%
\footnote{%
The improvement presented in \Cref{thm:mapplanar} for map graphs is thus huge, but one should also note that a detailed look at the very
simple FO interpretation of map graphs in planar graphs, together with \Cref{thm:twwplanar}, would very likely give another
upper bound in the order of hundreds or perhaps thousands.}

In both \Cref{thm:tww1planar} and \Cref{thm:mapplanar}, obtaining a corresponding contraction sequence
from the proofs would be straightforward in linear time.
However, we would need for that a representation of the input graph according to the definitions in the first place, but
the known recognition algorithm of map graphs by Thorup~\cite{DBLP:conf/focs/Thorup98} is not linear-time,
and the recognition problem of $1$-planar graphs is even NP-hard \cite{DBLP:journals/jgt/KorzhikM13}.
We hence omit the algorithmic aspects of the latter theorems in this paper.

\section{Notation and Tools}\label{sec:tools}
%%%%%%%%%%%%%%%%%%%%%%%%%%%%%%%%%%%%%%%%%%%%%%%%%%%%%%%%%%%%%%%%%%%%%%%

Our graphs are finite, have no loops, but may have parallel edges.
We assume usual set-~and graph-theory notation.
In addition to that, we shortly write $G-x$ or $G-X$ for the subgraph of $G$ obtained by removing the vertex $x$ or all vertices of $X\cap V(G)$
together with their edges, $\,G\setminus e$ for the subgraph of $G$ obtained by deleting the edge~$e$,
$\,G\cup\{x_1,x_2\}$ or $G\cup e$ for the graph obtained from $G$ by adding the edge $e=\{x_1,x_2\}$, where $x_1$, $x_2$ are assumed to be from~$V(G)$,
and $G\Delta H$ for the graph with the edge set $E(G)\Delta E(H)$ and the vertex set $V(G)\cup V(H)$ without isolated vertices.

We start with a few technical definitions and claims needed for the proofs.

\vspace*{-2ex}\subparagraph{BFS layering and contractions. }
Let $G$ be a {\em connected} graph and $r\in V(G)$ a fixed vertex.
The \emph{BFS layering of $G$ determined by~$r$} is the vertex partition $\ca L=(L_0=\{r\},L_1,L_2,\ldots)$ of $G$ such that 
$L_i$ contains all vertices of $G$ at distance exactly $i$ from $r$.
A path $P\subseteq G$ is {\em $r$-geodesic} if $P$ is a subpath of some shortest path from $r$ to any vertex of~$G$
(in particular, $P$ intersects every layer of $\ca L$ in at most one vertex).
Let $T$ be a {\em BFS tree} of $G$ rooted at the vertex $r$ as above
(that is, for every vertex $v\in V(G)$, the distance from $v$ to $r$ is the same in $G$ as in $T$).
A path $P\subseteq G$ is {\em $T$-vertical}, or shortly {\em vertical} with respect to implicit~$T$, if $P$ is a subpath of some root-to-leaf path of~$T$.
Notice that a $T$-vertical path is $r$-geodesic, but the converse may not be true.
Analogously, an edge $e\in E(G)$ is {\em$T$-horizontal}, or horizontal with respect to implicit~$\ca L$, if both ends of $G$ are in the same $\ca L$-layer.

Observe the following trivial claim which will be very useful in the proofs: %(cf.~also \Cref{cl:leveli3}):
\begin{claim}\label{cl:layeri3}
For every edge $\{v,w\}$ of $G$ with $v\in L_i$ and $w\in L_j$, we have $|i-j|\leq 1$, and so
a contraction of a pair of vertices from $L_i$ may create new red edges only to the remaining vertices of~$L_{i-1}\cup L_i\cup L_{i+1}$.
\end{claim}

\vspace*{-2ex}\subparagraph{Plane graphs; Left-aligned BFS trees. }
We will deal with {\em plane graphs}, which are planar graphs with a given (combinatorial) embedding in the plane, and one marked {\em outer face}
(the remaining faces are then {\em bounded}).
A plane graph is a {\em plane triangulation} if every face of its embedding is a triangle.
Likewise, a plane graph is a {\em plane quadrangulation} if every face of its embedding is of length~$4$.
It is easy to turn an embedding of any simple planar graph into a simple plane triangulation by adding vertices and incident edges into each non-triangular face.
Furthermore, twin-width is non-increasing when taking induced subgraphs, and so it suffices to focus on plane triangulations in the proof of
\Cref{thm:twwplanar}, and to similarly deal with plane quadrangulations in the proof of \Cref{thm:twwbiplanar}.

For algorithmic purposes, we represent a plane graph $G$ in the standard combinatorial way---as a graph
(the vertices and their adjacencies) with the counter-clockwise cyclic orders of the incident edges of each vertex,
and we additionally mark the outer face of~$G$.

We note in advance that the coming definition of a special BFS tree may resemble so-called leftist trees in planar graphs~\cite{DBLP:conf/gd/BadentBBC09},
but there is no direct relation between the notions since leftist trees are (typically) not BFS trees.

Consider a connected plane graph $G$, and a BFS tree $T$ spanning $G$ and rooted in a vertex $r$ of the outer face of $G$,
and picture (for clarity) the embedding $G$ such that $r$ is the vertex of $G$ most at the top.
We give formal meaning to the words ``being left of'' independent of a particular picture.
\begin{definition}\label{def:leftuv}
Consider a BFS tree $T$ of a plane graph $G$ as above, and let $u,v\in V(G)$ be (adjacent) vertices such that $\{u,v\}\in E(G)$ but $\{u,v\}\not\in E(T)$.
We say that {\em$u$ is to the left of~$v$} (wrt.~$T$) if the following holds:
If $r'$ is the least common ancestor of $u$ and $v$ in $T$ and $P_{r',u}$ (resp., $P_{r',v}$) denote the vertical path from $r'$ to $u$ (resp., to~$v$),
then the cycle $P_{r',u}\cup P_{r',v}\cup\{u,v\}$ has the triple $(r',u,v)$ in this counter-clockwise cyclic order.
\end{definition}

We add two remarks.
Since $T$ is a BFS tree, $\{u,v\}\not\in E(T)$ means that neither of $u,v$ lies on the vertical path from $r$ to the other.
Furthermore, since $T$ is rooted on the outer face, there is no danger of confusion in \Cref{def:leftuv} even if $\{u,v\}$ is a multiple edge of~$G$.

\begin{definition}[Left-aligned]\label{def:leftaligned}
A BFS tree $T$ of a plane graph $G$ as above, with the BFS layering $\ca L=(L_0,L_1,\ldots)$, is called {\em left-aligned} 
if there is no edge $f=\{u,v\}$ of $G$ such that, for some index $i$, $u\in L_{i-1}$ and $v\in L_i$, and $u$ is to the left of~$v$.
\end{definition}

\begin{figure}[tb]
$$
\begin{tikzpicture}[scale=1]
\tikzstyle{every node}=[draw, shape=circle, minimum size=2.4pt,inner sep=0pt, fill=black]
\node[minimum size=5pt, fill=none, label=above:$r$] at (0,5) {};
\node at (0,5) (a) {};
\node at (-2,4) (b) {}; \node at (-1,4) (bb) {}; \node at (1,4) (bbb) {}; \node at (2,4) (bbbb) {};
\node at (-2.5,3) (c) {}; \node[label=left:$u~$] at (-0.5,3) (cc) {}; \node at (1,3) (ccc) {}; \node at (2.5,3) (cccc) {};
\node at (-2,2) (d) {}; \node at (0,2) (dd) {}; \node[label=right:$\,v$] at (2,2) (ddd) {};
\node at (-1.5,1) (e) {}; \node at (1.5,1) (ee) {};
\tikzstyle{every path}=[draw, very thin]
\draw (a)--(b)--(bb)--(a)--(bbb)--(bbbb)--(a) ;
\draw (c)--(b)--(cc)--(bb)--(ccc)--(bbb)--(cccc)--(bbbb) (cc)--(ccc) ;
\draw (c)--(d)--(dd)--(cc)--(ddd)--(ccc)--(cccc)--(ddd) ;
\draw (d)--(e)--(dd)--(ee)--(ddd) (ee)--(e) ;
\tikzstyle{every path}=[draw, thick]
\draw (b)--(a)--(bb) (bbb)--(a)--(bbbb) ;
\draw (c)--(b)--(cc) (ccc)--(bb) (cccc)--(bbb) ;
\draw (c)--(d)--(e) (cc)--(dd)--(ee) (ccc)--(ddd) ;
\end{tikzpicture}
\qquad\qquad
\begin{tikzpicture}[scale=0.8]
\tikzstyle{every node}=[draw, shape=circle, minimum size=2.4pt,inner sep=0pt, fill=black]
\node[minimum size=5pt, fill=none, label=above:$r$] at (0,5) {};
\node at (0,5) (a) {};
\node at (-1,3.1) (b) {}; \node at (1,3.1) (bb) {};
\node at (0,2) (c) {}; \node at (0,0.7) (cc) {};
\node at (-3,0) (d) {}; \node at (3,0) (dd) {};
\tikzstyle{every path}=[draw, very thin]
\draw (a) to[bend right] (d) (d)--(dd) (dd) to[bend right] (a) ;
\draw (a)--(b)--(d)--(cc)--(dd)--(bb)--(a) ;
\draw (b)--(bb)--(c)--(cc) to[bend left] (b) (b)--(c) (bb) to[bend left] (cc) ;
\tikzstyle{every path}=[draw, thick]
\draw (a) to[bend right] (d) (dd) to[bend right] (a) ;
\draw (a)--(b)--(c) (bb)--(a) (cc)--(d) ;
\end{tikzpicture}
$$
\caption{Left: a BFS tree (in thick lines) rooted at $r$ which is not left-aligned -- \Cref{def:leftaligned} is violated by the edge $\{u,v\}$ 
	(and only by this edge).
	Right: An example of a left-aligned BFS tree (in thick lines) in a plane triangulation rooted at~$r$.}
\label{fig:leftaligni}
\end{figure}
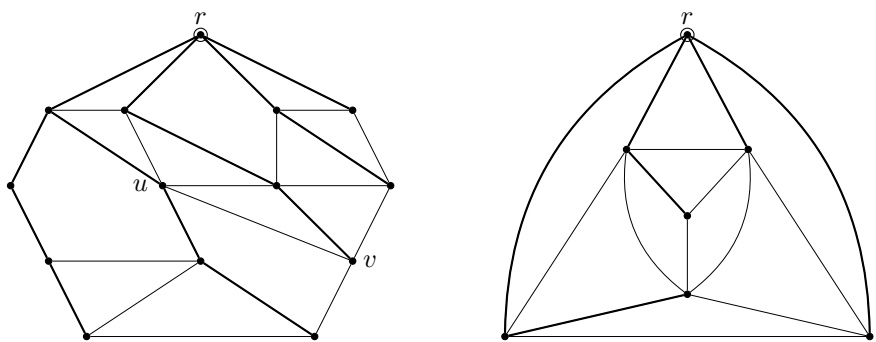

An informal meaning of \Cref{def:leftaligned} is that one cannot choose another BFS tree of $G$ which is ``more to the left'' of~$T$ in the 
geometric picture of $G$ and~$T$, such as by picking the edge $f$ instead of the parental edge of $v$ in~$T$.
See \Cref{fig:leftaligni}.
A left-aligned BFS tree of $G$ is not necessarily unique even with the fixed root~$r$ of~$T$, but this is not a problem for our application.

\begin{lemma}\label{clm:existslal}
Given a connected plane graph $G$ (not necessarily simple), and a vertex $r$ on the outer face, there exists a left-aligned BFS tree of $G$
and it can be found in linear time.
\end{lemma}

\begin{proof}
In the proof, we want to extend the above relation of ``being left of'' to edges~of~$T$ with a common parent.
For a BFS tree T of $G$, consider edges $f_1,f_2\in E(T)$ incident to $v\in V(G)$, 
such that neither of $f_1,f_2$ is the parental edge of $v$ in~$T$ (and note that $T$ has no parallel edges).
We write $f_1\leq_lf_2$ if there exist adjacent vertices $u_1,u_2\in V(G)$ such that $u_1$ is to the left of~$u_2$,
the least common ancestor of $u_1$ and $u_2$ in $T$ is~$v$ and, for $i=1,2$, the edge $f_i$ lies on the vertical path from $u_i$ to~$v$.
Observe the following; if $f_0$ is the parental edge of $v$ in~$T$ (or, in case of $v=r$, $f_0$ is a ``dummy edge'' pointing straight up from~$r$),
then $f_1\leq_lf_2$ implies that the counter-clockwise cyclic order around $v$ is $(f_0,f_1,f_2)$.

We first run a basic linear-time BFS search from $r$ on the outer face to determine the BFS layering $\ca L$~of~$G$.
Then we start the construction of a left-aligned BFS tree $T\subseteq G$ from $T:=\{r\}$, and we recursively (now in a ``DFS manner'') proceed as follows:
\begin{itemize}
\item Having reached a vertex $v\in V(T)\subseteq V(G)$ such that $v\in L_i$, we denote by $X:=(N_G(v)\cap L_{i+1})\setminus V(T)$
all neighbours of $v$ in $L_{i+1}$ which are not in $T$ yet.
We add to $T$ the vertices $X$ and one edge from $v$ to each vertex of~$X$, arbitrarily in case of multiedges.
\item We order the vertices in $X$ using the cyclic order of the chosen edges from $v$ to $X$ (to have it compatible with $\leq_l$ at~$v$),
and in this increasing order we recursively (depth-first, to be precise) call this procedure for them.
\end{itemize}
The result $T$ is clearly a BFS tree of $G$. 

Assume, for a contradiction, that constructed $T$ is not left-aligned, and let $u_1\in L_{i-1}$ and $u_2\in L_i$ be a witness pair of it, 
where $\{u_1,u_2\}\in E(G)$ and $u_1$ is to the left of~$u_2$.
Let $v$ be the least common ancestor of $u_1$ and $u_2$ in $T$, and let $v_1$ and $v_2$ be the children of $v$ on the $T$-paths from $v$ to $u_1$ and $u_2$, respectively.
So, by the definition, $vv_1\leq_lvv_2$ at $v$, and hence when $v$ has been reached in the construction of $T$,
its child $v_1$ has been ordered for processing before the child $v_2$.
Consequently, possibly deeper in the recursion, $u_1$ has been processed before the parent of~$u_2$ and,
in particular, the procedure has added the edge $\{u_1,u_2\}$ into $T$, a contradiction to $u_1$ being to the left of~$u_2$.

This recursive computation is finished in linear time, since every vertex of $G$ is processed only in one branch of the recursion,
and one recursive call takes time linear in the number of incident edges (to~$v$).
\end{proof}

Note that we have not assumed $G$ to be simple or a triangulation in the previous definition and in \Cref{clm:existslal}, 
which will be useful for the cases of bipartite planar and $1$-planar graphs.

\subparagraph{Level assignment in contraction sequences. }
A \emph{partial contraction sequence} of $G$ is defined in the same way as a contraction sequence of $G$, except that it does not have to end with a single-vertex graph.
We are going to work with partial contraction sequences which, preferably, preserve the BFS layers of $\ca L$ of connected~$G$.
However, we do {\em not always} preserve the layers, and so we need a notion which is related to the layers of $\ca L$, 
but it can differ from these layers when needed -- informally, when this ``causes no harm at all''.

\begin{definition}\label{def:goodla}
Let $G$ be a trigraph. A function $\levll G. :V(G)\to\mathbb N$ is called a {\em level assignment}.
Moreover, $\levll G.$ is a {\em good level assignment} if for every $e=\{u,v\}\in E(G)$ (black or red), 
we have $\big|\levll Gu -\levll Gv\big|\leq1$ (cf.~\Cref{cl:layeri3}).
\end{definition}

We are going to deal with the following natural {\em minimum level assignment} of the trigraphs along a partial contraction sequence of a (tri)graph~$G$.
If $G'$ is a trigraph along the sequence of~$G$, and an arbitrary vertex $x\in V(G')$ stems from a set $X\subseteq V(G)$
by (possible) contractions, we denote by $k$ the minimum value such that $\levll{G}{y}=k$ for some~$y\in X$
and by $m$ the maximum value of $\levll{G'}{z}$ over all neighbours $z$ of $x$ in~$G'$.
Observe that $m\geq k-1$ if $\levll G.$ is good.
Then $\levll{G'}.$ is called a {\em minimum level assignment} (with respect to~$G$)
if, for every $x\in V(G')$ and corresponding $k$ and $m$, we have $m\leq\levll{G'}{x}\leq k$~or $k=\levll{G'}{x}=m-1$.

It is practical to observe that every minimum level assignment can be defined recursively along one contraction step from $G'$ to $G''$
(without an explicit reference to the ``original'' graph~$G$).
If $z$ of $G''$ results by the contraction of $x$ and $y$ of $G'$, then $\levll{G''}{z}:=\min(\levll{G'}{x},\levll{G'}{y})$,
and we can also decrease (formally, say, in a dummy step of the contraction sequence)
the level of $x$ by $1$ if all neighbours $t$ of $x$ satisfy $\levll{G'}{t}\leq\levll{G'}{x}-1$.

For {\em not} every partial contraction sequence, a minimum level assignment is possible and/or {good}, 
but we are now going to formulate a condition under which it will be good.
\begin{definition}[Level preserving and respecting]\label{def:lrespecting}
Consider a partial contraction sequence of a trigraph $G$ and a level assignment $\levll..$ of each trigraph along it.
\begin{itemize}
\item
The sequence is \emph{level-preserving} if every step contracts only pairs of the same level
and this level is inherited by the contracted vertex.
\item
The sequence is \emph{level-respecting} if every step contracts, in a trigraph $G'$ along the sequence, 
only a pair $x,y\in V(G')$ such that the following holds;
the levels of $x$ and $y$ are the same, i.e.~$\levll{G'}{x}=\levll{G'}{y}$,  or $\levll{G'}{x}=\levll{G'}{y}-1$
and all neighbours of $y$ (red or black) in $G'$ are on the level $\levll{G'}{x}$.%
\footnote{Unlike in a level-preserving sequence, for a level-respecting sequence we do not prescribe the level of the contracted vertex beforehand.
See the example of a minimum level assignment defined above.}
\item
The sequence is \emph{min-level-respecting} if it is level-respecting, and in every trigraph $G'$ along the sequence, 
the level level assignment $\levll{G'}.$ is a minimum level assignment with respect to~$G$.
\end{itemize}
\end{definition}

A level-preserving partial contraction sequence obviously preserves also the property of being a good level assignment, cf.~\Cref{cl:layeri3},
and can prove the same for min-level-respecting sequences as follows.

\begin{lemma}\label{lem:leveli3}
Consider a \emph{min-level-respecting} partial contraction sequence of a trigraph $G$ with a level assignment $\levll G.$.
If $\levll{G}.$ is good (\Cref{def:goodla}), then, for every trigraph $G'$ along the sequence, 
its level assignment $\levll{G'}.$ (minimum one w.r.t.~$G$ by \Cref{def:lrespecting}) is also good.
\end{lemma}

\begin{proof}
We proceed by induction from $G$ along the sequence, the base case being trivial.
Assume that $\levll{G'}.$ is a good level assignment,
and $z\in V(G'')$ results from a contraction of a pair $x,y\in V(G')$ such that $\levll{G'}{x}\leq\levll{G'}{y}$,
and so $\levll{G''}{z}=\levll{G'}{x}$.
Therefore, there is no neighbour of $z$ on the levels lower than $\levll{G''}{z}-1$ by induction from~$G'$ and~$x$.
On the other hand, if $\levll{G'}{y}=\levll{G'}{x}$, no neighbour of $z$ in $G''$ is on level greater than $\levll{G''}{z}+1$ again by induction from~$G'$.
If $\levll{G'}{y}>\levll{G'}{x}$, then every neighbour of $y$ in $G'$ is by \Cref{def:lrespecting} on the level 
$\levll{G'}{x}$ and every neighbour of $x$ on the level at most $\levll{G'}{x}+1=\levll{G''}{z}+1$ by induction, 
and so the same holds for $z$ in~$G''$.
Lastly, if the level of $x$ of $G'$ is decreased by $1$ (as a dummy step of the sequence, see above) in $G''$, then the assumptions
of a minimum level assignment assert that every neighbour $t$ of $z$ satisfies $\levll{G''}{t}\leq\levll{G'}{x}-1=\levll{G''}{x}$,
again as desired.
\end{proof}

\subparagraph{Skeletal trigraphs. }
Another helpful concept of handling contraction sequences is the following one, which allows us to ``nicely partition''
a large graph into recursively manageable pieces, and take advantage of planarity on the global scale 
(even though the pieces themselves might not be planar, and usually will not be after we start contracting).
Rigorous handling of this intuitive concept requires several uneasy technical definitions which we present next,
and which we will informally illustrate in \Cref{fig:ilskeletal}.

Recall that in $2$-connected planar graphs, every face is bounded by a cycle.
For (tri)graphs $G$ and $S$ such that $V(S)\subseteq V(G)$, we call $X\subseteq G$ a {\em bridge} of $S$ if;
\begin{itemize}
\item $X$ is a connected component of $G\setminus V(S)$ together with all edges from $V(X)$ to $V(S)$, or
\item $X$ is an edge $f\in E(G)\setminus E(S)$ with both ends in $V(S)$ (a {\em trivial bridge}).
\end{itemize}
The $G$-neighbours of $X$ in $V(S)$, or the two ends of an edge $X=f$, are called the {\em attachments} of $X$ in~$S$
(the attachments do not formally belong to~$X$).
We will consider only connected trigraphs $G$, but the related definitions are sound even for disconnected~$G$.

In regard of the coming concepts, we emphasize the following formal fact.
If a partial contraction sequence of $G$ results in a trigraph $G'$, then the vertices of $V(G')\cap V(G)$ are exactly those not participating in any contraction 
of the sequence, while the vertices of $V(G')\setminus V(G)$ are those created by contractions from vertices of $V(G)\setminus V(G')$.

\begin{definition}[Skeletal trigraph and related terms]\label{def:skeletal}
Let $G$ be a trigraph and $S$ a $2$-connected planar graph such that $V(S)\subseteq V(G)$ and all edges of $G$ with both ends in $V(S)$ are black.
Fix a planar embedding of $S$, and call $S$ a {\em (plane) skeleton of~$G$}.
A {\em face assignment} of $G$ in $S$ is a function $\iota$ mapping every bridge $X$ of $S$ to a face $\phi$ of $S$ 
such that all attachments of $X$ in~$S$ belong to the boundary of~$\phi$.
Then, we call the triple $(G,S,\iota)$ a {\em skeletal trigraph}.

We will often denote a skeletal trigraph shortly as $(G,S)$ if $\iota$ is implicit from the context (such as in coming \Cref{def:naturalface}).

For a face $\phi$ of $S$, we denote by $U_\phi$ the union of all vertices of the non-trivial bridges assigned by $\iota$ to~$\phi$
(i.e., without the boundary vertices of~$\phi$), and we also transitively say that the vertices of $U_\phi$ are assigned to~$\phi$.
We furthermore define:
\begin{itemize}
\item
A~face $\phi$ of $S$ is {\em nonempty} if $\phi$ is assigned some nontrivial bridges in~$(G,S)$, i.e.,~$U_\phi\not=\emptyset$.
\item
A face $\phi$ of $S$ is {\em red} if some vertex of $U_\phi$ is incident to a red edge.
\item
A skeletal trigraph $(G,S)$ is {\em proper} if $S\subseteq G$ (which is not required in general).
\end{itemize}
\end{definition}

\begin{definition}[Skeleton-aware contractions]\label{def:skelaware}
Consider a skeletal trigraph $(G,S,\iota)$ as in \Cref{def:skeletal}, and a level assignment~$\levll G.$.
\begin{itemize}
\item
A contraction of a pair $x,y\in V(G)$ in $G$ is \emph{$S$-aware}, or \emph{skeleton-aware} referring to an implicit skeleton~$S$,
if $x,y\not\in V(S)$ belong to bridge(s) assigned to the same face $\phi$ of~$S$.
\item
A partial contraction sequence of $G$ ending in a trigraph $H$ is called \emph{$S$-aware} if $H\supseteq S$
and every its step is $S$-aware.
\item
Considering an $S$-aware contraction in $G$ resulting in a trigraph $G'$,
the bridge of $S$ containing the contracted vertex in $G'$ can be assigned to the same $S$-face as before the contraction by $\iota$ in~$G$,
and we call it the {\em face assignment inherited} from $(G,S,\iota)$.
We extend this definition inductively along $S$-aware partial contraction sequences.
\end{itemize}
\end{definition}

Here, we mostly deal with cases in which the whole graph $G$ is drawn (not necessarily planarly) within the uncrossed plane skeleton $S$, and then we define:
\begin{definition}[Natural face assignment]\label{def:naturalface}
Let $G$ and $S$ be as in \Cref{def:skeletal}, and assume that $G\cup S$ is drawn in the plane such that no edge of $S$ is crossed
(in particular, $S$ is planarly embedded within this drawing).
We call the mapping $\iota$, which assigns every bridge $X$ of $S$ to the face of $S$ in which $X$ is drawn, the {\em natural face assignment} of this drawing.
Then $(G,S,\iota)$, or shortly without ambiguity $(G,S)$, is the {\em natural skeletal trigraph} defined by~$S$~in~$G$.
\end{definition}

The concept of a skeletal trigraph can be smoothly interconnected with BFS trees, and this allows us to define
a specific kind of skeleton faces which will be most useful in the coming proofs.
We typically picture the root of the BFS tree as the topmost vertex of the considered skeleton~$S$.

\begin{definition}[Rooted skeletal trigraph]\label{def:rootskel}
Consider a skeletal trigraph $(G,S,\iota)$ with a fixed plane embedding of $S$, and a vertex $r\in V(S)$ on the outer face of~$S$.
Let $T$, $V(T)\supseteq V(S)$, be a tree with the root~$r$ such that $T\cap S$ is a BFS tree of $S$ rooted at~$r$
(notice that we do not insist on $V(T)\subseteq V(G)$ for technical reasons -- other contractions).
Then we call the tuple $(G,S,\iota,T)$ a \emph{rooted~skeletal~trigraph}, and we again shorten this to $(G,S,T)$ if $\iota$ is implicit.
\end{definition}

\begin{definition}[Wrapped faces, sink and lid]\label{def:wrapped}
Let $(G,S,\iota,T)$ be a {rooted skeletal trigraph}.
A cycle $C\subseteq S$ is \emph{$T$-wrapped} if $C=P_1\cup P_2\cup f$, where
$P_1\subseteq S$ and $P_2\subseteq S$ are edge-disjoint $T$-vertical paths sharing a common vertex~$u$, 
and $f\in E(S)$ connects the other ends of $P_1$ and $P_2$.
The vertex $u\in V(P_1)\cap V(P_2)$ is called the \emph{sink} and the edge $f$ the \emph{lid} of such~$C$.
If $P_1$, $P_2$, $f$ occur in this clockwise cyclic order on~$C$, then we call $P_1$ and $P_2$ 
the \emph{left} and \emph{right wrapping paths}, resp., of~$C$.%
\footnote{Recall that we picture (for clarity) the graph $G$ such that the root of~$T$ is the vertex of $G$ most at the top, and $T$ ``grows down''.}
We naturally extend the notions of $T$-wrapped, sink, lid, and wrapping paths to the faces of $S$ as to their bounding cycles.
\end{definition}

We remark that, for any $T$-wrapped face $\phi$ in \Cref{def:wrapped}, the lid edge $f$ of $\phi$ indeed is unique --- 
it is the unique edge of the boundary of $\phi$ which does not belong to~$T$, 
and the sink $u$ is the unique vertex of the boundary of $\phi$ closest to $r$ in~$S$.

See \Cref{fig:ilskeletal} for a brief visual explanation of the previous definitions.

\begin{figure}[t]
$$\quad
\begin{tikzpicture}[yscale=0.8, xscale=1.3]
\footnotesize
\tikzstyle{every node}=[draw, shape=circle, minimum size=3pt,inner sep=0pt, fill=black]
\node[label=right:\qquad root] at (0,9) (l0) {};
\node at (-1.3,8) (l1) {}; \node at (0,8) (q1) {}; \node at (1.4,8) (r1) {};
\node at (-2,7) (l2) {}; \node at (0.1,7) (q2) {}; \node at (2,7) (r2) {};
\node at (-2.5,6) (l3) {}; \node at (-1.0,6) (p3) {}; \node at (0.5,6) (q3) {}; \node at (2.6,6) (r3) {};
\node at (-2.8,5) (l4) {}; \node at (-0.5,5) (p4) {}; \node at (1.1,5) (q4) {}; \node at (2.9,5) (r4) {};
\node at (-3,4) (l5) {}; \node at (-0.3,4) (p5) {}; \node at (1.5,4) (q5) {}; \node at (3.1,4) (r5) {};
\node at (-3,3) (l6) {}; \node at (-0.2,3) (p6) {}; \node at (1.6,3) (q6) {}; \node at (3.1,3) (r6) {};
\node at (-2.9,2) (l7) {}; \node at (-0.1,2) (p7) {}; \node at (1.4,2) (q7) {}; \node at (3,2) (r7) {};
\node at (-2.7,1) (l8) {}; \node at (0,1) (p8) {}; \node at (2.8,1) (r8) {};
\node at (-2.4,0) (l9) {}; \node at (2.5,0) (r9) {};
\draw (q1)--(l1)--(q2)--(l2);
\draw (p5)--(q5)--(q6)--(p5)--(q4);
\draw (q6)--(p6)--(q7)--(p7) (p8)--(q7)--(r8)--(p8);
\draw (l9)--(r9)--(p8)--(l9) (l8)--(p8);
\draw (r4)--(q3)--(r2);
\begin{scope}[on background layer]
\draw[fill=white!85!black] (0,9)--(0,8)--(0.1,7)--(0.5,6)--(1.1,5)--(1.5,4)--(-0.3,4)--(-0.5,5)--(-1.0,6)--(-2,7)--(-1.3,8);
\draw[fill=white!85!black] (3.1,4)--(3.1,3)--(3,2)--(2.8,1)--(1.4,2)--(1.6,3);
\draw[fill=white!93!black] (0,9)--(0,8)--(0.1,7)--(0.5,6)--(1.1,5)--(1.5,4)--(1.6,3)--(3.1,4)--(2.9,5)--(2.6,6)--(2,7)--(1.4,8);
\draw[fill=white!93!black] (-0.1,2)--(-0.3,4)--(-0.5,5)--(-1.0,6)--(-2,7)--(-2.5,6)--(-2.8,5)--(-3,4)--(-3,3);
\draw[fill=white!85!black] (-3,3)--(-2.9,2)--(-2.7,1)--(0,1)--(-0.1,2);
\end{scope}
\tikzstyle{every path}=[very thick]
\draw (l0)--(l1)--(l2)--(l3)--(l4)--(l5)--(l6)--(l7)--(l8)--(l9);
\draw (l0)--(r1)--(r2)--(r3)--(r4)--(r5)--(r6)--(r7)--(r8)--(r9);
\draw (l2)--(p3)--(p4)--(p5)--(p6) (p7)--(p8);
\draw (l0)--(q1)--(q2)--(q3)--(q4)--(q5);
\draw (r5)--(q6)--(q7) (l6)--(p7);
\draw[dashed] (l9)--(r9)--(p8)--(l9) (l8)--(p8);
\draw[dashed] (p5)--(q5)--(q6)--(p5);
\draw[dashed] (q6)--(p6)--(q7)--(p7) (p8)--(q7)--(r8)--(p8);
\draw[dashed] (p5)--(q6) (p6)--(p7);
\tikzstyle{every path}=[]
\node at (0.7,8) (z1) {}; \node at (1,7) (z2) {}; \node at (1.5,6) (z3) {};
\node at (1.9,5) (z4) {}; \node at (2.3,4) (z5) {};
\node at (-0.2,6) (y1) {}; \node at (0.3,5) (y2) {}; \node at (0.7,4.3) (y3) {};
\node at (2.5,3) (t1) {}; \node at (2.4,2) (t2) {}; \node at (2.7,1.5) (t3) {};
\node at (-2,6) (x1) {}; \node at (-1.5,6) (x2) {};
\node at (-2.3,5) (x3) {}; \node at (-1.7,5) (x4) {}; \node at (-1.1,5) (x5) {};
\node at (-2.4,4) (x6) {}; \node at (-1.7,4) (x7) {}; \node at (-1.0,4) (x8) {};
\node at (-2.2,3) (x9) {}; \node at (-1.4,3) (x10) {};
\node at (-1.3,2) (x11) {}; \node at (-1.2,1.3) (x12) {};
\draw (l0)--(z1) (r5)--(t1) (x7)--(x4)--(x1)--(l2)--(x2) (p3)--(q2);
\draw (l3)--(x1) (l3)--(x4) (x1)--(p4) (l4) to[bend left] (x4);
\tikzstyle{every path}=[color=red]
\draw (z1)--(z2)--(z3)--(z4)--(z5)--(q6);
\draw (q1)--(z1)--(r1)--(z2)--(q1);
\draw (q2)--(r2) (q2)--(z1)--(r2)--(z3)--(q2);
\draw (q3)--(r3) (q3)--(z2)--(r3)--(z4)--(q3);
\draw (q4)--(r4) (q4)--(z3)--(r4)--(z5)--(q4);
\draw (q5)--(r5);
\draw (l2)--(y1)--(y2)--(y3)--(p5);
\draw (p3)--(q3) (q3)--(y2)--(p3)--(y1)--(q2);
\draw (p4)--(q4) (p4)--(y1)--(q4)--(y3)--(p4);
\draw (y3)--(q5)--(y2)--(p5);
\draw (t1)--(t2)--(t3)--(r8) (t2)--(r8);
\draw (q6)--(r6) (q6)--(t2)--(r6);
\draw (q7)--(r7) (q7)--(t1)--(r7)--(t3)--(q7);
\draw (x1)--(p3) (l4)--(p4) (l5)--(p5) (l6)--(p6) (l7)--(p7);
\draw (x2)--(x5)--(x8)--(x10) (x11)--(x12);
\draw (x1)--(x3)--(x6)--(x9);
\draw (l3)--(x3)--(l5) (x1)--(x5)--(p3);
\draw (x4)--(x2)--(p4) (l4)--(x6)--(x4)--(x8)--(p4);
\draw (x3)--(x7)--(x5)--(p5);
\draw (l5)--(x9)--(x7)--(x10)--(p5) (l5)--(x10)--(x6) (x8)--(x9)--(p5);
\draw (x6)--(l6)--(x11) (x7)--(p6) (x8)--(p6);
\draw (x9)--(p7)--(x10) (l7)--(x12)--(p7);
\draw[black] (l8)--(x11)--(p8)--(x12)--(l8);
\end{tikzpicture}
$$
\caption{A fragment of a rooted skeletal trigraph $(G,S,T)$, with the natural face assignment given by the depicted drawing.
	The BFS tree $T\subseteq S$ is drawn with thick edges and the root at the top, and remaining edges of $S$ are thick-dashed.
	A good level assignment associated with $T$ is indicated by vertical positions of vertices.
	Observe the two crucial properties of a skeletal trigraph -- no edge in the drawing crosses edges of the skeleton~$S$,
	and no edge with both ends in $V(S)$ is red.
	The five nonempty faces of $S$ in $(G,S)$ are shaded in gray, and all of them are $T$-wrapped and red.
	The outer cycle of this picture is $T$-wrapped as well.}
\label{fig:ilskeletal}
\end{figure}
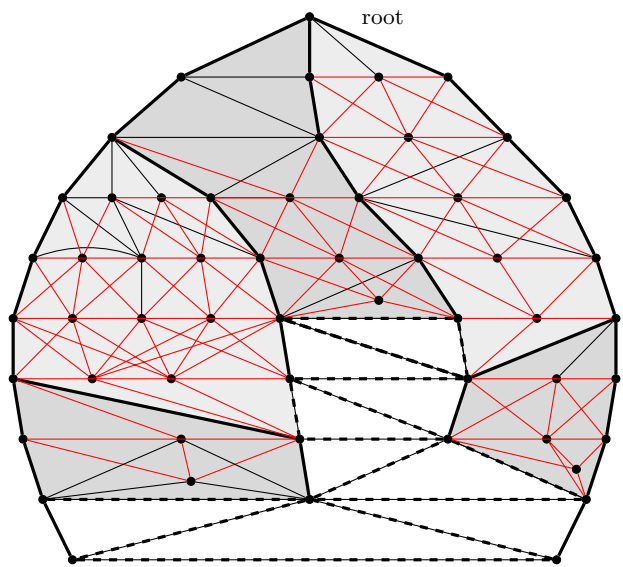

Let us now briefly illustrate the intended use of the concepts of rooted skeletal trigraphs and of wrapped faces.
If we are applying a contraction sequence which ensures that all inherited level assignments are good
(such as a level-preserving or a min-level-respecting one, by \Cref{lem:leveli3}),
then any vertex assigned to a red wrapped face $\phi$ of an intermediate rooted skeletal trigraph may have at most 
three red neighbours in each of the two wrapping paths of $\phi$ by \Cref{def:goodla}.
So, if we can moreover bound (or reduce) the number of vertices assigned to such $\phi$ in each level, we 
immediately get an upper bound on the maximum red degree ``inside~$\phi$''.
However, a problem remains with vertices of the skeleton~$S$ which may potentially be incident to arbitrarily many red faces of~$S$.
To solve this problem, we can observe (see the proof of~\Cref{lem:coremore}) that for $x\in V(S)$, 
all wrapped faces incident to $x$ except at most two have $x$ as their sink, and we will avoid red edges into the sinks altogether.
We will use the following two notion to formulate these ideas.

\begin{definition}[Reduced face]\label{def:reduced}
Let $(G,S,\iota)$ be a {skeletal trigraph} and $\phi$ a face of~$S$.
As in \Cref{def:skeletal}, let $U_\phi$ denote all vertices (except the boundary) assigned by $\iota$ to~$\phi$.
The face $\phi$ is {\em$k$-reduced} with respect to a level assignment $\levll G.$ if $\phi$ is assigned at most $k$ vertices of each level,
that is, $|\{x\in U_\phi:\levll Gx=i\}|\leq k$ for all integer~$i$.
\end{definition}
\begin{definition}[Sink-protecting contraction sequence]\label{def:sinkprot}
Assume a graph $G$ with the level assignment $\levll G.$ defined by the BFS layers of $T$,
and a rooted skeletal trigraph $(G',S,\iota,T)$ where $G'$ is obtained in a partial contraction sequence of $G$.
We say that the sequence is \emph{sink-protecting} if the following holds for every T-wrapped face~$\phi$ of~$S$;
if $x\in V(G')\setminus V(S)$ is assigned to $\phi$ in $(G',S)$ and $x$ is adjacent to the sink $u\in V(S)$ of~$\phi$,
then all vertices $y\in V(G)$ forming $x$ via contractions in the sequence from~$G$ satisfy $\levll Gy=\levll{G'}x$.
\end{definition}
The purpose of \Cref{def:sinkprot} is, exactly, to make the following lemma work:

\begin{lemma}\label{lem:sinknored}
Let $G$ be a plane graph with a good level assignment, and $T$ be a BFS tree of $G$ rooted on the outer face.
Let $S\subseteq G$ be $2$-connected and let a trigraph $H\supseteq S$ result from $G$ by an $S$-aware partial contraction sequence
which is level-preserving, or sink-protecting.
Then, for every $T$-wrapped face $\phi$ of $(H,S,T)$ in the inherited natural face assignment from~$G$,
the sink of $\phi$ has no red edge to vertices assigned to~$\phi$ in~$(H,S)$.
\end{lemma}
\begin{proof}
Let $C\subseteq S$ be the boundary of $\phi$ and $u\in V(C)$ be the sink of $\phi$, 
and assume there is an edge $\{u,x\}$ for some $x$ assigned to~$\phi$.
Let $X\subseteq V(G)$ be the set of vertices which have been contracted along the sequence to make~$x$, 
and $y\in X$ be a representative such that~$\{u,y\}\in E(G)$.
Then $\levll Gy>\levll Gu$ since $u$ is the sink w.r.t.~$T\subseteq G$, but $\levll Gy\leq\levll Gu+1$ since $\{u,y\}$ is an edge.
Consequently, $\levll Gy=\levll Gu+1$, and hence $\levll Gz=\levll Gu+1$ for all $z\in X$ by the assumption on the sequence.

Since the sequence is $S$-aware, all vertices of $z\in X$ are embedded inside $C$ in~$G$, and since $\levll Gz=\levll Gu+1$,
each such $z$ must have a neighbour in $G$ on the level $\levll Gu$, which can only be the sink $u$ itself 
(as the unique such vertex of the least level on $C$).
In other words, $X$ is fully adjacent to $u$ in $G$, and the edge $\{u,x\}$ in $H$ is black.
\end{proof}

\medskip
An upper bound on the red degree from the previous informal sketch will be refined even further using the following 
crucial technical claim interconnecting the concepts of wrapped faces and of a left-aligned BFS tree.

\begin{lemma}\label{lem:leftalign}
Let $G$ be a plane graph and $T$ be a left-aligned BFS tree of $G$ rooted at the outer face of~$G$,
such that the BFS layers of $T$ define the good level assignment $\levll G.$ of~$G$.
Let $S\subseteq G$ be $2$-connected and
let a trigraph $H\supseteq S$ be obtained from $G$ in an $S$-aware min-level-respecting partial contraction sequence,
such that $(H,S,T)$ is a rooted skeletal trigraph with the inherited natural face assignment from~$G$ (\Cref{def:skelaware} and \Cref{def:naturalface}).
Let $\levll H.$ be a minimum level assignment with respect to~$G$.
If $C\subseteq S$ is a $T$-wrapped cycle and $x\in V(C)$ lies on the right wrapping path of~$C$,
then there is \emph{no} neighbour (red or black) $y$ of $x$ in $H$ such that $\levll Hy=\levll Hx-1$ and;
$y\in V(C)$ is not on the right wrapping path of~$C$, or $y\in V(S)$ is embedded inside~$C$ in~$S$ (if $C$ is not a facial cycle of~$S$), 
or $y\not\in V(S)$ is assigned to an $S$-face contained inside~$C$.
\end{lemma}

\begin{proof}
Assume, for a contradiction, that such $y\in V(H)$ exists.
Let $Y\subseteq V(G)$ be the set of vertices which have been contracted along the sequence to make~$y$ (possibly $Y=\{y\}$ for $y\in V(G)\cap V(H)$).
Since $\{x,y\}\in E(H)$ and $x\in V(G)$, there exists $y_0\in Y$ such that $\{x,y_0\}\in E(G)$,
and since the contraction sequence is $S$-aware, $y_0$ is embedded inside the cycle $C$ in~$G$.
The $T$-vertical path from $y_0$ towards the root in $G$ thus has to intersect $C$, and consequently it certifies that $y_0$ is 
to the left of~$x$ with respect to~$T$ (\Cref{def:leftuv}).

Moreover, due to the edge $\{x,y\}\in E(H)$, the definition of a minimum level assignment asserts that 
$\levll G{y_1}=\levll Hy=\levll Gx-1$ holds for some $y_1\in Y$.
If $\levll G{y_0}\geq\levll Gx$, then, at the first moment in our sequence that a vertex which stems from $y_0$ is contracted with
a vertex of level $\levll G{y_1}$ (e.g., with $y_1$ itself), \Cref{def:lrespecting} of a level-respecting sequence is violated by the neighbour~$x$.
Hence, $\levll G{y_0}=\levll Gx-1$ and the edge $\{x,y_0\}\in E(G)$ contradicts the assumption that $T$ is left-aligned in~$G$ (\Cref{def:leftaligned}).
\end{proof}

\section{Proof of \Cref{thm:twwplanar}; the Planar Case}
\label{sec:prooftwwplanar}
%%%%%%%%%%%%%%%%%%%%%%%%%%%%%%%%%%%%%%%%%%%%%%%%%%%%%%%%%%%%%%%%%%%%%%%%%%%%%%%%%

\subsection{Induction setup with a skeletal trigraph}
\label{sub:induplanar}
%%%%%%%%%%%%%%%%%%%%%%%%%%%%%%%%%%

Our proof assumes a simple connected plane triangulation $G$  in a fixed embedding
(as every simple planar graph can be extended into this form by adding new vertices),
and proceeds by induction (in other words recursively) on suitably defined subregions of it.
This approach is formally captured with the following assumptions.

In this section, we will consider {only proper skeletal trigraphs}, and only skeleton-aware partial contraction sequences.
We will always assume skeletal trigraphs with the face assignment inherited from the natural face assignment of~$G$.
We will also work with a fixed {\em left-aligned BFS tree} $T$ of $G$, rooted on the outer face, and with a \emph{minimum level assignment}
in trigraphs contracted from $G$ as derived from the initial BFS layering of $G$ by~$T$, which is always {\em good} by~\Cref{lem:leveli3}.
In particular, if $S\subseteq G$ is $2$-connected and $H\supseteq S$ is a trigraph resulting by $S$-aware contractions from $G$ (avoiding $V(S)$);
\begin{itemize}
\item $(H,S)$ unambiguously denotes the skeletal trigraph with the face assignment inherited from the natural face assignment of~$G$, and
\item $\levll H.$ denotes the implicit minimum level assignment determined by the BFS layers of $G$ by~$T$ and the partial contraction sequence, 
and we simply speak about the \emph{levels in~$H$}.
\end{itemize}

In regard of \Cref{def:reduced}, we say that a face $\phi$ with the boundary cycle $C\subseteq S$ 
is \emph{maximally $k$-reduced} if $\phi$ is $k$-reduced and the following holds:
For $m:=\max_{x\in V(C)}\levll Gx$ the maximum level occurring on the boundary of $\phi$,
no vertex of $U_\phi$ -- the vertices assigned~to~$\phi$, is of level greater than~$m$, except that if $C$ is a triangle,
then level up to $m+1$ can occur~in~$U_\phi$.

Furthermore, for such a proper skeletal trigraph $(H,S)$ and a cycle $C\subseteq S$, 
the {\em subgraph of $H$ bounded by $C$}, denoted shortly by $H_C$, 
is the subgraph of $H$ formed by the union of the facial cycles of all $S$-faces contained in the bounded region of the embedding of $C$
and of all bridges of $S$ assigned to these $S$-faces.
Notice that this definition is largely independent of implicit $S$, except that $S\supseteq C$.
For an illustration, if $C\subseteq G$, then $G_C$ is formed precisely by the vertices and edges of $C$ and the vertices and edges of $G$ 
embedded inside~$C$ (regardless of a particular choice of~$S\supseteq C$).

One more special term we need to introduce just for the coming proof is that of a vh-divided (a shortcut from a ``vertical-horizontal division'') wrapped face:
\begin{definition}[Vh-division]\label{def:vhdivided}
Let $(H,S,T)$ be a rooted proper skeletal trigraph and $D\subseteq S$ a $T$-wrapped cycle with the wrapping paths left $P_1\subseteq D$ and right $P_2\subseteq D$.
Assume a cycle $C\subseteq S\cap H_D$ and two entities $B_1$ and $B_2$ such that, for $i=1,2$, either $B_i=\emptyset$, or 
$B_i\subseteq S\cap H_D$ is a $T$-wrapped facial cycle of $S$ which is not a triangle, and the lid edge of $B_i$ is $T$-horizontal.
We say that $D$ admits a {\em vh-division into $(C,B_1,B_2)$} if the following is moreover true:
\begin{enumerate}[(a)]
\item The symmetric difference of the three cycles $C\Delta B_1\Delta B_2$ equals~$D$.
(Equivalently, the topological disks bounded by $C$,~$B_1$,~$B_2$ are internally disjoint and together cover the disk of~$D$.
In particular, we have $C=D$ if $B_1=B_2=\emptyset$.)
\item If $B_1\not=\emptyset$, then the sink of $B_1$ equals the sink of $D$ and the left and right wrapping paths of $B_1$ are 
subpaths of the respective wrapping paths $P_1$ and $P_2$ of~$D$.
\item If $B_2=\emptyset$, then $C$ is an $S$-face.
\item\label{it:vhdivC2}
If $B_2\not=\emptyset$, then $B_2$ is edge-disjoint from $B_1$ and the left wrapping path of $B_2$ is a subpath of~$P_1$.
Moreover, there is a $T$-vertical path $P_0\subseteq S\cap H_D$ such that
\begin{itemize}
\item the right wrapping path of $B_2$ is a subpath of~$P_0$, or $B_1=\emptyset$ and the right wrapping path of $B_2$ is a subpath of $P_0\cup P_2$, and
\item the sub-skeleton $S\cap H_D$ consists precisely of the union $D\cup C\cup P_0$ and two additional edges $\{v_0,v_1\}$, $\{v_0,v_2\}$,
where $v_0$ is the end of $P_0$ farther from the root and $\{v_1,v_2\}$ is the lid edge of~$D$, and the $S$-face $(v_1,v_2,v_0)$ is empty.
\end{itemize}
\end{enumerate}
\end{definition}

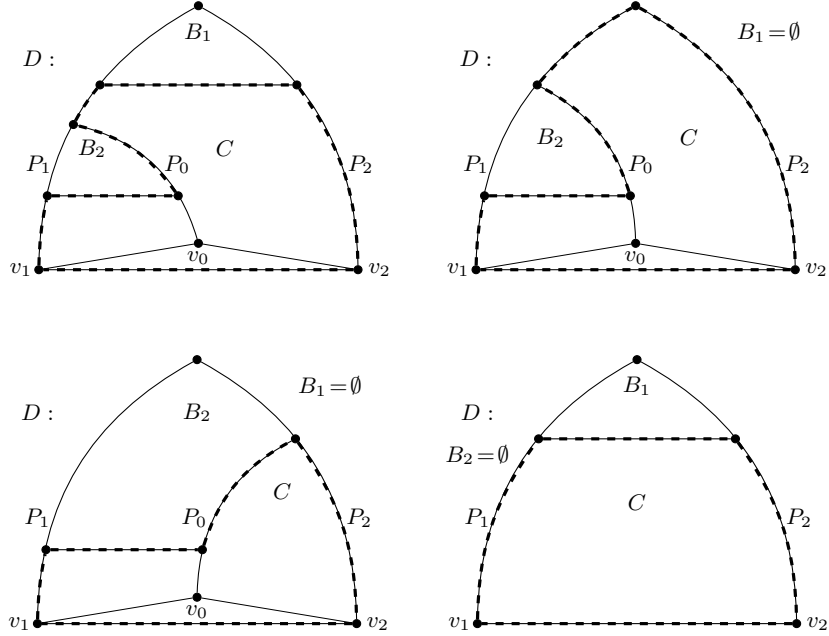
\begin{figure}[tb]
$$
\begin{tikzpicture}[scale=0.7]
\footnotesize
\tikzstyle{every node}=[draw, shape=circle, minimum size=3pt,inner sep=0pt, fill=black]
\node[label=left:$v_1$] at (0,0) (x) {};
\node[label=right:$\,v_2$] at (6,0) (y) {};
\node at (3,5) (z) {};
\draw (x)--(y) to[bend right] (z) (z) to[bend right] (x);
\node at (1.15,3.5) (t) {}; \node at (4.85,3.5) (tt) {};
\draw (t)--(tt);
\node at (0.65,2.75) (u) {}; \node[label=below:$\!v_0\!$] at (3,0.5) (p) {};
\draw (u) to[bend left] (p); \draw[very thin] (x)--(p)--(y);
\node at (0.16,1.4) (w) {}; \node at (2.62,1.4) (ww) {};
\draw (w)--(ww);
\draw[very thick,dashed] (x) to[bend left=5] (w) (w) -- (ww) to[bend right=20] (u)
	(u) to[bend left=5] (t) (t) -- (tt) to[bend left=17] (y) (y)--(x);
\node[draw=none,fill=none] at (0,4) {$D:$};
\node[draw=none,fill=none] at (0,2) {$P_1$};
\node[draw=none,fill=none] at (6.05,2) {$P_2$};
\node[draw=none,fill=none] at (2.6,2) {$P_0$};
\node[draw=none,fill=none] at (3,4.5) {$B_1$};
\node[draw=none,fill=none] at (1,2.3) {$B_2$};
\node[draw=none,fill=none] at (3.5,2.3) {$C$};
\end{tikzpicture}
\qquad
\begin{tikzpicture}[scale=0.7]
\footnotesize
\tikzstyle{every node}=[draw, shape=circle, minimum size=3pt,inner sep=0pt, fill=black]
\node[label=left:$v_1$] at (0,0) (x) {};
\node[label=right:$\,v_2$] at (6,0) (y) {};
\node at (3,5) (z) {};
\draw (x)--(y) to[bend right] (z) (z) to[bend right] (x);
\node at (1.15,3.5) (u) {}; \node[label=below:$\!v_0\!$] at (3,0.5) (p) {};
\draw (u) to[bend left] (p); \draw[very thin] (x)--(p)--(y);
\node at (0.16,1.4) (w) {}; \node at (2.9,1.4) (ww) {};
\draw (w)--(ww);
\draw[very thick,dashed] (x) to[bend left=5] (w) (w) -- (ww) to[bend right=22] (u)
	(u) to[bend left=10] (z) (z) to[bend left=30] (y) (y)--(x);
\node[draw=none,fill=none] at (0,4) {$D:$};
\node[draw=none,fill=none] at (0,2) {$P_1$};
\node[draw=none,fill=none] at (6.05,2) {$P_2$};
\node[draw=none,fill=none] at (3.1,2) {$P_0$};
\node[draw=none,fill=none] at (5.5,4.5) {$B_1\!=\!\emptyset$};
\node[draw=none,fill=none] at (1.4,2.5) {$B_2$};
\node[draw=none,fill=none] at (4,2.5) {$C$};
\end{tikzpicture}
$$\medskip$$
\begin{tikzpicture}[scale=0.7]
\footnotesize
\tikzstyle{every node}=[draw, shape=circle, minimum size=3pt,inner sep=0pt, fill=black]
\node[label=left:$v_1$] at (0,0) (x) {};
\node[label=right:$\,v_2$] at (6,0) (y) {};
\node at (3,5) (z) {};
\draw (x)--(y) to[bend right] (z) (z) to[bend right] (x);
\node at (4.85,3.5) (u) {}; \node[label=below:$\!v_0\!$] at (3,0.5) (p) {};
\draw (u) to[bend right] (p); \draw[very thin] (x)--(p)--(y);
\node at (0.16,1.4) (w) {}; \node at (3.1,1.4) (ww) {};
\draw (w)--(ww);
\draw[very thick,dashed] (x) to[bend left=5] (w) (w) -- (ww) to[bend left=22] (u)
	(u) to[bend left=17] (y) (y)--(x);
\node[draw=none,fill=none] at (0,4) {$D:$};
\node[draw=none,fill=none] at (0,2) {$P_1$};
\node[draw=none,fill=none] at (6.05,2) {$P_2$};
\node[draw=none,fill=none] at (2.95,2) {$P_0$};
\node[draw=none,fill=none] at (5.5,4.5) {$B_1\!=\!\emptyset$};
\node[draw=none,fill=none] at (3,4) {$B_2$};
\node[draw=none,fill=none] at (4.6,2.5) {$C$};
\end{tikzpicture}
\qquad
\begin{tikzpicture}[scale=0.7]
\footnotesize
\tikzstyle{every node}=[draw, shape=circle, minimum size=3pt,inner sep=0pt, fill=black]
\node[label=left:$v_1$] at (0,0) (x) {};
\node[label=right:$\,v_2$] at (6,0) (y) {};
\node at (3,5) (z) {};
\draw (x)--(y) to[bend right] (z) (z) to[bend right] (x);
\node at (1.15,3.5) (t) {}; \node at (4.85,3.5) (tt) {};
\draw (t)--(tt);
\draw[very thick,dashed] (x) to[bend left=17] (t) (t) -- (tt) to[bend left=17] (y) (y)--(x);
\node[draw=none,fill=none] at (0,4) {$D:$};
\node[draw=none,fill=none] at (0,2) {$P_1$};
\node[draw=none,fill=none] at (6.05,2) {$P_2$};
\node[draw=none,fill=none] at (3,4.5) {$B_1$};
\node[draw=none,fill=none] at (0,3.2) {$B_2\!=\!\emptyset$};
\node[draw=none,fill=none] at (3,2.3) {$C$};
\end{tikzpicture}
$$
\caption{A brief schematic illustration of the possible shapes of a cycle $D$ in a rooted skeletal trigraph $(H,S,T)$,
	that admits vh-division into $(C,B_1,B_2)$ as in \Cref{def:vhdivided}.
	$T$ is rooted above the picture, $C$ is thickly dashed, and only the relevant skeleton edges are shown.}
\label{fig:vhdivided}
\end{figure}

See \Cref{fig:vhdivided}.
Informally, the purpose of \Cref{def:vhdivided} is to differently handle specific $T$-horizontal edges in the contraction
sequence constructed in the coming proof of \Cref{lem:core} (and this specific handling is the key ingredient which allows us to reach the bound of~$8$).

\begin{lemma}\label{lem:core}
Let $G$ be a simple plane triangulation and $T$ a left-aligned BFS tree of $G$ rooted at a vertex $r\in V(G)$ of the outer triangular face.
Let $S\subseteq G$, $r\in V(S)$, and let a trigraph $H\supseteq S$ be obtained from $G$ in a min-level-respecting $S$-aware and sink-protecting partial contraction sequence,
such that $(H,S,T)$ is a rooted proper skeletal trigraph (see in order \Cref{def:lrespecting}, \Cref{def:skelaware}, \Cref{def:sinkprot}).
Let $D\subseteq S$ be a $T$-wrapped cycle that admits a vh-division into $(C,B_1,B_2)$ (\Cref{def:vhdivided}) in $(H,S,T)$.

Assume that $H_C=G_C$ (meaning that the subgraph $H_C$ of $H$ bounded by $C$ has not been touched by any contraction since~$G$), 
and let $U:=V(H_C)\setminus V(C)$ denote the vertices of $H$ bounded by~$C$, and $S^*=S-U$ be the sub-skeleton of $S$ having $C$ a facial cycle.
Assume that the $S$-face(s) bounded by $B_i$, $i\in\{1,2\}$ if $B_i\not=\emptyset$, are maximally $1$-reduced.  % (\Cref{def:reducprot}).

Then there exists a min-level-respecting partial contraction sequence of $H$ which contracts only pairs of vertices that 
are in or stem from $U$ (and so is $S^*$-aware), the sequence is sink-protecting, and
results in a trigraph $H^*$ such that the face $\phi$ bounded by $C$ is maximally $1$-reduced in $(H^*,S^*)$.
The sequence can be constructed in linear time $\ca O(|U|)$.
Moreover, the following conditions are satisfied for every trigraph $H'$ along this sequence from $H$ to~$H^*$:
\begin{enumerate}[(i)]
\item\label{it:all8}
Every vertex of $U':=V(H'_C)\setminus V(C)$ has red degree at most $8$ in whole~$H'$.
\item\label{it:bound56}
Every vertex of $C$ has red degree at most $5$ in~$H'_C$, except that 
if $x\in V(C)$ is the vertex of the lid edge of $B_2\not=\emptyset$ to the right (that is, not to the left by \Cref{def:leftuv}), 
then the red degree of $x$ may be up to~$6$ in~$H'_C$.
\item\label{it:right3}
If $x\in V(C)$ lies on the right wrapping path of~$D$, then $x$ has red degree at most~$3$~in~$H'_C$.
\item\label{it:topp2}
If $B_1\not=\emptyset$ and $x\in V(C\cap B_1)$ is any vertex of the lid edge of $B_1$, 
or $B_2\not=\emptyset$ and $x\in V(C\cap B_2)$ is the vertex of the lid edge of $B_2$ to the left,
then $x$ has red degree at most~$2$~in~$H'_C$.
\end{enumerate}
\end{lemma}

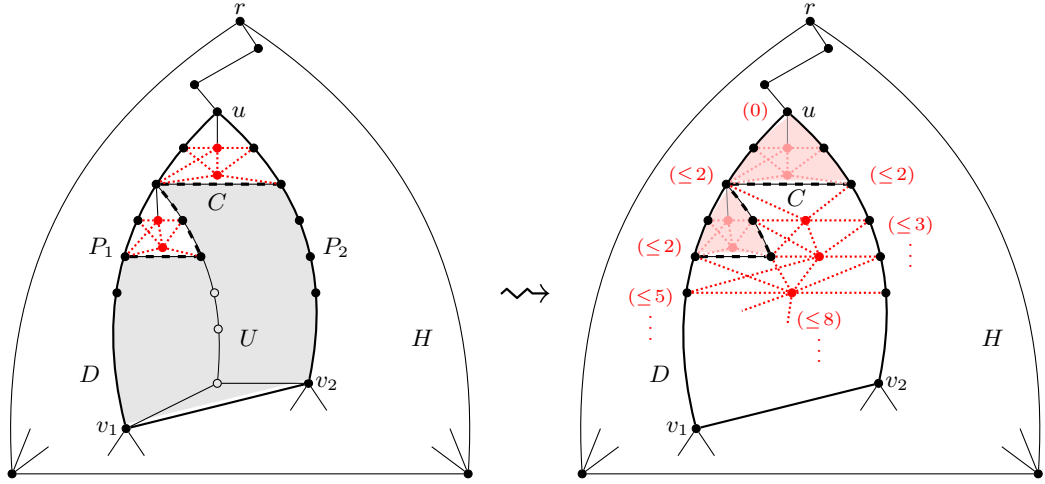
\begin{figure}[tb]
$$
\begin{tikzpicture}[scale=1.2]
\small
\tikzstyle{every node}=[draw, shape=circle, minimum size=3pt,inner sep=0pt, fill=black]
\node at (0.5,0) (x) {};
\node at (5.5,0) (y) {};
\node[label=above:$r$] at (3,5) (z) {};
\draw (x)--(y) to[bend right] (z) (z) to[bend right] (x);
\draw (x)-- +(0.4,0.3); \draw (x)-- +(0.2,0.5);
\draw (y)-- +(-0.4,0.3); \draw (y)-- +(-0.2,0.5);
\node[label=right:$v_2$] at (3.75,1) (v2) {};
\node[label=right:$~{u}$] at (2.75,4) (u1) {};
\node[label=left:$v_1$] at (1.75,0.5) (v1) {};
\draw[draw=none, fill=gray!20] (v2) to[bend right=20] (3.45,3.2)
	--(2.08,3.2) to[bend left=8] (2.57,2.4)--(1.74,2.4) to[bend right=13] (v1) (v1) -- (v2);
\draw[thick] (v2) to[bend right] (u1) (u1) to[bend right] (v1) (v1) -- (v2);
\draw (z)--(3.2,4.7) node{}--(2.5,4.3) node{}--(u1);
\draw (v1)-- +(-0.2,-0.3); \draw (v1)-- +(0.2,-0.3);
\draw (v2)-- +(-0.2,-0.3); \draw (v2)-- +(0.2,-0.3);
\node at (2.38,3.6) (l1) {}; \node at (3.15,3.6) (r1) {};
\node at (2.08,3.2) (l2) {}; \node at (3.45,3.2) (r2) {};
\node at (1.88,2.8) (l3) {}; \node at (3.65,2.8) (r3) {};
\node at (1.74,2.4) (l4) {}; \node at (3.77,2.4) (r4) {};
\node at (1.65,2) (l5) {}; \node at (3.83,2) (r5) {};
\node[fill=none] at (2.75,1) (m0) {}; 
\node at (2.37,2.8) (m3) {}; \node at (2.57,2.4) (m4) {}; 
\draw (l2)--(r2) (l4)--(m4) (v1)--(m0)--(v2) ;
\draw[very thin] (m0) to[bend right=22] (l2);
\node[fill=gray!20] at (2.72,2.0) (m5) {}; \node[fill=gray!20] at (2.76,1.6) (m6) {}; 
\draw[very thick,dashed] (l4)--(m4) to[bend right=8] (l2) (l2)--(r2) ;
\node[red] at (2.75,3.6) (a) {}; \node[red] at (2.75,3.3) (aa) {};
\draw[red, thick,densely dotted] (l1)--(a)--(r1)--(aa)--(l1) (a)--(aa)--(l2)--(a) (aa)--(r2);
\node[red] at (2.1,2.8) (b) {}; \node[red] at (2.15,2.5) (bb) {};
\draw[red, thick,densely dotted] (l3)--(b)--(m3)--(bb)--(l3) (b)--(bb)--(l4)--(b) (bb)--(m4);
\draw (u1)--(a) (l2)--(b);
\node[draw=none,fill=none] at (3.1,1.5) {$U$};
\node[draw=none,fill=none] at (1.35,1.1) {$D$};
\node[draw=none,fill=none] at (2.75,3) {$C$};
\node[draw=none,fill=none] at (1.48,2.5) {$P_1$};
\node[draw=none,fill=none] at (4.05,2.5) {$P_2$};
\node[draw=none,fill=none] at (5,1.5) {$H$};
\end{tikzpicture}
\raise15ex\hbox{\huge\boldmath~$\leadsto$~}
\begin{tikzpicture}[scale=1.2]
\small
\tikzstyle{every node}=[draw, shape=circle, minimum size=3pt,inner sep=0pt, fill=black]
\node at (0.5,0) (x) {};
\node at (5.5,0) (y) {};
\node[label=above:$r$] at (3,5) (z) {};
\draw (x)--(y) to[bend right] (z) (z) to[bend right] (x);
\draw (x)-- +(0.4,0.3); \draw (x)-- +(0.2,0.5);
\draw (y)-- +(-0.4,0.3); \draw (y)-- +(-0.2,0.5);
\node[label=right:$v_2$] at (3.75,1) (v2) {};
\node[label=right:$~{u}$] at (2.75,4) (u1) {};
\node[label=left:$v_1$] at (1.75,0.5) (v1) {};
\draw[draw=none, fill=red!13] (l2)--(r2) to[bend right=9] (u1) to[bend right=9] (l2);
\draw[draw=none, fill=red!13] (l4)--(m4) to[bend right=9] (l2) to[bend right=9] (l4);
\draw[thick] (v2) to[bend right] (u1) (u1) to[bend right] (v1) (v1) -- (v2);
\draw (z)--(3.2,4.7) node{}--(2.5,4.3) node{}--(u1);
\draw (v1)-- +(-0.2,-0.3); \draw (v1)-- +(0.2,-0.3);
\draw (v2)-- +(-0.2,-0.3); \draw (v2)-- +(0.2,-0.3);
\node at (2.38,3.6) (l1) {}; \node at (3.15,3.6) (r1) {};
\node at (2.08,3.2) (l2) {}; \node at (3.45,3.2) (r2) {};
\node at (1.88,2.8) (l3) {}; \node at (3.65,2.8) (r3) {};
\node at (1.74,2.4) (l4) {}; \node at (3.77,2.4) (r4) {};
\node at (1.65,2) (l5) {}; \node at (3.83,2) (r5) {};
\node at (2.37,2.8) (m3) {}; \node at (2.57,2.4) (m4) {}; 
\draw (l2)--(r2) (l4)--(m4) ;
\draw[very thin] (m4) to[bend right=9] (l2);
\draw[very thick,dashed] (l4)--(m4) to[bend right=8] (l2) (l2)--(r2) ;
\node[red!40] at (2.75,3.6) (a) {}; \node[red!40] at (2.75,3.3) (aa) {};
\draw[red!40, thick,densely dotted] (l1)--(a)--(r1)--(aa)--(l1) (a)--(aa)--(l2)--(a) (aa)--(r2);
\node[red!40] at (2.1,2.8) (b) {}; \node[red!40] at (2.15,2.5) (bb) {};
\draw[red!40, thick,densely dotted] (l3)--(b)--(m3)--(bb)--(l3) (b)--(bb)--(l4)--(b) (bb)--(m4);
\draw[gray] (u1)--(a) (l2)--(b);
\node[red] at (2.8,2.0) (a5) {}; \node[red] at (3.1,2.4) (a4) {}; \node[red] at (2.95,2.8) (a3) {};
\draw[red, thick,densely dotted] (m3)--(a3)--(r3)--(a4)--(m3) (l2)--(a3)--(m4)--(a4)--(a3)--(r2);
\draw[red, thick,densely dotted] (r4)--(a4)--(a5)--(l4) (m4)--(a5)--(r5) (l5)--(a5)--(r4) (l5)--(a4);
\draw[red, thick,densely dotted] (a5)-- +(-0.55,-0.2);
\draw[red, thick,densely dotted] (a5)-- +(-0.05,-0.35);
\node[draw=none,fill=none,red] at (3.1,1.7) {\scriptsize$(\leq\!8)$};
\node[draw=none,fill=none,red] at (3.1,1.45) {\scriptsize$\vdots$};
\node[draw=none,fill=none,red] at (1.25,1.95) {\scriptsize$(\leq\!5)$};
\node[draw=none,fill=none,red] at (1.25,1.7) {\scriptsize$\vdots$};
\node[draw=none,fill=none,red] at (4.1,2.75) {\scriptsize$(\leq\!3)$};
\node[draw=none,fill=none,red] at (4.1,2.5) {\scriptsize$\vdots$};
\node[draw=none,fill=none,red] at (3.9,3.3) {\scriptsize$(\leq\!2)$};
\node[draw=none,fill=none,red] at (1.7,3.3) {\scriptsize$(\leq\!2)$};
\node[draw=none,fill=none,red] at (1.35,2.5) {\scriptsize$(\leq\!2)$};
\node[draw=none,fill=none,red] at (2.4,4) {\scriptsize$(0)$};
\node[draw=none,fill=none] at (1.35,1.1) {$D$};
\node[draw=none,fill=none] at (2.85,3.05) {$C$};
\node[draw=none,fill=none] at (5,1.5) {$H$};
\end{tikzpicture}
$$
\caption{(left) A possible setup of \Cref{lem:core} for a trigraph $H$, where the cycle $D$ with its left and right
	wrapping paths $P_1$ and $P_2$ has a vh-division into $(C,B_1,B_2)$ such that~$B_1\not=\emptyset\not=B_2$.
	The interior of $C$ is shown in gray shade.
	(right) The outcome of the claimed partial contraction sequence of $H$ which contracts only vertices of $U$
	inside the shaded region bounded by the cycle~$C$, and which maintains bounded red degrees in the region and on its boundary $C$.
	Note that the bounds indicated in the picture refer to the subgraphs $H'_C$ only.
	Not all red vertices and dotted red edges may exist, and some of the red edges may actually be black.}
\label{fig:corelem}
\end{figure}

We informally illustrate the assumptions and intended outcome of \Cref{lem:bicore} in \Cref{fig:corelem}.
Before proceeding further with the proof, we make some related statements.
First, we show that the assumptions and conclusions stated in \Cref{lem:core} imply that the maximum red degree in $H'$ is at most~$8$
in \Cref{lem:coremore}, and second, that \Cref{lem:core} implies validity of the main \Cref{thm:twwplanar}.

\begin{lemma}\label{lem:corebasic}
Assume the context of \Cref{lem:core}. 
If $C_0\subseteq S$ is a $T$-wrapped cycle bounding a face of $S$ which is $1$-reduced in $(H,S)$,
and $P_1\subseteq C_0$ and $P_2\subseteq C_0$ are the left and right wrapping paths of $C_0$, then the following hold:
\begin{enumerate}[a)]
\item\label{it:sinknored}
The sink of $C_0$ has \emph{no} red edge to vertices of $U_0:=V(H_{C_0})\setminus V(C_0)$.
\item\label{it:left3right2}
Every vertex of $P_1$ has at most $3$ red edges, and every vertex of $P_2$ has at most $2$ red edges, to vertices of $U_0$.
\item\label{it:redin7}
Every vertex of $U_0$ has red degree at most $7$ in~$H$.
\end{enumerate}
\end{lemma}

\begin{proof}\ref{it:sinknored})
This is proved in \Cref{lem:sinknored}.

\ref{it:left3right2})
Any vertex $x\in V(C_0)$, by \Cref{lem:leveli3}, has neighbours only in the same level as $x$ and in the two (previous and next) consecutive levels,
which gives at most $3$ possible red neighbours since the face of $C_0$ is $1$-reduced.
Moreover, if $x\in V(P_2)$, then there is no neighbour of $x$ in level $\levll Hx-1$ due to \Cref{lem:leftalign},
which gives the upper bound of~$2$.

\ref{it:redin7})
Similarly, any vertex $y\in U_0$ can have neighbours only in the same three levels as before, and there are 
three vertices in each level of $H_{C_0}$ -- one in $P_1$, one in $P_2$, and one in $U_0$ since $C_0$ is $1$-reduced.
There are no neighbours of $y$ outside of $H_{C_0}$ since $(H,S)$ is a skeletal trigraph.
And again due to \Cref{lem:leftalign}, there is no neighbour of $y$ in $P_2$ on the level $\levll Hy+1$.
Hence the resulting upper bound is~$2+2+3=7$.
\end{proof}

\begin{lemma}\label{lem:coremore}
Assume the assumptions and conclusions of \Cref{lem:core}. 
If all red faces in~$(H,S^*)$ except the one bounded by $C$ are $1$-reduced and $T$-wrapped, 
then the maximum red degree of each of the considered trigraphs $H'$ is at most~$8$.
\end{lemma}
\begin{proof}
Let $x\in V(H')\setminus V(S^*)$.
If $x\in U'=V(H'_C)\setminus V(C)$, then the claim is given in \Cref{lem:core}\eqref{it:all8}.
Otherwise, $x$ is assigned to an $S^*$-face $\sigma$ in $(H',S^*)$, and $\sigma$ is the same face in $(H,S^*)$.
So, $\sigma$ is $1$-reduced in $(H,S^*)$ and the claim follows from \Cref{lem:corebasic}\eqref{it:redin7}.

Let $y\in V(S^*)\setminus V(H'_D)$.
Since $(H',S^*)$ is a skeletal trigraph, red neighbours of $y$ can come only from incident $S^*$-faces.
All red $S^*$-faces incident to $y$ are $T$-wrapped.
Let $f\in E(T)$ be the edge incident to $y$ on the path towards the root of~$T$. Then $f\in E(S^*)$.
For any $S^*$-face $\sigma$ incident to $y$ but not incident to $f$, the sink of $\sigma$ must be $y$, 
and so this face contributes no red edge to $y$ by \Cref{lem:corebasic}\eqref{it:sinknored}.
Consequently, at most two $S^*$-faces incident to $y$ contribute up to $3$ red edges each by \Cref{lem:corebasic}\eqref{it:left3right2}
(this can be improved to $3$ and $2$), giving that the red degree of $y$ is at most $3+3=6$.

We are left with the case of $z\in C\cup B_1\cup B_2\subseteq V(S^*)$.
If $z\not\in V(C)$, then the same argument as with an upper bound on red degree $6$ for $y$ applies here with $z$, 
since the possible faces of $B_1$ and/or $B_2$ are also $1$-reduced.
If $z\in V(C)$ is not incident to the lid edges of possible cycles $B_1\not=\emptyset$ or $B_2\not=\emptyset$,
then an analogous argument gives an upper bound of the red degree of $z$ at most $3+5=8$, thanks to the bound of \Cref{lem:core}\eqref{it:bound56}.

However, if $z\in V(C)$ is a vertex of one of the lid edges of possible $B_1\not=\emptyset$ or $B_2\not=\emptyset$,
the situation gets different in the sense that up to three incident $S^*$-faces may contribute red edges to~$z$
-- that of $C$, that of the respective $B_1$ or $B_2$, and another $S^*$-face adjacent to $H'_D$ from outside.
In the subcases captured by \Cref{lem:core}\eqref{it:topp2}, we still get an upper bound of $2+3+3=8$ on the red degree of~$z$.
Finally, in the subcase of $z$ being the vertex of the lid edge of $B_2\not=\emptyset$ to the right, as in \Cref{lem:core}\eqref{it:bound56},
we get that $z$ may receive up to $6$ red edges from the face of $C$, and then only up to $2$ red edges from the $1$-reduced
face of $B_2$ by \Cref{lem:corebasic}\eqref{it:left3right2}, again summing to at most~$8$.
\end{proof}

We also show how \Cref{lem:core} implies the main result of this section:

\begin{proof}[Proof of \Cref{thm:twwplanar}]
We start with a given simple planar graph $G_0$, and extend any plane embedding of $G_0$ into a simple plane triangulation $G$
such that $G_0$ is an induced subgraph of~$G$.
In more detail, we add a new vertex into each face of $G_0$ bounded by a cycle and connect it to this cycle, but we suitably add more
such new vertices into a face whose boundary walk repeats some vertices.
Then we choose a root $r$ on the outer face of $G$ and, for some left-aligned BFS tree $T$ of $G$ rooted in $r$ which exists by \Cref{clm:existslal},
the graph $H=G$ and the outer triangle $C=S$ of~$G$, and $B_1=B_2=\emptyset$, we apply \Cref{lem:core}.

This way we get a partial contraction sequence from $G$ to a trigraph $H^*$ of maximum red degree $8$ along the sequence, by \Cref{lem:coremore}.
Observe that $H^*\setminus V(C)$ consists of at most $2$ vertices since $C$ is maximally $1$-reduced now,
and hence in the final phase, we may pairwise contract the remaining vertices in an arbitrary order.
The restriction of this whole contraction sequence of $G$ to only $V(G_0)$ then certifies that the twin-width of $G_0$ is at most~$8$.

Regarding overall runtime of the underlying routine, the construction of~$G$ from $G_0$ can be easily accomplished in linear time using 
the linear-time planarity algorithm \cite{DBLP:journals/jacm/HopcroftT74}.
The tree $T$ is computed in linear time using \Cref{clm:existslal}.
Finally, \Cref{lem:core} constructs the contraction sequence also in linear time $\ca O(|V(G)|=\ca O(|V(G_0)|$.
\end{proof}

\subsection{Finishing the proof}\label{sub:finishing8}
%%%%%%%%%%%%%%%%%%%%%%%%%%%%%%%%%%%%%%%%%%%%%%%%%

The informal high-level idea of the proof of \Cref{lem:core} is as follows.
We are going to ``vertically divide'' the face $\phi$ of $C$ into two parts and independently contract each part by a recursive invocation of \Cref{lem:core}.
Subsequently, we merge the two already contracted parts and, together with the vertical dividing path between them,
we hence obtain an intermediate trigraph in which the interior of $\phi$ contains at most $3$ vertices of each level ($3$-reduced).
The core argument then shows how to contract these interior vertices, level by level, so that we get at most two vertices per level in the first stage,
and finally at most one vertex per level in the second stage (i.e., the face $\phi$ turns $1$-reduced as desired).

We start with proving the easier second stage of this sketch.
\begin{lemma}\label{lem:secondstage}
Respecting the notation and assumptions of \Cref{lem:core}, 
we assume that a min-level-respecting and sink-protecting partial contraction sequence of the trigraph $H$ contracting only within $U$
ends in a trigraph $H^\circ$, such that the face $\phi$ bounded by $C\subseteq S^*$ is $2$-reduced in $(H^\circ,S^*)$.
Moreover, we assume that every step of the sequence up to $H^\circ$ satisfies the conditions \eqref{it:all8}--\eqref{it:topp2} of \Cref{lem:core},
%%% the vertex (if existing) $x\in V(C\cap B_2)$ of condition \eqref{it:topp2} has at most two neighbours in~$V(H^\circ_C)\setminus V(C)$,
and every vertex $x\in V(H^\circ_C)\setminus V(C)$ has at most three red edges in $H^\circ_C$ to vertices of the previous level $\levll{H^\circ}x-1$.
Then the sequence can be prolonged to reach a trigraph $H^*$ such that the conclusion of \Cref{lem:core} is satisfied.
\end{lemma}

\begin{proof}
Let $U^\circ:=V(H^\circ_C)\setminus V(C)$, and let $k$ be the minimum level $\levll{H^\circ}x$ for~$x\in V(C)$
and $m$ be the maximum level $\levll{H^\circ}x$ for~$x\in U^\circ$.
Firstly, for $i=m,m-1,\ldots,k+1$ in this order; if there are two distinct vertices $x\not=x'\in U^\circ$ such that 
$i=\levll{H^\circ}x=\levll{H^\circ}{x'}$, then we contract $x$ with $x'$.
Now $\phi$ is $1$-reduced in the resulting skeletal trigraph $(H^\#,S^*)$.

Let $U^\#:=V(H^\#_C)\setminus V(C)$.
Let, furthermore, $\ell$ be the maximum level $\levll{H^\circ}x$ for~$x\in V(C)$ if $C$ is not a triangle, and $\ell=k+2$ if $C$ is a triangle.
Secondly, for $j=m,m-1,\ldots,\ell+1$ such that there is a vertex $x\in U^\#$ with $j=\levll{H^\#}x$, we perform one of the following.
If there exists $y\in U^\#$ with $\levll{H^\#}y=j-1$, then we contract $x$ with $y$ since all neighbours of $x$ are currently 
on the level $j-1$ and the contraction is allowed by \Cref{def:lrespecting}.
Otherwise, we decrease the level of $x$ to $j-1$, again since all neighbours of $x$ are on the level $j-1$
(notice that there can be neighbour(s) of $x$ in $V(C)$ if $j-1=\ell$) which is correct with respect to a minimum level assignment.
For the resulting trigraph $H^*$, $\phi$ is maximally $1$-reduced in $(H^*,S^*)$,
and the sequence is still sink-protecting thanks to our special choice of~$\ell$.

\smallskip
What remains to be proved is that each of the conditions \eqref{it:all8}--\eqref{it:topp2} of \Cref{lem:core} 
is satisfied in any trigraph $H'$ along the constructed sequence from $H^\circ$~to~$H^*$.
Consider \eqref{it:all8} and a vertex $x\in V(H'_C)\setminus V(C)$.
All neighbours of $x$ now belong to $H'_C$ since the sequence is $S^*$-aware.
If $H'$ is from the first cycle of the previously described sequence from $H^\circ$ to $H^*$, and the iteration is $i\geq\levll{H'}x+2$,
then the claim follows from the assumptions of \Cref{lem:secondstage} (informally, nothing has changed so far in the neighbourhood of~$x$).
With $i\leq\levll{H'}x+1$, in $H'_C$ there are at most three vertices of level $\levll{H'}x+1$,
at most three besides $x$ of level $\levll{H'}x$, and at most three red neighbours of $x$ of level $\levll{H'}x-1$ by the assumption
(or trivially after the iteration~$i=\levll{H'}x-1$).
Moreover, by \Cref{lem:leftalign}, the vertex of $H'_C$ of level $\levll{H'}x+1$ belonging to the right wrapping path of~$C$
is not adjacent to~$x$, and so the total red degree of $x$ is at most~$8$.
For the second cycle of the described sequence, the conclusion of \eqref{it:all8} is already trivial.

Consider \eqref{it:bound56} and a vertex $x\in V(C)$.
Again, if $H'$ is from an iteration $i\geq\levll{H'}x+2$ of the first cycle, then the claim follows from the assumptions.
Otherwise, there is only one vertex of level $\levll{H'}x+1$ left in $V(H'_C)\setminus V(C)$, and twice at most
two vertices on the levels $\levll{H'}x$ and $\levll{H'}x-1$.
Since there are no red edges between vertices of $C$, the upper bound of at most $1+2+2=5$ red neighbours of $x$ in $H'_C$ follows.
The same is trivial in the second cycle of the described sequence.
The condition \eqref{it:right3} follows in the same way, thanks to \Cref{lem:leftalign} which excludes red edges to the level $\levll{H'}x-1$.

Consider \eqref{it:topp2}.
If $x\in V(C)$ is a vertex of the possible lid edge of $B_1\not=\emptyset$, then we can see from the embedding and the BFS tree $T$ of $G$ that
all vertices of $H^\circ_C$ except those of the lid edge are of level (at least, and hence equal to) $\levll{H^\circ}x+1$.
Consequently, $x$ can have at most two neighbours in $V(H^\circ_C)\setminus V(C)$ since $\phi$ is $2$-reduced.
If $x\in V(C)$ is the left vertex of the possible lid edge of $B_2\not=\emptyset$, then we refer to the fine details of \Cref{def:vhdivided}\eqref{it:vhdivC2}.
Using the notation of \Cref{def:vhdivided}, consider the $T$-wrapped cycle $C_0$ formed by $P_1\cup P_0\cup\{v_1,v_0\}$;
then all neighbours of $x$ in the graph $H_C$ must belong to $H_{C_0}$ by planarity and the natural face assignment of~$G$.
So, in particular, all neighbours of $x$ in $V(H_C)\setminus V(C)$ in $H$ are of level $\levll{H}x+1=\levll{H^\circ}x+1$ and the same
holds also in $H^\circ$ by the assumed partial contraction sequence of~$H$.
Again, $x$ can have at most two neighbours in $V(H^\circ_C)\setminus V(C)$ since $\phi$ is $2$-reduced.
This upper bound on the neighbours of $x$ trivially remains true along our sequence from $H^\circ$ to $H^*$, and so \eqref{it:topp2} is proved for~$H'$.
\end{proof}

Then we get to the core proof of \Cref{lem:core} which will conclude our main result.

\begin{proof}[Proof of \Cref{lem:core}]
We proceed by induction on $2|U|+|V(D)|$.
If $U=\emptyset$, we are immediately done with the empty partial contraction sequence. So, we assume $U\not=\emptyset$.

Considering $B_2=\emptyset$, we have $V(C)\subseteq V(D)$ (\Cref{fig:vhdivided}, bottom-right), and $C=D$ holds iff $B_1=\emptyset$.
We denote by $f_0=\{v_1,v_2\}\in E(D)$ the lid edge of $D$.
If there is no edge from $v_1$ to $U$, then for the other neighbour $w$ of $v_1$ on $C$, $(v_1,v_2,w)$ forms a triangular face in $H_C=G_C$.
If $w\not\in V(D)$, then $\{v_1,w\}$ would be the lid edge of $B_1\not=\emptyset$, and so $U=\emptyset$ which was handled above.
Hence we have $w\in V(D)$ and we denote by $D'=(D-v_1)\cup\{w,v_2\}$ the shortened $T$-wrapped cycle 
which admits a vh-division into $(C',B_1,\emptyset)$ where analogously $C'=(C-v_1)\cup\{w,v_2\}$ and $U$ stays the same.
We finish by inductively applying \Cref{lem:core} to~$D'$ in $(H,S\cup\{w,v_2\},T)$.
The symmetric argument applies to~$v_2$.

If neither of the previous is true, and we still have $B_2=\emptyset$,
there exists $v_0\in U$ forming the (unique) bounded triangular face $(v_0,v_1,v_2)$ adjacent to $f_0$ in $H_C=G_C$.
Let $P_0$ denote the $T$-vertical path from $v_0$ to the first vertex $u_1\in V(D)$.
Then $u_1\in V(C)$, due to planarity and the possible lid edge of $B_1\not=\emptyset$.
If $u_1=v_1$ (the case of $u_1=v_2$ is handled analogously), then we let $D''=(D\setminus f_0)\cup\{v_1,v_0\}\cup\{v_0,v_2\}$
with the lid edge $\{v_0,v_2\}$ and $C''=(C\setminus f_0)\cup\{v_1,v_0\}\cup\{v_0,v_2\}$, and appropriately $S''=S\cup\{v_1,v_0\}\cup\{v_0,v_2\}$.
Then $D''$ admits a vh-division into $(C'',B_1,\emptyset)$, and we inductively apply \Cref{lem:core} to~$D''$ in~$(H,S'',T)$.
The recursively obtained partial contraction sequence of $H$ makes the face of $C''$ $1$-reduced, and so $C$ is now $2$-reduced in~$S$
(consider the extra vertex~$v_0$ inside $C$).
We easily finish the sequence by \Cref{lem:secondstage}.

While still having $B_2=\emptyset$, we are left with the possibility that $P_0$ is disjoint from $v_1$ and $v_2$, and then
we form a skeleton $S_0$ from $S$ by adding the path $P_0$ and the edges $\{v_0,v_1\}$, $\{v_0,v_2\}$.
Notice that we are now (with $P_0$ and $v_0$) in a situation very similar to \Cref{def:vhdivided}\eqref{it:vhdivC2}.
On the other hand, if $B_2\not=\emptyset$, we keep $S_0=S$ since $P_0$ is already defined by \Cref{def:vhdivided}\eqref{it:vhdivC2},
and we also define $u_1\in V(D)\cap V(P_0)$ as the first vertex in which $P_0$ intersects~$D$.

\vspace*{-2ex}%
\subparagraph{\textcolor{lipicsGray}{Recursive partial contraction sequence~$\pi_1$.}}
We now enter the core section of our proof, with a path $P_0$ defined above; informally, $P_0$ looks like one of the first three cases of \Cref{fig:vhdivided}
(see also left of \Cref{fig:corelem}).
Let $P_1\subseteq D$ and $P_2\subseteq D$ denote the left and right wrapping paths of~$D$, and $u\in V(P_1)\cap V(P_2)$ be the sink of~$D$.
We first check, in the subgraph $H_D$, the situation ``to the left of the path~$P_0$.''

Let $f_1=\{x_1,x_2\}\in E(H_C)$ be a $T$-horizontal edge such that $x_1\in V(P_1)\cap V(C)$ and $x_2\in V(P_0)\setminus\{u_1\}$,
$\,f_1$ is disjoint from $V(B_2)$ if $B_2\not=\emptyset$, $\,x_2$ is not a neighbour of $u_1$ if $u_1\in V(P_1)$,
and that the distance from $u_1$ to $x_2$ (on $P_0$) is least possible within the stated conditions.
If such $f_1$ does not exist, we choose $f_1:=\{v_1,v_0\}$.%
\footnote{Note that $f_1=\{v_1,v_0\}$ can be chosen as $\{x_1,x_2\}$ also in the previous step if $\{v_1,v_0\}$ is $T$-horizontal.}
Furthermore, let $D_1$ by the unique $T$-wrapped cycle in $P_1\cup P_2\cup P_0\cup f_1$ with the lid~$f_1$.
Then $D_1$ is not a triangle unless possibly when $f_1=\{v_1,v_0\}$, and we have one of the following three exclusive cases:
\begin{itemize}
\item If $B_2\not=\emptyset$, then we choose $C_1=D_1\Delta B_2$ and~$B_1'=B_2$.
\item If $B_2=\emptyset$, and $u_1\in V(P_1)$ or $B_1=\emptyset$, then $D_1$ is a facial cycle of $S_0$ and we choose $C_1=D_1$ and $B_1'=\emptyset$.
\item If $B_2=\emptyset\not=B_1$ and $u_1\in V(P_2)\setminus V(P_1)$, then let $C_1=D_1\Delta B_1$ and~$B_1'=B_1$.
\end{itemize}
It is easy to check that, in any case, $D_1\subseteq H_D$ admits a vh-division into $(C_1,B_1',\emptyset)$ where $C_1\subseteq H_C$, 
and we may inductively apply \Cref{lem:core} to $D_1$ in $(H,S_0,T)$.

Hence, we inductively obtain an $S_0$-aware (where $S_0\supseteq S^*$) partial contraction sequence $\pi_1$ of $H$, 
which results in a trigraph $H^1$ such that the face of $C_1$ is maximally $1$-reduced in the skeletal trigraph $(H^1,S_0,T)$.
It is routine to verify that $\pi_1$ at every step $H'$ satisfies the conditions \eqref{it:all8}--\eqref{it:topp2} of \Cref{lem:core}
also with respect to the skeleton $S^*$ and the cycle~$C$:
Since $H'_{C_1}\subseteq H'_C$, validity of \eqref{it:all8} is trivial, and
the same can be said for \eqref{it:bound56} since the exception with degree $6$ does not apply for~$C_1$.
For \eqref{it:right3} it is enough to observe that the right wrapping path of $D_1$ is a subset of that of~$D$.
Regarding \eqref{it:topp2}, the vertices considered there with respect to $C$ are either not affected by $\pi_1$ (receive no more red edges there),
or they are considered in \eqref{it:topp2} with respect to $C_1$ as well, and so the claimed bound holds.

\vspace*{-2ex}%
\subparagraph{\textcolor{lipicsGray}{Recursive partial contraction sequence~$\pi_2$.}}
For the next inductive invocation of \Cref{lem:core}, in order to handle what happens ``to the right of $P_0$'',
we need a slight technical modification of the skeleton.
Let $S_1=S_0$ if $B_1'=\emptyset$, and $S_1$ be obtained from $S_0$ by removing the lid edge of $B_1'$ otherwise.
Then $D_1$ is a face of $S_1$, and since the face of $B_1'\not=\emptyset$ is maximally $1$-reduced in $(H^1,S_0,T)$ and not a triangle,
and all vertices assigned to the face of $C_1$ are of higher level than those of the horizontal lid edge of $B_1'$, 
we get that the face of $D_1$ is maximally $1$-reduced in $(H^1,S_1,T)$, too.

We further set $B_1'':=\emptyset$ if $u_1\in V(P_2)$ or $B_1=\emptyset$, or $B''_1:=B_1\not=\emptyset$ otherwise.
Now we have one of the following two possibilities in $(H^1,S_1,T)$:
\begin{itemize}
\item If $f_1\not=\{v_1,v_0\}$, we let $D_2:=D$, $C_2:=C\Delta C_1$ and $B_2'':=D_1$.
\item If $f_1=\{v_1,v_0\}$, we let $f_2=\{v_0,v_2\}$ and $D_2$ be the unique $T$-wrapped cycle in $P_1\cup P_2\cup P_0\cup f_2$ with the lid~$f_2$.
Then $C_2:=D_2$ if $B_1''=\emptyset$, and $C_2:=D_2\Delta B_1''$ otherwise.
\end{itemize}
In either case, $D_2\subseteq H^1_D$ clearly admits a vh-division into $(C_2,B_1'',B_2'')$ where $C_2\subseteq H^1_C$, 
and we may inductively apply \Cref{lem:core} to $D_2$ in $(H^1,S_1,T)$.

We inductively obtain a partial contraction sequence $\pi_2$ of $H^1$, which results in a trigraph~$H^2$.
This sequence is $S_1$-aware, and also $S_0$-aware since $\pi_2$ has not touched the subgraph $H^1_{D_1}$ 
to which the possible lid edge of $B_1'$ (removed from $S_0$) belongs to.
The face of $C_2$ is maximally $1$-reduced in $(H^2,S_1,T)$.
The next task is to verify that $\pi_2$ at every step $H'$ satisfies the conditions \eqref{it:all8}--\eqref{it:topp2} of \Cref{lem:core}
with respect to the skeleton $S^*$ and~$C$.

Consider \eqref{it:all8} and a vertex $x\in U'=V(H'_C)\setminus V(C)$.
If $x\not\in V(H'_{C_2})$, then, by the face assignment in $(H',S_1)$, contractions within $\pi_2$ have no effect on~$x$ and
\eqref{it:all8} holds from $(H^1,S_1,T)$ before~$\pi_2$.
If $x\in U'\cap V(C_2)$, then \eqref{it:all8} holds true by \Cref{lem:coremore}.
The remaining possibility is that of $x\in V(H'_{C_2})\setminus V(C_2)$, 
in which \eqref{it:all8} holds true from inductive invocation of the same condition with respect to~$C_2$.

Consider \eqref{it:bound56} and a vertex $x\in V(C)$.   
If $x\not\in V(C_2)$, then, again, \eqref{it:bound56} holds from $(H^1,S_1,T)$ before~$\pi_2$.
Otherwise, we have $x\in V(C)\cap V(C_2)$ and the only red edges between $x$ and $U'$ are from $U'':=V(H'_{C_2})\setminus V(C_2)$
--- then \eqref{it:bound56} follows from inductive invocation of the same with respect to~$C_2$, or we have one of the following cases:
\begin{itemize}
\item If $x$ is the sink of $C_1=D_1$, then there are no red edges from $x$ to $U'\setminus U''$ by \Cref{lem:corebasic}\eqref{it:sinknored}.
\item If $x$ is the right vertex of the lid edge of $B_2\not=\emptyset$, we are in the exceptional case of red degree up to $6$ in \eqref{it:bound56},
and this enables us to count also the possible one red edge -- since the face of $C_1$ is already $1$-reduced, between $x$ and $U'\setminus U''$.
\item If $x$ is the left vertex of the lid $f_1\not=\{v_1,v_0\}$ of $D_1$, then $x$ has at most two red edges to $U''$
by inductive invocation of the condition \eqref{it:topp2} with respect to $C_2$, and at most three red edges to 
$U'\setminus U''$ by \Cref{lem:corebasic}\eqref{it:left3right2}. Together at most~$5$.
\end{itemize}
\vspace*{-1ex}

Consider now \eqref{it:right3} and a vertex $x\in V(C)\cap V(P_2)$.
Then $x\in V(C_1)\setminus V(C_2)$ or $x\in V(C_2)\setminus V(C_1)$, 
and \eqref{it:right3} follows by inductive invocation of the same with respect to $C_1$ or $C_2$, respectively.
Or, $x\in V(C_1)\cap V(C_2)$ is the sink of $C_2$ and one can again use the bound with respect to $C_1$ together with \Cref{lem:corebasic}\eqref{it:sinknored}.
Finally, regarding \eqref{it:topp2}, the vertices considered there with respect to $C$ are either not affected by $\pi_2$,
or they are considered in \eqref{it:topp2} with respect to $C_2$ as well.

\vspace*{-2ex}%
\subparagraph{\textcolor{lipicsGray}{Merging by partial contraction sequence~$\pi_3$.}}
The final task of the proof is to construct a partial contraction sequence $\pi_3$ of $H^2$ which ``merges'' the faces of $C_1$ and $C_2$
into one face of $C$ which will be $2$-reduced, and to which we can subsequently apply \Cref{lem:secondstage}.
We first observe that the set $U_0:=V(H^2_C)\setminus V(C)$ is $3$-reduced in $H^2$;
there is at most one vertex of each level from $U_1:=V(H^2_{C_1})\setminus V(C_1)$,
at most one of each level from $U_2:=V(H^2_{C_2})\setminus V(C_2)$, and at most one of each level from $V(P_0)\cap U_0$.
Let $k$ be the minimum and $m$ the maximum value of $\levll{H^2}x$ over $x\in V(P_0)\cap U_0$.

The sequence $\pi_3$ is as follows.
For $i=m,m-1,\ldots,k+1$ in this order; if there are vertices $y\in V(P_0)$ and $z\in U_2$ such that $i=\levll{H^2}y=\levll{H^2}z$,
then we contract $y$ with~$z$.
For $i=k$ we do the same, unless we have $y\in V(P_0)$ such that $k=\levll{H^2}y=\levll{H^2}{u_1}+1$ and
$y$ is a neighbour of $u_1\in V(P_1)\cap V(P_0)$ in $H^2$ (cf.~the choice of $f_1$ in $H_C$ above).
In the latter case, we contract $y$ with the vertex $t\in U_1$ such that $k=\levll{H^2}y=\levll{H^2}t$, 
or do nothing if such $t$ does not exist.
This finishes $\pi_3$, and in the resulting trigraph $H^3$ the face of $C$ is $2$-reduced.
What remains is to verify that a trigraph $H'$ at every step $i$ of $\pi_3$ satisfies the conditions \eqref{it:all8}--\eqref{it:topp2} 
of \Cref{lem:core} with respect to the skeleton $S^*$ and~$C$.

We start with proving \eqref{it:right3} for $x\in V(C)\cap V(P_2)$.
The claim holds even with the bound of $2$ in $H^2$ by \Cref{lem:corebasic}\eqref{it:left3right2},
and with help of \Cref{lem:leftalign} we see that contractions of $\pi_3$ cannot increase this maximum until, possibly,
$i=k$ and we are contracting $y$ with $t\in U_1$ there.
Then the red degree of (only) $x\in V(C)\cap V(P_2)$ where $\levll{H^3}x\in\{k,k-1\}$ may raise to~$3$ by the contracted vertex,
but that is the end of~$\pi_3$.
Hence \eqref{it:right3} is true along~$\pi_3$.

For the rest, we first show (*) that when $\pi_3$ on level $i$ contracts $y\in V(P_0)$ with $z\in U_2$ into $y'$, then $y'$
is not adjacent to any vertex $v\in U_1\cup V(P_1)$ such that $\levll{H'}v=i-1$.
This follows from two facts; that $z$ is not adjacent to $v$ because of the face assignment in $(H^2,S_1)$,
and $y$ is not adjacent to $v$ by \Cref{lem:leftalign} applied in $(H^1,S_0)$.
From this we easily see that $y'$ in $H'$ has at most three neighbours on the level $i-1$; one in each of the sets $V(P_0)$, $U_2$ and $V(P_2)$.
Note that this holds also in the special case that $i=k$ and $B_2\not=\emptyset$, when one of the possible neighbours of $y'$
is a vertex of the lid edge of $B_2$ lying on~$P_0$.

Consider now the condition \eqref{it:bound56} which holds in $H^2$, too, and a vertex $x\in V(C)\setminus V(P_2)$.
The condition trivially remains true until the step $i=\levll{H^2}x+1$ (contracting in the neighbourhood of~$x$),
from which onward we use the finding (*) of the previous paragraph --- that $x$ cannot be adjacent to the newly contracted vertex
on level $\levll{H^2}x+1$, which leaves an upper bound of $1+2+2=5$ on the red degree of $x$ in $H'_C$ by \Cref{lem:leveli3}.

Regarding validity of the condition \eqref{it:topp2} for a vertex $x\in V(C\cap B_1)$ or $x\in V(C\cap B_2)$ as specified there,
we may simply repeat the arguments presented in the proof of \Cref{lem:secondstage} which hold here in the first place.

\smallskip
The main task is to prove \eqref{it:all8} for every vertex $x\in V(H'_C)\setminus V(C)$.
The condition is true in $H^2$, and so it trivially remains true for steps $i>\levll{H^2}x+1$ (as above).
For the remainder of $\pi_3$, we first consider $x\in U_1$ and~$i\geq k+1$.
By the face assignment in~$(H^2,S_1)$, such $x$ cannot be adjacent to $V(P_2)$ or $U_2$ in $H^2$, and so neither in~$H'$ by the definition of~$\pi_3$.
Hence the red degree of $x$ in $H'$ is upper-bounded (level-by-level) by $3+2+3=8$. % by \Cref{lem:leveli3}.
The same is true if $i=k$ and $x$ is not contracted by~$\pi_3$.
If $x\in U_1$ is contracted in step $i=k$ of $\pi_3$, then, by the definition of $\pi_3$, the at most three possible red neighbours
of the vertex that stems from $x$ on the level $k-1$ are from $V(C)\cup U_2$.

Secondly, we consider \eqref{it:all8} with $x\in \big(V(P_0)\setminus V(C)\big)\cup U_2$, which is relevant only for the step $i=\levll{H'}x+1$.
Then there are only four potential neighbours of $x$ on the level~$i$, and at most two red neighbours of $x$ on each of the levels $i$ and~$i-1$
by the face assignment in $(H^2,S_1)$ (which has not been affected yet by~$\pi_3$). Together at most $4+2+2=8$.

Third, we have $x\in V(H^3)\setminus V(H^2)$, i.e., a vertex contracted by $\pi_3$, and $i\leq\levll{H'}x$.
For simplicity, this case in our proof includes also the case that $x\in V(P_0)\setminus V(C)$ has no mate in $U_2$ 
on the same level $\levll{H'}x$, and so $\pi_3$ skips a contraction in this step.
Assume further that $x$ stems by a contraction of $y\in V(P_0)$ with $z\in U_2$.
\begin{itemize}
\item If there is $y_1\in V(P_1)$ such that $\levll{H^2}y=\levll{H^2}{y_1}$ and $\{y_1,y\}\in E(H^2)$, then $\{y_1,y\}$ has
been chosen as the lid $f_1$ of $D_1$ in the starting section of the proof, and so~$\levll{H'}x=m$.
Since $C_1$ was maximally $1$-reduced, there is no vertex of $U_1$ on the level~$m+1$.
Then $x$ has at most two red neighbours on the level $m+1$ by \Cref{lem:leftalign}, at most three on the level $m$,
and at most three on the level $m-1$ again using \Cref{lem:leftalign} applied to $(H^1,S_0)$.
Together at most $2+3+3=8$.
\item Otherwise, there is no $T$-horizontal edge from $x$ to $V(P_1)$.
Then, similarly as in the previous point, $x$ has at most three red neighbours on the level $m+1$ by \Cref{lem:leftalign}, 
at most two on the level $m$ since the horizontal one to $V(P_1)$ is nonexistent,
and at most three on the level $m-1$ using \Cref{lem:leftalign} applied to $(H^1,S_0)$.
Together at most~$8$.
\end{itemize}
Finally, if $x$ stems by a contraction of $y\in V(P_0)$ with $t\in U_1$, then, by the definition of $\pi_3$,
we have $\levll{H'}x=k$ and $x$ can have only three potential neighbours on the level $k-1$ in $V(C)\cup U_2$.
One of them is $u_1\in V(P_1)\cap V(P_0)$, and the edge $\{x,u_1\}$ is black in $H'$ since $u_1$ is in this case the sink
of $C_1$ and the edge $\{t,u_1\}\in E(H^2)$ was also black.
Then $x$ has at most two red neighbours on the level $k-1$, and at most three in each level $k$ and $k+1$ as shown previously,
together at most~$2+3+3=8$.

\smallskip
Altogether, the concatenated sequence $\pi_1.\pi_2.\pi_3$ (applied in this order) is $S^*$-aware, sink protecting by its definition,
and results in the trigraph $H^3$ in which the face of $C$ is $2$-reduced.
We can finish the whole proof by applying \Cref{lem:secondstage} to $H^\circ=H^3$ since all assumptions of the lemma have been verified above.

On the algorithmic side, we easily identify the $T$-vertical path $P_0$ (if not already given with a vh-division of~$D$) 
and the horizontal edge $f_1$ in linear time $\ca O(|V(P_0)|)\leq\ca O(|U|)$.
We recursively compute the sequence $\pi_1$ in time which is linear in $p_1:=|V(H_{C_1})\setminus V(C_1)|$, by \Cref{lem:core},
and similarly the sequence $\pi_2$ in time linear in $p_2:=|V(H_{C_2})\setminus V(C_2)|$.
Since the recursive sequences $\pi_1$ and $\pi_2$ leave the faces of $C_1$ and of $C_2$ maximally $1$-reduced,
the total number of vertices handled by the merging sequence $\pi_3$ is linear in $|V(P_0)|$
(though, recall that all possible vertices bounded by one of $C_1$ or $C_2$ that are of levels less than~$k$ are skipped by~$\pi_3$),
and so is the number of steps required to construct~$\pi_3$.
For the same reason, the final processing in \Cref{lem:secondstage} also takes time linear in $|V(P_0)|$.
Finally, we have $|V(P_0)|+p_1+p_2\leq|U|$.
\end{proof}

\section{Proof of \Cref{thm:twwbiplanar}; the Bipartite Planar Case}
\label{sec:biplanar}
\label{sec:prooftwwbiplanar}
%%%%%%%%%%%%%%%%%%%%%%%%%%%%%%%%%%%%%%%%%%%%%%%%%%%%%%%%%%%%%%%%%%%%%%%%%%%%%%%%%

On a high level, the proof will proceed similarly as in previous \Cref{sec:prooftwwplanar}.
However, the current bipartite case carries two major differences from the proof of the general planar case:
\begin{itemize}
\item Since our graph is bipartite, we will work with a plane quadrangulation (instead of a triangulation).
This does not bring any significant new challenges to the proof.
\item Since, again, our graph is bipartite, there are no edges within the same BFS layer.
Hence, if we follow a level-preserving partial contraction sequence (\Cref{def:lrespecting}), we will
never create a red edge within the same level.
This is the crucial saving which allows us to derive a better upper bound of~$6$ on the red degree along the constructed sequence.
\end{itemize}

We start with formalizing the latter observation.
\begin{lemma}\label{lem:bipartlevels}
Consider a bipartite graph $G$ and a level assignment $\levll G.$ derived from a BFS layering of~$G$.
If $G'$ is obtained from $G$ in a level-preserving partial contraction sequence of~$G$,
then $G'$ is bipartite, too, and for every edge $\{x,y\}\in E(G')$ (red or black), $\big|\levll{G'}x-\levll{G'}y\big|=1$.
\end{lemma}
\begin{proof}
If $\{x,y\}\in E(G')$, then, by the definition, there are sets $X,Y\subseteq E(G)$ (possibly just $\{x\}$ or $\{y\}$)
such that $x$ stems by (possible) contractions from $X$ and $y$ from~$Y$,
and that $\{s,t\}\in E(G)$ for some $s\in X$ and $t\in Y$.
Then $\levll Gs=\levll{G'}x$ and $\levll Gt=\levll{G'}y$, and since $\levll G.$ is a BFS layering of bipartite $G$,
we have $|\levll Gs-\levll Gt|=1$.
\end{proof}

Aside of these technical details, the coming formulation and a proof of \Cref{lem:bicore} is much simpler than that of \Cref{lem:core} 
since we do not have to deal with complicated \Cref{def:vhdivided} in the induction.

\begin{lemma}\label{lem:bicore}
Let $G$ be a simple plane quadrangulation and $T$ a left-aligned BFS tree of $G$ rooted at a vertex $r\in V(G)$ of the outer quadrangular face.
Let $S\subseteq G$, $r\in V(S)$, and let a trigraph $H\supseteq S$ be obtained from $G$ in a level-preserving $S$-aware partial contraction sequence,
such that $(H,S,T)$ is a rooted proper skeletal trigraph (see in order \Cref{def:lrespecting}, \Cref{def:skelaware}, \Cref{def:sinkprot}).
Let $C\subseteq S$ be a $T$-wrapped facial cycle of~$S$.

Assume that $H_C=G_C$ (the subgraph $H_C$ of $H$ bounded by $C$ has not been touched by any contraction since~$G$), 
and let $U:=V(H_C)\setminus V(C)$ be the vertices of $H$ bounded by~$C$.

Then there exists a level-preserving partial contraction sequence of $H$ which contracts only pairs of vertices that 
are in or stem from $U$ (and so is $S$-aware), and
ends in a trigraph $H^*$ such that the face $\phi$ bounded by $C$ is $1$-reduced in $(H^*,S)$.
The sequence can be constructed in linear time $\ca O(|U|)$.
Moreover, the following conditions are satisfied for every trigraph $H'$ along this sequence from $H$ to~$H^*$:
\begin{enumerate}[(i)]
\item\label{it:all6}
Every vertex of $U':=V(H'_C)\setminus V(C)$ has red degree at most $6$ in whole~$H'$.
\item\label{it:bound4}
Every vertex of $C$ has red degree at most $4$ in~$H'_C$.
\item\label{it:right2}
If $x\in V(C)$ lies on the right wrapping path of~$C$, then $x$ has red degree at most~$2$~in~$H'_C$.
\end{enumerate}
\end{lemma}

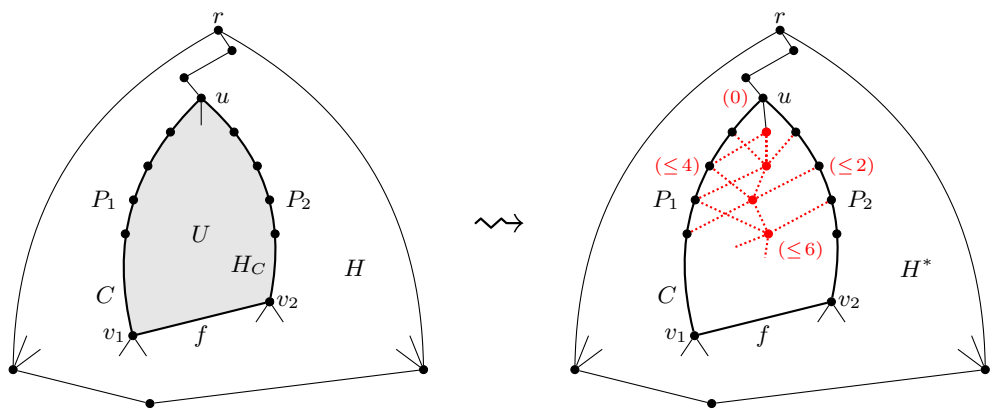
\begin{figure}[tb]
$$
\begin{tikzpicture}[scale=0.9]
\small
\tikzstyle{every node}=[draw, shape=circle, minimum size=3pt,inner sep=0pt, fill=black]
\node at (0,0) (x) {};
\node at (2,-0.5) (xx) {};
\node at (6,0) (y) {};
\node[label=above:$r$] at (3,5) (z) {};
\draw (x)--(xx)--(y) to[bend right] (z) (z) to[bend right] (x);
\draw (x)-- +(0.4,0.3); \draw (x)-- +(0.2,0.5);
\draw (y)-- +(-0.4,0.3); \draw (y)-- +(-0.2,0.5);
\node[label=right:$v_2$] at (3.75,1) (v2) {};
\draw[thick, fill=gray!20] (v2) to[bend right] (2.75,4) to[bend right] (1.75,0.5) -- (v2);
\node[label=right:$~{u}$] at (2.75,4) (u1) {};
\node[label=left:$v_1$] at (1.75,0.5) (v1) {};
\draw (z)--(3.2,4.7) node{}--(2.5,4.3) node{}--(u1);
\draw (v1)-- +(-0.2,-0.3); \draw (v1)-- +(0.2,-0.3);
\draw (v2)-- +(-0.2,-0.3); \draw (v2)-- +(0.2,-0.3);
\draw (u1)-- +(0,-0.4);
\node at (2.3,3.5){}; \node at (3.23,3.5){};
\node at (1.97,3){}; \node at (3.57,3){};
\node[label=left:$P_1~$] at (1.76,2.5){}; \node[label=right:$~P_2$] at (3.75,2.5){};
\node at (1.64,2){}; \node at (3.83,2){};
\node[label=above:$\!\!H_C\quad~$] at (v2) {};
\node[draw=none,fill=none] at (2.75,2) {$U$};
\node[draw=none,fill=none] at (1.35,1.1) {$C$};
\node[draw=none,fill=none] at (2.75,0.5) {$f$};
\node[draw=none,fill=none] at (5,1.5) {$H$};
\end{tikzpicture}
\quad\raise15ex\hbox{\huge\boldmath~$\leadsto$~}\quad
\begin{tikzpicture}[scale=0.9]
\small
\tikzstyle{every node}=[draw, shape=circle, minimum size=3pt,inner sep=0pt, fill=black]
\node at (0,0) (x) {};
\node at (2,-0.5) (xx) {};
\node at (6,0) (y) {};
\node[label=above:$r$] at (3,5) (z) {};
\draw (x)--(xx)--(y) to[bend right] (z) (z) to[bend right] (x);
\draw (x)-- +(0.4,0.3); \draw (x)-- +(0.2,0.5);
\draw (y)-- +(-0.4,0.3); \draw (y)-- +(-0.2,0.5);
\node[label=left:$v_1$] at (1.75,0.5) (v1) {};
\node[label=right:$v_2$] at (3.75,1) (v2) {};
\node[label=right:$~{u}$] at (2.75,4) (u1) {};
\draw[thick] (v2) to[bend right] (u1) (u1) to[bend right] (v1) (v1) -- (v2);
\draw (z)--(3.2,4.7) node{}--(2.5,4.3) node{}--(u1);
\draw (v1)-- +(-0.2,-0.3); \draw (v1)-- +(0.2,-0.3);
\draw (v2)-- +(-0.2,-0.3); \draw (v2)-- +(0.2,-0.3);
\node at (2.3,3.5) (pa){}; \node at (3.23,3.5) (qa){};
\node at (1.97,3) (pb){}; \node at (3.57,3) (qb){};
\node[label=left:$P_1~$] at (1.76,2.5) (pc){}; \node[label=right:$~P_2$] at (3.75,2.5) (qc){};
\node at (1.64,2) (pd){}; \node at (3.83,2) (qd){};
\node[draw=none,fill=none] at (5,1.5) {$H^*$};
\node[draw=none,fill=none] at (1.35,1.1) {$C$};
\node[draw=none,fill=none] at (2.75,0.5) {$f$};
\node[draw=none,fill=none,red] at (3.3,1.75) {\scriptsize$(\leq\!6)$};
\node[draw=none,fill=none,red] at (1.5,3) {\scriptsize$(\leq\!4)$};
\node[draw=none,fill=none,red] at (4.05,3) {\scriptsize$(\leq\!2)$};
\node[draw=none,fill=none,red] at (2.35,4) {\scriptsize$(0)$};
\node[red] at (2.8,3.5) (ra) {}; \node[red] at (2.8,3.5) (rb) {};
\node[red] at (2.8,3) (rc) {}; \node[red] at (2.6,2.5) (rd) {}; \node[red] at (2.83,2) (re) {};
\draw (u1)--(ra);
\draw[red, thick,densely dotted] (re)--(rd)--(rc)--(rb) (rc)--(ra);
\draw[red, thick,densely dotted] (pc)--(rc)--(qa) (pd)--(rd)--(qb) (ra)--(pb) (rc)--(pa) (rd)--(pb) (qc)--(re)--(pc);
\draw[red, thick,densely dotted] (re)-- +(-0.5,-0.2);
\draw[red, thick,densely dotted] (re)-- +(-0.05,-0.35);
\end{tikzpicture}
$$
\caption{(left) The setup of \Cref{lem:bicore} for a trigraph $H$, where $P_1$ and $P_2$ are the left and right wrapping paths 
	of the $T$-wrapped cycle~$C$, and the shaded subgraph $H_C$ has not yet participated in any contraction.
	(right) The outcome of the claimed partial contraction sequence of $H$ which contracts only vertices of $U$ inside the shaded region
	from the left, and which maintains bounded red degrees in the region and on its boundary $C$. The sink $u$ has no red edge there.}
\label{fig:corelembi}
\end{figure}

An illustration of \Cref{lem:bicore} can be seen in \Cref{fig:corelembi}.
We again, before diving into a proof of the lemma, show the sought implications of it.

\begin{lemma}\label{lem:bicoremore}
Assume the assumptions and conclusions of \Cref{lem:bicore}. 
If all red faces in~$(H,S)$ except the one bounded by $C$ are $1$-reduced and $T$-wrapped, 
then the maximum red degree of each of the considered trigraphs $H'$ is at most~$6$.
\end{lemma}
\begin{proof}
Let $x\in V(H')\setminus V(S)$.
If $x\in U'=V(H'_C)\setminus V(C)$, then the claim is given in \Cref{lem:bicore}\eqref{it:all6}.
Otherwise, $x$ is assigned to an $S$-face $\sigma$ in $(H',S)$, and $\sigma$ is the same face in $(H,S)$.
So, $\sigma$ is $T$-wrapped and $1$-reduced in $(H,S)$ and in $(H',S)$ by the assumptions, and $x$ can have at most $3+3=6$ neighbours
in the levels $\levll Hx\pm1$, cf.~\Cref{lem:bipartlevels}.

Let $y\in V(S)$. As in the proof of \Cref{lem:coremore}, there are at most two $S$-faces into which $y$ can have red edges.
At most one of the two faces is bounded by $C$, and then $y$ has red degree at most $4$ in $H'_C$ by \Cref{lem:bicore}\eqref{it:bound4}.
The other incident face(s) of $y$ is $1$-reduced by the assumptions, and so $y$ can have at most $2$ red neighbours in it
in the levels $\levll Hx\pm1$.
Together at most~$4+2=6$.
\end{proof}

\begin{proof}[Proof of \Cref{thm:twwbiplanar}]
We start with a given simple planar graph $G_0$, and extend any plane embedding of $G_0$ into a simple plane quadrangulation $G$
such that $G_0$ is an induced subgraph of~$G$.
Then we choose a root $r$ on the outer face of $G$ and, for some left-aligned BFS tree $T$ of $G$ rooted in $r$ which exists by \Cref{clm:existslal},
the graph $H=G$ and the outer face $C=S$ of~$G$, we apply \Cref{lem:bicore}.

This way we get a level-preserving partial contraction sequence from $G$ to a trigraph $H^*$ of maximum red degree $6$ along the sequence, by \Cref{lem:bicoremore}.
Observe that one face of $C$ is empty and the other one is $1$-reduced.
We can hence finish by successively contracting the vertices of $V(H^*)\setminus V(C)$ from the farthest ones,
and finally contracting the last one of them with $V(C)$ in any order.
The red degree in this final phase never exceeds $2+|V(C)|=6$.

The restriction of this whole contraction sequence of $G$ to only $V(G_0)$ then certifies that the twin-width of $G_0$ is at most~$6$.
The algorithmic part then easily follows in the same way as in the proof of \Cref{thm:twwplanar}.
\end{proof}

The high-level idea of the proof of \Cref{lem:bicore} is the same as presented in \Cref{sub:finishing8};
to ``vertically divide'' the face $\phi$ of $C$ into two parts and independently contract each part by a recursive invocation of \Cref{lem:bicore},
and then to ``merge'' these partial solutions together in two stages (reaching $2$-reduced and then $1$-reduced face of~$C$).

We again start with proving the easier second stage of our proof outline.
\begin{lemma}\label{lem:bisecondstage}
Respecting the notation and assumptions of \Cref{lem:bicore}, 
we assume that a level-respecting partial contraction sequence of the trigraph $H$ contracting only within $U$
ends in a trigraph $H^\circ$, such that the face $\phi$ bounded by $C\subseteq S$ is $2$-reduced in $(H^\circ,S)$.
Moreover, we assume that every step of the sequence up to $H^\circ$ satisfies the conditions \eqref{it:all6}--\eqref{it:right2} 
of \Cref{lem:bicore} and the following additional condition:
every vertex $x\in V(H^\circ_C)\setminus V(C)$ has at most three red neighbours in $H^\circ_C$ in the previous level $\levll{H^\circ}x-1$,
unless all vertices of $C$ are of level at most $\levll{H^\circ}x$.
Then the sequence can be prolonged to reach a trigraph $H^*$ such that the conclusion of \Cref{lem:bicore} is satisfied.
\end{lemma}

\begin{proof}
Let $U^\circ:=V(H^\circ_C)\setminus V(C)$, and let $k$ be the minimum level $\levll{H^\circ}x$ for~$x\in V(C)$
and $m$ be the maximum level $\levll{H^\circ}x$ for~$x\in U^\circ$.
For $i=m,m-1,\ldots,k+1$ in this order; if there are two distinct vertices $x\not=x'\in U^\circ$ such that 
$i=\levll{H^\circ}x=\levll{H^\circ}{x'}$, then we contract $x$ with $x'$.
Now $\phi$ is $1$-reduced in the resulting skeletal trigraph $(H^*,S)$.

It only remains to prove that each of the conditions \eqref{it:all6}--\eqref{it:right2} of \Cref{lem:bicore}
is satisfied in any trigraph $H'$ in step~$i$ along the constructed sequence from $H^\circ$~to~$H^*$.
Consider \eqref{it:all6} and a vertex $x\in V(H'_C)\setminus V(C)$.
If $i>\levll{H'}x$, then an assumption of \Cref{lem:bisecondstage} says that $x$ has at most $3$ red neighbours in the previous level $\levll{H'}x-1$,
no one in the level $\levll{H'}x$ by \Cref{lem:bipartlevels}, and at most $3$ in the next level $\levll{H'}x+1$ thanks to
$\phi$ being $2$-reduced and $T$-wrapped in $(H^\circ,S)$ and \Cref{lem:leftalign}.
Or, we are in the case that no vertex of $C$ is of level greater than $\levll{H'}x$, and then $x$ has at most $4$ red neighbours 
in the previous level $\levll{H'}x-1$ since $\phi$ is $2$-reduced and $T$-wrapped, and at most $2$ in $U^\circ$ in the next level $\levll{H'}x+1$.
If $i\leq\levll{H'}x$, then again, $x$ has at most $4$ red neighbours in the previous level $\levll{H'}x-1$,
and at most $2$ in the next level $\levll{H'}x+1$ since possible two vertices of $U^\circ$ in that level have already been contracted.
Together at most~$6$ in each case.

Condition \eqref{it:bound4} is now trivial for any vertex $x\in V(C)$ since the only red neighbours of $x$ in $H'_C$
must be from $V(H'_C)\setminus V(C)$, and by \Cref{lem:bipartlevels} and $\phi$ being $2$-reduced, we get an upper bound of $2+2=4$
regardless of the step~$i$ of the sequence.
Condition \eqref{it:right2} follows in the same easy way as \eqref{it:bound4}, using additionally \Cref{lem:leftalign} to argue that there
are no such neighbours of $x$ in the level $\levll{H'}x-1$.
\end{proof}

Now we get to the core proof of \Cref{lem:bicore} which will conclude the main result here.

\begin{proof}[Proof of \Cref{lem:bicore}]
We proceed by induction on $2|U|+|V(C)|$.
If $U=\emptyset$, we are immediately done with the empty partial contraction sequence. So, we assume $U\not=\emptyset$.

Let $P_1$ and $P_2$ be the left and right $T$-wrapping paths of the cycle~$C$,\, $u$ be the sink of $C$, and $f=\{v_1,v_2\}$ be the lid edge such that $v_1\in V(P_1)$.
Since $T$ is left-aligned, $\levll H{v_1}=\levll H{v_2}+1$.
Let $A=(v_1,v_2,v_3,v_4)$ be the facial $4$-cycle of $H_C$ adjacent to the edge~$f$ and bounded by~$C$.
Then $\levll H{v_4}=\levll H{v_1}\pm1$, and we first explore the case of $\levll H{v_4}=\levll H{v_1}-1=\levll H{v_2}$,
in which we may have $v_4\in V(P_1)$ or~$v_4\in U$.
Consequently, $\levll H{v_3}=\levll H{v_2}\pm1$, and $\levll H{v_3}=\levll H{v_2}-1$ can happen only if $v_3\in V(P_3)$ since $T$ is left-aligned,
while $v_3\not\in V(P_3)$ when $\levll H{v_3}=\levll H{v_2}+1$.
This straightforwardly leads to four possible subcases which are schematically depicted in \Cref{fig:divisionbi6}\,a),b),c),e),
and the subcase f) in the picture is a variant of c).

On the other hand, for the case of $\levll H{v_4}=\levll H{v_1}+1=\levll H{v_2}+2$ we immediately get that
$\levll H{v_3}=\levll H{v_1}=\levll H{v_2}+1$ and $v_3,v_4\in U$.
This single case will be further divided into two variants now depicted in \Cref{fig:divisionbi6}\,c),f).

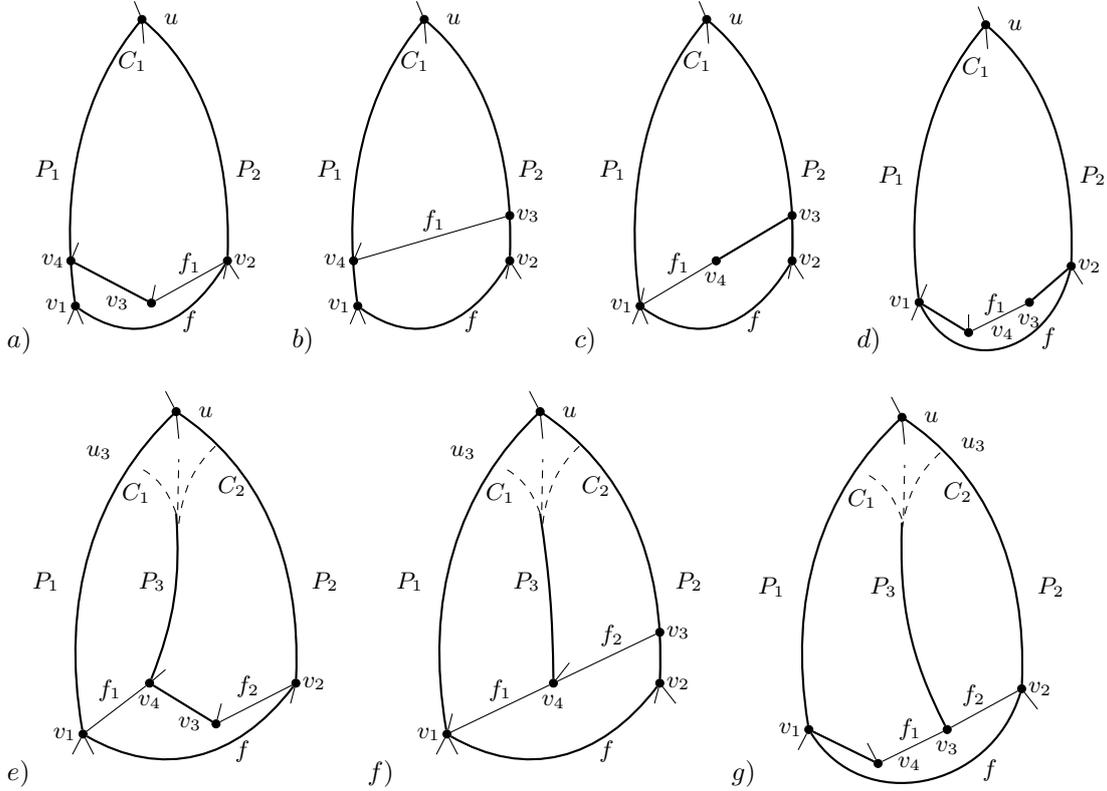
\begin{figure}[tb]
$$ a)
\begin{tikzpicture}[xscale=0.5,yscale=0.8]
\small
\tikzstyle{every node}=[draw, shape=circle, minimum size=3pt,inner sep=0pt, fill=black]
\node[label=left:$v_1$] at (1,0.25) (v1) {};
\node[label=right:$v_2$] at (5,1) (v2) {};
\node[label=right:$~~{u}$] at (2.75,5) (u1) {};
\node[label=left:$v_3~~$] at (3,0.3) (v3) {};
\node[label=left:$~v_4$] at (0.88,1) (v4) {};
\draw[thick] (v2) to[bend right=33] (2.75,5) to[bend right=33] (v1) (v1) to[bend right=35] (v2);
\draw (v2)--(v3)--(v4); \draw[thick] (v4)--(v3);
\draw (u1)-- +(-0.2,0.3) (u1)-- +(0.05,-0.4);
\draw (v1)-- +(-0.2,-0.3); \draw (v1)-- +(0.2,-0.3);
\draw (v2)-- +(-0.1,-0.3); \draw (v2)-- +(0.3,-0.3);
\draw (v3)-- +(0.1,+0.3); \draw (v4)-- +(0.2,+0.3);
\node[draw=none,fill=none, label=left:$P_1$] at (0.85,2.5){};
\node[draw=none,fill=none, label=right:$P_2$] at (5,2.5){};
\node[draw=none,fill=none] at (2.5,4.3) {$C_1$};
\node[draw=none,fill=none] at (4,0) {$f$};
\node[draw=none,fill=none] at (4,0.95) {$f_1$};
\end{tikzpicture}
\quad b)
\begin{tikzpicture}[xscale=0.5,yscale=0.8]
\small
\tikzstyle{every node}=[draw, shape=circle, minimum size=3pt,inner sep=0pt, fill=black]
\node[label=left:$v_1$] at (1,0.25) (v1) {};
\node[label=right:$v_2$] at (5,1) (v2) {};
\node[label=right:$~~{u}$] at (2.75,5) (u1) {};
\node[label=right:$v_3$] at (5,1.75) (v3) {};
\node[label=left:$~v_4$] at (0.89,1) (v4) {};
\draw[thick] (v2) to[bend right=33] (2.75,5) to[bend right=33] (v1) (v1) to[bend right=35] (v2);
\draw (v3)--(v4);
\draw (u1)-- +(-0.2,0.3) (u1)-- +(0.05,-0.4);
\draw (v1)-- +(-0.2,-0.3); \draw (v1)-- +(0.2,-0.3);
\draw (v2)-- +(-0.1,-0.3); \draw (v2)-- +(0.3,-0.3);
\draw (v4)-- +(0.2,+0.3);
\node[draw=none,fill=none, label=left:$P_1$] at (0.85,2.5){};
\node[draw=none,fill=none, label=right:$P_2$] at (5,2.5){};
\node[draw=none,fill=none] at (2.5,4.3) {$C_1$};
\node[draw=none,fill=none] at (4,0) {$f$};
\node[draw=none,fill=none] at (3,1.66) {$f_1$};
\end{tikzpicture}
\quad c)
\begin{tikzpicture}[xscale=0.5,yscale=0.8]
\small
\tikzstyle{every node}=[draw, shape=circle, minimum size=3pt,inner sep=0pt, fill=black]
\node[label=left:$v_1$] at (1,0.25) (v1) {};
\node[label=right:$v_2$] at (5,1) (v2) {};
\node[label=right:$~~{u}$] at (2.75,5) (u1) {};
\node[label=right:$v_3$] at (5,1.75) (v3) {};
\node[label=below:$v_4$] at (3,1) (v4) {};
\draw[thick] (v2) to[bend right=33] (2.75,5) to[bend right=33] (v1) (v1) to[bend right=35] (v2);
\draw (v2)--(v3)--(v4)--(v1); \draw[thick] (v4)--(v3);
\draw (u1)-- +(-0.2,0.3) (u1)-- +(0.05,-0.4);
\draw (v1)-- +(-0.2,-0.3); \draw (v1)-- +(0.2,-0.3);
\draw (v2)-- +(-0.1,-0.3); \draw (v2)-- +(0.3,-0.3);
\draw (v1)-- +(0.1,+0.3);
\node[draw=none,fill=none, label=left:$P_1$] at (0.85,2.5){};
\node[draw=none,fill=none, label=right:$P_2$] at (5,2.5){};
\node[draw=none,fill=none] at (2.5,4.3) {$C_1$};
\node[draw=none,fill=none] at (4,0) {$f$};
\node[draw=none,fill=none] at (2,0.95) {$f_1$};
\end{tikzpicture}
\quad d)
\begin{tikzpicture}[xscale=0.5,yscale=0.8]
\small
\path [use as bounding box] (0,-0.35) rectangle (6,5.5);
\tikzstyle{every node}=[draw, shape=circle, minimum size=3pt,inner sep=0pt, fill=black]
\node[label=left:$v_1$] at (1,0.4) (v1) {};
\node[label=right:$v_2$] at (5,1) (v2) {};
\node[label=right:$~~{u}$] at (2.75,5) (u1) {};
\node[label=below:$v_3$] at (3.9,0.4) (v3) {};
\node[label=right:$~~v_4$] at (2.3,-0.1) (v4) {};
\draw[thick] (v2) to[bend right=33] (2.75,5) to[bend right=33] (v1) (v1) to[bend right=60] (v2);
\draw (v2)--(v3)--(v4)--(v1); \draw[thick] (v4)--(v1);
\draw[thick] (v3)--(v2);
\draw (u1)-- +(-0.2,0.3) (u1)-- +(0.05,-0.4);
\draw (v1)-- +(-0.2,-0.3); \draw (v2)-- +(0.3,-0.3);
\draw (v4)-- +(-0,+0.3); \draw (v1)-- +(0.2,+0.3);
\node[draw=none,fill=none, label=left:$P_1$] at (0.85,2.5){};
\node[draw=none,fill=none, label=right:$P_2$] at (5,2.5){};
\node[draw=none,fill=none] at (2.5,4.3) {$C_1$};
\node[draw=none,fill=none] at (4.4,-0.2) {$f$};
\node[draw=none,fill=none] at (3,0.37) {$f_1$};
\end{tikzpicture}
$$ $$ e)
\begin{tikzpicture}[xscale=0.7,yscale=0.9]
\small
\tikzstyle{every node}=[draw, shape=circle, minimum size=3pt,inner sep=0pt, fill=black]
\node[label=left:$v_1$] at (1,0.25) (v1) {};
\node[label=right:$v_2$] at (5,1) (v2) {};
\node[label=right:$~~{u}$] at (2.75,5) (u1) {};
\node[label=left:$v_3~$] at (3.5,0.4) (v3) {};
\node[label=below:$v_4$] at (2.25,1) (v4) {};
\draw[thick] (v2) to[bend right=33] (2.75,5) to[bend right=33] (v1) (v1) to[bend right=35] (v2);
\draw (v2)--(v3)--(v4)--(v1); \draw[thick] (v4)--(v3);
\draw[thick] (v4) to[bend right=16] (2.75,3.5);
\draw[dashed] (2.75,3.5) to[bend right=22] (2,4.2);
\draw[dashed] (2.8,3.35) to[bend left=22] (3.5,4.5);
\draw[dashed] (2.78,3.55) -- (2.8,4.3);
\draw (u1)-- +(-0.2,0.3) (u1)-- +(0.05,-0.4);
\draw (v1)-- +(-0.2,-0.3); \draw (v1)-- +(0.2,-0.3);
\draw (v2)-- +(-0.1,-0.3); \draw (v4)-- +(0.3,+0.2);
\draw (v3)-- +(0.1,+0.3);
\node[draw=none,fill=none, label=left:$P_1~$] at (0.85,2.5){};
\node[draw=none,fill=none, label=right:$~P_2$] at (5,2.5){};
\node[draw=none,fill=none] at (2.3,2.5) {$P_3$};
\node[draw=none,fill=none] at (2,3.8) {$C_1$};
\node[draw=none,fill=none] at (3.8,3.9) {$C_2$};
\node[draw=none,fill=none] at (4,0) {$f$};
\node[draw=none,fill=none] at (1.3,4.4) {$u_3$};
\node[draw=none,fill=none] at (1.5,0.9) {$f_1$};
\node[draw=none,fill=none] at (4.1,0.95) {$f_2$};
\end{tikzpicture}
\quad f)
\begin{tikzpicture}[xscale=0.7,yscale=0.9]
\small
\tikzstyle{every node}=[draw, shape=circle, minimum size=3pt,inner sep=0pt, fill=black]
\node[label=left:$v_1$] at (1,0.25) (v1) {};
\node[label=right:$v_2$] at (5,1) (v2) {};
\node[label=right:$~~{u}$] at (2.75,5) (u1) {};
\node[label=right:$v_3$] at (5,1.75) (v3) {};
\node[label=below:$v_4$] at (3,1) (v4) {};
\draw[thick] (v2) to[bend right=33] (2.75,5) to[bend right=33] (v1) (v1) to[bend right=35] (v2);
\draw (v2)--(v3)--(v4)--(v1);
\draw[thick] (v4) to[bend right=5] (2.75,3.5);
\draw[dashed] (2.75,3.5) to[bend right=22] (2,4.2);
\draw[dashed] (2.8,3.35) to[bend left=22] (3.5,4.5);
\draw[dashed] (2.78,3.55) -- (2.8,4.3);
\draw (u1)-- +(-0.2,0.3) (u1)-- +(0.05,-0.4);
\draw (v1)-- +(-0.2,-0.3); \draw (v1)-- +(0.2,-0.3);
\draw (v2)-- +(-0.1,-0.3); \draw (v2)-- +(0.3,-0.3);
\draw (v1)-- +(0.1,+0.3); \draw (v4)-- +(0.3,+0.3);
\node[draw=none,fill=none, label=left:$P_1~$] at (0.85,2.5){};
\node[draw=none,fill=none, label=right:$~P_2$] at (5,2.5){};
\node[draw=none,fill=none] at (2.5,2.5) {$P_3$};
\node[draw=none,fill=none] at (2,3.8) {$C_1$};
\node[draw=none,fill=none] at (3.8,3.9) {$C_2$};
\node[draw=none,fill=none] at (4,0) {$f$};
\node[draw=none,fill=none] at (1.3,4.4) {$u_3$};
\node[draw=none,fill=none] at (2,0.9) {$f_1$};
\node[draw=none,fill=none] at (4.1,1.7) {$f_2$};
\end{tikzpicture}
\quad g)
\begin{tikzpicture}[xscale=0.7,yscale=0.9]
\small
\path [use as bounding box] (0,-0.35) rectangle (6,5.5);
\tikzstyle{every node}=[draw, shape=circle, minimum size=3pt,inner sep=0pt, fill=black]
\node[label=left:$v_1$] at (1,0.4) (v1) {};
\node[label=right:$v_2$] at (5,1) (v2) {};
\node[label=right:$~~{u}$] at (2.75,5) (u1) {};
\node[label=below:$v_3$] at (3.6,0.4) (v3) {};
\node[label=right:$~\>v_4$] at (2.3,-0.1) (v4) {};
\draw[thick] (v2) to[bend right=33] (2.75,5) to[bend right=33] (v1) (v1) to[bend right=60] (v2);
\draw (v2)--(v3)--(v4)--(v1); \draw[thick] (v4)--(v1);
\draw[thick] (v3) to[bend left=18] (2.75,3.45);
\draw[dashed] (2.75,3.5) to[bend right=22] (2,4.2);
\draw[dashed] (2.75,3.35) to[bend left=22] (3.5,4.5);
\draw[dashed] (2.78,3.55) -- (2.8,4.3);
\draw (u1)-- +(-0.2,0.3) (u1)-- +(0.05,-0.4);
\draw (v1)-- +(-0.2,-0.3); \draw (v2)-- +(0.3,-0.3);
\draw (v4)-- +(-0.2,+0.3);
\node[draw=none,fill=none, label=left:$P_1~$] at (0.85,2.5){};
\node[draw=none,fill=none, label=right:$~P_2$] at (5,2.5){};
\node[draw=none,fill=none] at (2.4,2.5) {$P_3$};
\node[draw=none,fill=none] at (2,3.8) {$C_1$};
\node[draw=none,fill=none] at (3.8,3.9) {$C_2$};
\node[draw=none,fill=none] at (4.4,-0.2) {$f$};
\node[draw=none,fill=none] at (4.1,4.6) {$u_3$};
\node[draw=none,fill=none] at (2.9,0.4) {$f_1$};
\node[draw=none,fill=none] at (4.1,0.9) {$f_2$};
\end{tikzpicture}
\qquad
$$
\caption{Seven schematic cases of decomposing the drawing of bipartite $H_C$ into subregions,
	as discussed at the beginning of the proof of \Cref{lem:bicore}.
	Vertical positions of the vertices of the $4$-cycle $A=(v_1,v_2,v_3,v_4)$ outline their levels in~$\levll H.$.
	The $T$-vertical path $P_3$, existing in the second row of the picture, is emphasized with thick lines,
	and $P_3$ starts in $v_4$ or $v_3$ and ends in a vertex $u_3\in V(C)$ which can lie on either of $P_1$ or $P_2$, or~$u_3=u$.
	Each of the depicted subcases defines a $T$-wrapped cycle $C_1$ with the lid $f_1$, and the subcases in the second row
	also bring another $T$-wrapped cycle $C_2$ with the lid~$f_2$.}
\label{fig:divisionbi6}
\end{figure}

Secondly, we investigate the maximal $T$-vertical paths starting from the vertices of $\{v_3,v_4\}\cap U$.
In the subcases of \Cref{fig:divisionbi6}\,a),e), the $T$-vertical path of $v_3$ must continue through $v_4$ since $T$ is left-aligned,
and then reach $C$ either directly in $v_4\in V(P_1)$, or in another vertex $u_3\in V(C)$ if $v_4\in U$.
So, let $P_3$ denote the $T$-vertical path from $v_3$ to $u_3$ (which is internally disjoint from $C$) in the latter subcase.
Likewise, let $P_3$ denote the $T$-vertical path from $v_4$ to $u_3\in V(C)$ in the subcase of \Cref{fig:divisionbi6}\,f), that is when $v_3\not\in V(P_3)$.

In the subcases of \Cref{fig:divisionbi6}\,d),g), again, the $T$-vertical path of $v_4$ must continue through $v_1$ since $T$ is left-aligned.
On the other hand, the $T$-vertical path of $v_3$ may contain $v_2$ or not, and in the latter subcase, which is \Cref{fig:divisionbi6}\,g),
we again denote by $P_3$ the $T$-vertical path from $v_3$ to $u_3\in V(C)$ which avoids $v_2$ and is internally disjoint from~$C$.

Altogether, in the four subcases which miss the path $P_3$, i.e., as depicted in \Cref{fig:divisionbi6}\,a),b),c),d),
we are getting a $T$-wrapped cycle $C_1=C\Delta A$ with the lid edge~$f_1$.
Then we inductively apply \Cref{lem:bicore} to $C_1$ in the rooted skeletal trigraph $(H,S\cup A,T)$.
The recursively obtained partial contraction sequence of $H$ makes the face of $C_1$ $1$-reduced, and so the face of $C$ 
is $2$-reduced in $S$ (consider the possible extra vertices $v_3,v_4$ inside~$C$). 
We easily finish the sequence by using \Cref{lem:bisecondstage}.

\vspace*{-2ex}%
\subparagraph{\textcolor{lipicsGray}{Recursive partial contraction sequences~$\pi_1$ and~$\pi_2$.}}
The three subcases with existing $T$-vertical path $P_3$, as in \Cref{fig:divisionbi6}\,e),f),g), are more complex to deal with.
We proceed analogously to the proof of \Cref{lem:core}, albeit in a simpler setting.
As discussed above (and depicted in \Cref{fig:divisionbi6}\,e),f),g)), we have got a $T$-vertical path $P_3$ such that $P_3$ starts on the facial
cycle $A$ and ends in a vertex $u_3$ such that $\{u_3\}\in V(C)\cap V(P_3)$.
Let $S_1=S\cup A\cup P_3$.
Then the sub-skeleton $H_C\cap S_1$ has exactly three bounded facial cycles $A$ and $C_1,C_2$ where; 
$C_1$ is $T$-wrapped with the left wrapping path contained in $P_1$, the right one in $P_3\cup P_2$ and the lid $f_1\in E(A)$,
and $C_2$ is $T$-wrapped with the right wrapping path contained in $P_2$, the left one in $P_3\cup P_1$ and the lid $f_2\in E(A)$.
The cycles $C_1$ and $C_2$ have their sinks as $u$ and $u_3$ in some order, and it may be that $u=u_3$.

We inductively apply \Cref{lem:bicore} to $C_1$ in the skeletal trigraph $(H,S_1,T)$, obtaining a level-preserving
partial contraction sequence $\pi_1$ that ends in a trigraph $H^1$.
Then we analogously apply \Cref{lem:bicore} to $C_2$ in $(H^1,S_1,T)$, obtaining a level-preserving partial contraction sequence $\pi_2$ that ends in~$H^2$.
We claim that the concatenation $\pi_1.\pi_2$ satisfies all conditions of \Cref{lem:bicore} now with respect to the original skeleton~$S$.
Regarding condition~\eqref{it:all6}, this follows from \Cref{lem:bicoremore} applied for the inductive invocations of \Cref{lem:bicore} to $C_1$ and~$C_2$.
As for conditions \eqref{it:bound4} and \eqref{it:right2}, we immediately get their validity for all vertices of $C$ from the inductive
invocations of \Cref{lem:bicore}, except for the sink vertices $u_3$ and $u$ for which we additionally use \Cref{lem:sinknored}.

\vspace*{-2ex}%
\subparagraph{\textcolor{lipicsGray}{Merging by partial contraction sequence~$\pi_3$.}}
The final task of the proof is to construct a partial contraction sequence $\pi_3$ of $H^2$ which ``merges'' the faces of $C_1$ and $C_2$
into one face of $C$ which will be $2$-reduced, and to which we can subsequently apply \Cref{lem:bisecondstage}.
We first observe that the set $U_0:=V(H^2_C)\setminus V(C)$ is $3$-reduced in $H^2$;
there is at most one vertex of each level from $U_1:=V(H^2_{C_1})\setminus V(C_1)$,
at most one of each level from $U_2:=V(H^2_{C_2})\setminus V(C_2)$, and at most one of each level from $U_3:=V(P_3\cup A)\setminus V(C)$.
See \Cref{fig:divisionbi6}, and note that $U_3$ may contain one vertex not on $P_3$, namely $v_4$ in the subcase~g).
Let $k$ be the minimum and $m$ the maximum value of $\levll{H^2}x$ over $x\in U_3$, and $p$ be the maximum of $\levll{H^2}y$~over~$y\in U_2$.

The sequence $\pi_3$ is simply as follows.
For $i=p,p-1,\ldots,k+1$ in this order; if there are distinct vertices $y\in U_1\cup U_3$ and $z\in U_2$ 
such that $i=\levll{H^2}y=\levll{H^2}z$, and $y\in U_3$ if $i\leq m$, then we contract $y$ with~$z$.
Clearly, the sequence is level-preserving, $S$-aware, and in the resulting trigraph $H^3$ the face of $C$ is $2$-reduced.
What remains is to verify that a trigraph $H'$ at every step $i$ of $\pi_3$ satisfies the conditions \eqref{it:all6}--\eqref{it:right2} of \Cref{lem:bicore}.

We start with \eqref{it:bound4} for $x\in V(C)$.
If $i>\levll{H'}x+1$, then validity is inherited from $H^2$.
For the rest, recall that red neighbours of $x$ can be only in levels $\levll{H'}x\pm1$ by \Cref{lem:bipartlevels}.
If $i-1\leq\levll{H'}x\leq i$, then there are at most two candidates for a red edge from $x$ to the next level $\levll{H'}x+1$
(which has already been considered for a contraction), 
but at most one to the previous level $\levll{H'}x-1$ due to $1$-reduced $C_1$ or $C_2$ in~$H^2$.
The latter holds true even if $x=u_3$ thanks to \Cref{lem:sinknored}.
If $i\leq\levll{H'}x-1$, then we simply have at most two candidates for a red edge from $x$ to each of the levels $\levll{H'}x\pm1$.

Regarding \eqref{it:right2} for $x\in V(C)$, we repeat same arguments as for \eqref{it:bound4} with the addition of using 
\Cref{lem:leftalign} to argue that there are no red edges from $x$ to the previous level~$\levll{H'}x-1$.

Condition \eqref{it:all6} for $x\in U'=V(H'_C)\setminus V(C)$ is verified as follows.
First, all red edges of $x$ belong to $H'_C$ due to the skeletal trigraph $(H',S)$.
Second, we call the vertex $x$ {\em special} if $\levll{H'}x=\levll{H^2}{v_4}=\levll{H^2}{v_1}+1$ (the unique case of \Cref{fig:divisionbi6}\,g)).
We claim:
\begin{itemize}
\item If $\levll{H'}x\leq m$ and $x$ is {\em not} special, then $x$ has at most $3$ red neighbours in the previous level $\levll{H'}x-1$.
\end{itemize}
This claims holds true if $i>\levll{H'}x$ from the skeletal trigraph $(H^2,S_1)$, 
and the same applies if $i=\levll{H'}x\leq m$ and $x\in U_1$ (then $x$ is not contracted by $\pi_3$).
If $i<\levll{H'}x\leq m$ and $x\in U_1$, then again, at most $3$ red edges from $x$ into the level $\levll{H'}x-1$ are possible since
$x$ is not adjacent to $V(P_2)$ in~$(H^2,S_1)$, and so neither in~$H'$.
Otherwise, we have that $i\leq\levll{H'}x\leq m$ and $x\not\in U_1$ is not special, which means (by the definition of $\pi_3$), 
that $x$ is the result of a contraction of $y\in V(P_3)$ and $z\in U_2$.
Let $s\in V(P_1)$ and $t\in U_1$ (if existing) be such that $\levll{H'}s=\levll{H'}t=\levll{H'}x-1$.
Then neither of $s,t$ were adjacent to $z$ in the skeletal trigraph $(H^2,S_1)$, 
nor adjacent to $y$ by \Cref{lem:leftalign} applied to the $T$-wrapped cycle $C_1$ of $(H^2,S_1,T)$.
Therefore, $x$ is not adjacent to either of $s,t$, and $x$ can have at most $3$ red neighbours in the level $\levll{H'}x-1$.

From this claim, we already easily derive \eqref{it:all6} for~$x$.
If $\levll{H'}x>m$, or $\levll{H'}x=m$ and $x$ is special, then (cf.~\Cref{fig:divisionbi6})
no vertex of $C$ is of level $\levll{H'}x+1$ and we have at most one red neighbour of $x$ in $V(H'_C)\setminus V(C)$ in the level $\levll{H'}x+1$.
Furthermore, $x$ can have at most $5$ neighbours in the previous level $\levll{H'}x-1$ since $C$ is $3$-reduced in $H^2$.
This gives at most~$6$ together.
Otherwise, $x$ has at most $3$ red neighbours in the level $\levll{H'}x-1$ by the previous claim.
If $i>\levll{H'}x+1$, then $x$ has at most $3$ red neighbours in the level $\levll{H'}x+1$ from the skeletal trigraph $(H^2,S_1)$.
If $i\leq\levll{H'}x+1$, then $x$ cannot have a neighbour on $P_2$ in the level $\levll{H'}x+1$ by \Cref{lem:leftalign},
and so there are at most $3$ neighbours in the level $\levll{H'}x+1$ left.
Again, $x$ has at most $6$ red edges altogether.

\smallskip
Altogether, the concatenated sequence $\pi_1.\pi_2.\pi_3$ (applied in this order) is $S$-aware, level-preserving by its definition,
and results in the trigraph $H^3$ in which the face of $C$ is $2$-reduced.
We can finish the whole proof by applying \Cref{lem:bisecondstage} to $H^\circ=H^3$ since all assumptions of the lemma have been verified above.
On the algorithmic side, we easily finish as in the proof of \Cref{lem:core}.
\end{proof}

\section{Proof of \Cref{thm:tww1planar}; the 1-Planar Case}
\label{sec:prooftww1planar}
%%%%%%%%%%%%%%%%%%%%%%%%%%%%%%%%%%%%%%%%%%%%%%%%%%%%%%%%%%%%%%%%%%%%%%%%%%%%%%%%%

We again start with an overview of the setup for this coming proof.
There are two important conceptual differences.
First, we are going to employ skeletal trigraphs $(G,S)$ which are not necessarily proper ($S$ may contain other edges than those of $G$),
and $S$ may not be simple.
Briefly explaining this issue, we observe that if a $1$-planar drawing of $G$ contains a crossing pair of edges, then these two edges have 
no other crossing and we may ``closely encircle'' them with an uncrossed $4$-cycle on the four end vertices.
This cycle will be used as a part of $S$ even if its edges are not in~$G$ or are parallel to edges of $G$ drawn elsewhere.
Second, we will use a special good level assignment which will be, roughly saying, composed of pairs of consecutive BFS layers of $S$
(in other words, a BFS tree of $S$ does not constitute a good level assignment for~$G$ itself).
This is necessary in this setup since, when encircling a pair of crossing edges of $G$, one of the crossed edges may span across two BFS layers of~$S$.

To intuitively distinguish this new view of a skeleton of $G$, we shall use a notation $S^+$ instead of plain~$S$ here.
The necessary definitions and claims follow.

\begin{definition}\label{def:4framingsk}
A drawing of a graph is good if every pair of crossing edges do so once and have no common end vertex.
A skeletal graph $(G,S^+)$ is {\em $4$-framing} if there is a good drawing of $G\cup S^+$ in the plane such that the subdrawing of $S^+$
is an uncrossed $2$-connected plane graph with only faces of size~$3$ and~$4$, and for every $4$-face $\phi$ of $S^+$ the diagonally opposite
pairs of vertices on the boundary of $\phi$ form a pair of crossing edges drawn inside~$\phi$.
(Observe that all edges of $E(G)\setminus E(S^+)$ then have to be of the latter kind.)
\end{definition}

\begin{lemma}[folklore, see e.g.~\cite{DBLP:conf/isaac/BekosLH022}]\label{lem:4framingsk}
If $G$ is a simple $1$-planar graph, then there exists a $4$-framing skeletal graph $(G,S^+)$ such that $V(S^+)=V(G)$
and no edge of $E(G)\setminus E(S^+)$ is parallel to an edge of~$E(S^+)$.
\end{lemma}

\begin{definition}\label{def:doublelev}
Assume a $4$-framing skeletal graph $(G,S^+)$ such that $V(S^+)=V(G)$, with a drawing as in \Cref{def:4framingsk},
and the BFS layering $\ca L=(L_0=\{r\},L_1,L_2,\ldots)$ of $S^+$ determined by a BFS tree $T\subseteq S^+$ rooted at~$r$ on the outer face.
The {\em double level assignment} of $G$ determined by $T$ is defined as follows;
$\levll Gr:=0$, and $\levll Gx:=i$ if $x\in L_{2i-1}\cup L_{2i}$ for all $i>0$ with the following exception.
If $y\in L_{2j}$ and there exists a $T$-wrapped cycle $C\subseteq S^+$ with the sink $u$, such that $u\in L_{2j-1}$
and $y$ is drawn inside~$C$, then $\levll Gy:=j+1$.

We will moreover say, for $\levll Gx=i$, that $x$ is in the {\em upper half} of level $i$ if $x\in L_{2i}$, and
that $x$ is in the {\em lower half} of level $i$ otherwise.
(The latter terms will not be used for vertices created by level-preserving contractions in~$G$.)
\end{definition}

One may use the example pictures in \Cref{fig:corelem1p} as an illustration of the following.
\begin{lemma}\label{lem:doublegood}\label{lem:neighonpath}
\begin{enumerate}[a)]
\item If $y\in L_{2j}$ is assigned level $\levll Gy=j+1$ in \Cref{def:doublelev}, then $y$ has no neighbour in levels at most~$j-1$.
\item The double level assignment $\levll G.$ of \Cref{def:doublelev} is good both for $G$ and for~$G\cup S^+$.
\item Suppose that $P\subseteq S^+$ is a $T$-vertical path, and that $H$, $V(H)\supseteq V(P)$, 
is a trigraph obtained from $G$ by a level-preserving partial contraction sequence.
If $z\in V(H)\setminus V(P)$, then $z$ has at most $6$ neighbours in $V(P)$ in~$H$, 
and if this number is exactly $6$, then one of the neighbours of $z$ in $V(P)$ is in the upper half of level $\levll Hz+1$.
\end{enumerate}
\end{lemma}
\begin{proof}
a) There is no neighbour of $y$ in $G\cup S^+$ in BFS layers $L_i$ for $i<2j-1$ due to the $T$-wrapped
cycle~$C\subseteq S^+$ which separates $y$ from the root of $T$ and has its sink $u$ in the layer~$L_{2j-1}$.
Hence, no one of level at most~$j-1$.

b) Every edge of $E(S^+)$ is good in $\levll G.$ by the definition, and
by \Cref{def:4framingsk}, every edge of $E(G)\setminus E(S^+)$ has its ends at distance at most $2$ in~$S^+$, which is good again.
The only potential problem is with the exceptional vertices $y\in L_{2j}$ of \Cref{def:doublelev} which get $\levll Gy=j+1$,
and that has been resolved in a).

c) Let $\levll Hz=k$.
It follows from \Cref{def:doublelev} and the argument in paragraph a) that no neighbours of the vertices of $G$ 
contracted into $z$ belong to $L_i$ for $i<2k-3$ or~$i>2k+2$.
Since $P$ is $T$-vertical, it intersects each of the remaining six layers $L_{2k-3},L_{2k-2},\ldots,L_{2k+2}$ of~$V(G)$.
This bounds the sought number of neighbours of $z$ by at most $6$, and the second claim also follows.
\end{proof}

In the rest, we will assume a fixed drawing of $G\cup S^+$ from \Cref{def:4framingsk} and always use, in subsequent skeletal trigraphs
$(H,S)$ which will come from skeleton aware contractions, the face assignment inherited from the natural one of this drawing
(cf.~\Cref{def:skelaware} and~\Cref{def:naturalface}).
Respecting this, we will continue to use the notation $H_C$ for the subgraph of $H$ bounded by a cycle $C\subseteq S^+$, however,
understanding that not all edges of $C$ must belong to~$H_C$ now.
Additionally, when $H_C=G_C$ (meaning that the vertices of $H$ bounded by $C$ have not participated in any contraction since~$G$),
we shall denote by $H^+_C$ the subgraph of $H\cup S^+$ bounded by the cycle $C$, which in the case of $H_C=G_C$ equals the subgraph of $G\cup S^+$ bounded by~$C$.

We will consider only level-preserving contractions, with an exception of the very end of our constructed sequence.
There is one important difference to be noted; unlike in the previous sections, here not every edge of $S$ is an edge of $G$ or $H$,
and this may influence occurrence of red edges in $H$, in particular at the sinks of wrapped~faces.

Another notable difference from the previous proofs is that a $T$-wrapped cycle $C$ in a rooted skeletal trigraph $(H,S,T)$
may now consist only of {\em two parallel edges} --- the left one (at the sink) belonging to~$T$ and the other parallel one forming the lid of~$C$,
cf.~\Cref{lem:core1p}.

Before proceeding to the main proof, we introduce one more specialised notion (for induction).
If $(H,S,T)$ is a rooted skeletal trigraph, and $f\in E(H)\setminus E(S)$ in uncrossed (we have a valid face assignment for $S\cup\{f\}$), 
then we say that a cycle $C\subseteq S\cup\{f\}$ is {\em$T$-wrapped with added lid~$f$} 
if $C$ is $T$-wrapped in $S\cup\{f\}$ and $f$ is its lid there (neglecting that $T$ may not be a BFS tree of $S\cup\{f\}$).
Naturally, such $C$ is called {facial} if it is a facial cycle of $S\cup\{f\}$ 
-- we are going to use this term only when the (plane) drawing of $f$ is given.

The following claim is nearly a direct consequence of \Cref{lem:doublegood}, 
but we have to state it separately due to the exceptional case of $y$ in \Cref{def:doublelev}.

Furthermore, we will use the following adjusted version of \Cref{lem:leftalign}:

\begin{lemma}\label{lem:leftaligndouble}
Let $(G,S^+)$ be a $4$-framing skeletal graph with the double level assignment by a BFS tree $T\subseteq S^+$ as in \Cref{def:doublelev}.
Assume that $T$ is left-aligned and rooted on the outer face of~$S^+$, and that a cycle $C\subseteq S^+\cup G$ is $T$-wrapped, possibly with added lid.
Suppose that $H$, $V(H)\supseteq V(C)$, is a trigraph obtained from $G$ by a level-preserving $C$-aware partial contraction sequence,
and that $x\in V(C)$ and $y\in V(H_C)$ are such that $\levll Hy=\levll Hx-1$ and $x$ lies on the right wrapping path of~$C$.
Then there is no edge (black or red) in $H$ between $x$ and $y$ if at least one of the following conditions is satisfied;
$x$ is in the upper half of level $i$, or $y\in V(C)$ lies on the left wrapping path of~$C$ in the lower half of level~$i-1$.
\end{lemma}
\begin{proof}
Let $x\in L_k$ within the BFS layering of~$S^+$.
First, one can easily check that under the stated conditions on $x$ and $y$, in either case, we get that all vertices
of $G$ that have got contracted into $y$ belong to $L_j$ for some $j\leq k-2$ since our sequence is level-preserving.
If $xy\in E(H)$, then there has to be an edge $xy'\in E(G)$ where $y'\in V(G)$ is among the vertices contracted into~$y$,
and $y'$ is assigned inside $C$ in the skeletal graph $(G,C)$.
Consequently, the distance from $x$ to $y'$ in $S^+$ is $2$ and $xy'\in E(G)\setminus E(S^+)$ joins diagonally opposite vertices
of some $4$-cycle $D\subseteq S^+$ by \Cref{def:4framingsk}.
An edge of $E(D)\setminus E(T)$ incident to $x$ then contradicts the property of $T$ being left-aligned in~$S^+$ (\Cref{def:leftaligned}).
\end{proof}

A $T$-wrapped facial cycle $C\subseteq S$, or $C\subseteq S\cup\{f\}$ in the case of added lid~$f$, in a skeletal trigraph $(H,S,T)$ is {\em almost $1$-reduced}
if it is $1$-reduced except that there can be two vertices in $V(H_C)\setminus V(C)$ of the level $\levll Hu+1$ where $u$ is the sink of~$C$.
Notice also that \Cref{def:doublelev} ensures no vertices of $V(H_C)\setminus V(C)$ are on the level $\levll Hu$.

The main technical lemma follows now.

\begin{lemma}\label{lem:core1p}
Let $(G,S^+)$, $V(S^+)=V(G)$, be a $4$-framing skeletal graph with the double level assignment determined by 
the BFS tree $T\subseteq S^+$ as in \Cref{def:doublelev}.
Assume that $T$ is left-aligned and rooted in $r$ on the outer face of~$S^+$.
Let $S\subseteq S^+$, and let a trigraph $H$ such that $V(H)\supseteq V(S)$ be obtained from $G$ in a level-preserving 
$S$-aware partial contraction sequence, such that $(H,S,T)$ is a rooted skeletal trigraph.
Let $C\subseteq S$ be a $T$-wrapped, or $T$-wrapped with added lid $f\in E(G)\setminus E(S)$, facial cycle of~$S$ in $(H,S,T)$.

Assume that $H_C=G_C$ (and so $H^+_C=(G\cup S^+)_C$ is as in the starting skeletal graph),
and let $U:=V(H_C)\setminus V(C)$ be the vertices of $H$ bounded by~$C$.

Then there exists a level-preserving partial contraction sequence of $H$ which contracts only pairs of vertices that 
are in or stem from $U$ (and so is $S$-aware), and ends in a trigraph $H^*$ such that the face $\phi$ bounded by $C$ is almost $1$-reduced in $(H^*,S)$.
Moreover, the following conditions are satisfied for every trigraph $H'$ along this sequence from $H$ to~$H^*$:
\begin{enumerate}[(i)]
\item\label{it:all16}
Every vertex of $U':=V(H'_C)\setminus V(C)$ has red degree at most $16$ in whole~$H'$.
\item\label{it:added15}
If $t\in U$ is such that $t$ together with the lid $f$ of $C$ form a facial triangle in $H^+_C$, and $t'\in U'$ equals or stems from $t$ in $H'$,
then $t'$ has red degree at most $15$ in~$H'$.
\item\label{it:bound6}
Every vertex of $C$ has red degree at most $6$ in~$H'_C$ and the sink of $C$ has no red edge.
\end{enumerate}
\end{lemma}

\begin{figure}[tb]
$$
\begin{tikzpicture}[scale=1]
\small
\tikzstyle{every node}=[draw, shape=circle, minimum size=3pt,inner sep=0pt, fill=black]
\node at (0,0) (x) {};
\node at (2,-0.5) (xx) {};
\node at (6,0) (y) {};
\node[label=above:$r$] at (3,5) (z) {};
\draw (x)--(xx)--(y) to[bend right] (z) (z) to[bend right] (x);
\draw (x)-- +(0.4,0.3); \draw (x)-- +(0.2,0.5);
\draw (y)-- +(-0.4,0.3); \draw (y)-- +(-0.2,0.5);
\draw[dotted] (0,4.7)--(6,4.7) (0,3.8)--(6,3.8) (0,2.8)--(6,2.8) (0,1.8)--(6,1.8);
\node[label=left:$v_1$] at (1.75,0.5) (v1) {};
\node[label=right:$v_2$] at (3.75,1) (v2) {};
\node[label=right:$~{u}$] at (2.75,4) (u1) {};
\draw[thick] (v2) to[bend right] (u1) (u1) to[bend right] (v1) (v1) -- (v2);
\draw (z)--(2.5,4.4) node{}--(u1);
\draw (v1)-- +(-0.2,-0.3); \draw (v1)-- +(0.2,-0.3);
\draw (v2)-- +(-0.2,-0.3); \draw (v2)-- +(0.2,-0.3);
\node at (2.3,3.5) (pa){}; \node at (3.23,3.5) (qa){};
\node at (1.97,3) (pb){}; \node at (3.57,3) (qb){};
\node[label=left:$P_1~$] at (1.76,2.5) (pc){}; \node[label=right:$~P_2$] at (3.77,2.5) (qc){};
\node at (1.64,2) (pd){}; \node at (3.85,2) (qd){};
\node at (1.60,1.5) (pe){}; \node at (3.83,1.5) (qe){};
\node[draw=none,fill=none] at (5,1.5) {$H^*$};
\node[draw=none,fill=none] at (1.35,1.1) {$C$};
\node[draw=none,fill=none] at (2.75,0.5) {$f$};
\node[draw=none,fill=none,red] at (2.6,1) {\scriptsize$(\leq\!16)$};
\node[draw=none,fill=none,red] at (1.5,3) {\scriptsize$(\leq\!6)$};
\node[draw=none,fill=none,red] at (4.05,3) {\scriptsize$(\leq\!6)$};
\node[red] at (2.6,3.25) (ra) {}; \node[red] at (3.0,3.25) (rb) {};
\node[red] at (2.8,2.25) (rc) {}; \node[red] at (2.7,1.4) (rd) {};
\draw (u1)--(ra);
\draw[red, thick,densely dotted] (rd)--(rc)--(rb)--(ra)--(rc);
\draw[red, thick,densely dotted] (pa)--(ra)--(qa)--(rb)--(pa) (pb)--(ra)--(qb)--(rb)--(pb) (pc)--(ra)--(qc)--(rb)--(pc) (ra)--(pd)--(rb);
\draw[red, thick,densely dotted] (pa)--(rc)--(qa) (pb)--(rc)--(qb) (pc)--(rc)--(qc) (pd)--(rc)--(qd) (pe)--(rc)--(qe) (rc)-- +(-1.05,-1);
\draw[red, thick,densely dotted] (pc)--(rd)--(qc) (pd)--(rd)--(qd) (pe)--(rd)--(qe) (rd)-- +(-0.6,-0.1) (rd)-- +(0.6,-0.1) (rd)-- +(-0,-0.2);
\end{tikzpicture}
\qquad\qquad
\begin{tikzpicture}[scale=1]
\small
\tikzstyle{every node}=[draw, shape=circle, minimum size=3pt,inner sep=0pt, fill=black]
\node at (0,0) (x) {};
\node at (2,-0.5) (xx) {};
\node at (6,0) (y) {};
\node[label=above:$r$] at (3,5) (z) {};
\draw (x)--(xx)--(y) to[bend right] (z) (z) to[bend right] (x);
\draw (x)-- +(0.4,0.3); \draw (x)-- +(0.2,0.5);
\draw (y)-- +(-0.4,0.3); \draw (y)-- +(-0.2,0.5);
\draw[dotted] (0,4.85)--(6,4.85) (0,4.2)--(6,4.2) (0,3.3)--(6,3.3) (0,2.3)--(6,2.3) (0,1.3)--(6,1.3);
\node[label=left:$v_1$] at (1.75,0.5) (v1) {};
\node[label=right:$v_2$] at (3.75,1) (v2) {};
\node[label=right:$~{u}$] at (2.75,4) (u1) {};
\draw[thick] (v2) to[bend right] (u1) (u1) to[bend right] (v1) (v1) -- (v2);
\draw (z)--(3.2,4.7) node{}--(2.5,4.3) node{}--(u1);
\draw (v1)-- +(-0.2,-0.3); \draw (v1)-- +(0.2,-0.3);
\draw (v2)-- +(-0.2,-0.3); \draw (v2)-- +(0.2,-0.3);
\node at (2.3,3.5) (pa){}; \node at (3.23,3.5) (qa){};
\node at (1.97,3) (pb){}; \node at (3.57,3) (qb){};
\node[label=left:$P_1~$] at (1.76,2.5) (pc){}; \node[label=right:$~P_2$] at (3.77,2.5) (qc){};
\node at (1.64,2) (pd){}; \node at (3.85,2) (qd){};
\node at (1.6,1.5) (pe){}; \node at (3.83,1.5) (qe){};
\node[draw=none,fill=none] at (5,1.5) {$H^*$};
\node[draw=none,fill=none] at (1.35,1.1) {$C$};
\node[draw=none,fill=none] at (2.75,0.5) {$f$};
\node[draw=none,fill=none,red] at (2.6,1) {\scriptsize$(\leq\!16)$};
\node[draw=none,fill=none,red] at (1.5,3) {\scriptsize$(\leq\!6)$};
\node[draw=none,fill=none,red] at (4.05,3) {\scriptsize$(\leq\!6)$};
\node[red] at (2.6,2.8) (ra) {}; \node[red] at (3.0,2.8) (rb) {}; \node[red] at (2.8,1.8) (rc) {};
\draw (u1)--(ra);
\draw[red, thick,densely dotted] (rd)--(rc)--(rb)--(ra)--(rc);
\draw[red, thick,densely dotted] (pa)--(ra)--(qa)--(rb)--(pa) (pb)--(ra)--(qb)--(rb)--(pb) (pc)--(ra)--(qc)--(rb)--(pc);
\draw[red, thick,densely dotted] (ra)--(pd)--(rb)--(qd)--(ra) (ra)--(pe)--(rb);
\draw[red, thick,densely dotted] (pb)--(rc)--(qb) (pc)--(rc)--(qc) (pd)--(rc)--(qd) (pe)--(rc)--(qe);
\draw[red, thick,densely dotted] (rc)-- +(-1,-0.55) (rc)-- +(-0.8,-0.55) (rc)-- +(0.85,-0.55);
\end{tikzpicture}
$$
\caption{The schematic outcome $H^*$ of \Cref{lem:core1p} for a trigraph $H$ and the wrapped cycle $C=P_1\cup P_2\cup\{f\}$.
	The horizontal dotted lines sketch the applied double level assignment, as derived from a BFS tree rooted at~$r$.
	Left: the case of the sink $u$ having an even distance from the root~$r$ in~$S$. 
	Right: $u$ having an odd distance from~$r$ (cf.~\Cref{def:doublelev}).
	The sink $u$ has no red edge there.}
\label{fig:corelem1p}
\end{figure}
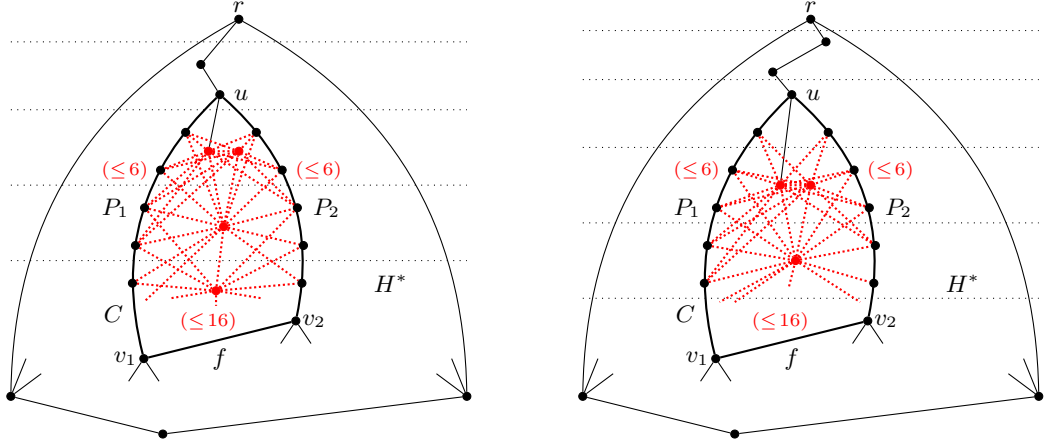

An illustration of \Cref{lem:core1p} can be seen in \Cref{fig:corelem1p}.
We again, before diving into a proof of the lemma, show the sought implications of it.

\begin{lemma}\label{lem:core1pmore}
Let a rooted skeletal trigraph $(H,S,T)$ and $C\subseteq S$ be as in \Cref{lem:core1p}. 
\begin{enumerate}[a)]
\item Suppose $D\subseteq S$ is a $T$-wrapped cycle which is almost $1$-reduced.
Then every vertex of $V(D)$ has at most $4$ edges into $U_0:=V(H_D)\setminus V(D)$ (hence, red degree at most $4$ in $H_D$),
and every vertex of $U_0$ has at most $14$ edges (red or black) in whole $H$.
Consequently,~if~$D$ can be made almost $1$-reduced after deleting at most two vertices of $U_0$, then the said bounds
become $6$ and $16$, respectively, which fulfills the respective conditions in \Cref{lem:core1p}.
\item If the conditions of \Cref{lem:core1p} are fulfilled, and all red faces of $(H,S)$ except that of $C$ are almost $1$-reduced,
$T$-wrapped, and their sinks have no red edges inside the face,
then the maximum red degree of each of the considered trigraphs $H'$ in \Cref{lem:core1p} is at most~$16$.
\end{enumerate}
\end{lemma}
\begin{proof}a)
If $x\in V(D)$, then $x$ can have at most $1$ red neighbour in each one of three levels of $1$-reduced $U_0$,
plus a fourth neighbour in the lowest one if almost $1$-reduced,
Altogether at most~$4$, and plus two more red neighbours in those possibly deleted from~$U_0$.
If $y\in U_0$, then there are up to $6$ neighbours of $y$ on the left wrapping path of $D$ by \Cref{lem:neighonpath} and
up to $5$ such on the right wrapping path of $D$ by the same reason and by \Cref{lem:leftaligndouble}.
Up to $2$ red neighbours may lie in the previous level and $1$ in the next level, summing to~$14$.
Then we may add up to $2$ neighbours among those possibly deleted from~$U_0$.

\smallskip b)
Let $x\in V(H')\setminus V(S)$.
If $x\in U'=V(H'_C)\setminus V(C)$, then the claim is given in \Cref{lem:core1p}\eqref{it:all16}.
Otherwise, $x$ is assigned to an almost $1$-reduced face and this case has been addressed in a).
Let $y\in V(S)$. 
As in the proof of \Cref{lem:coremore}, there are at most two red $S$-faces of which $y$ is not the sink, and so $y$
can have red edges only into those.
By a) and \Cref{lem:core1p}\eqref{it:bound6}, $y$ then has red degree at most $12$.
\end{proof}

\begin{proof}[Proof of \Cref{thm:tww1planar}]
We start with a construction of a $4$-framing skeletal graph $(G,S^+)$ such that $V(S^+)=V(G)$, by \Cref{lem:4framingsk}.
Then we choose a root $r$ on the outer face of $S^+$ and, for some left-aligned BFS tree $T$ of $S^+$ rooted in $r$,
the graph $H=G$ and the (indeed $T$-wrapped) outer face $C=S$ of~$S^+$, we apply \Cref{lem:core1p}.

This way we get a level-preserving partial contraction sequence from $G$ to a trigraph $H^*$ of maximum red degree $16$ 
along the sequence, by \Cref{lem:core1pmore}.
We can hence finish by successively contracting the vertices of $V(H^*)\setminus V(C)$ from the farthest ones,
and finally contracting the last one of them with $V(C)$ in any order, fulfilling the degree-$16$ condition with a large margin.
\end{proof}

\subsection{Finishing the proof}\label{sub:finishing16}
%%%%%%%%%%%%%%%%%%%%%%%%%%%%%%%%%%%%%%%%%%%%%%%%%

\begin{proof}[Proof of \Cref{lem:core1p}]
The coming proof is, again, conceptually similar to the proof of \Cref{lem:core}.
There is a ``decomposition'' phase, with inductive invocations of \Cref{lem:core1p} on ``smaller'' instances of the bounding cycle,
followed by a ``merging'' phase which suitably combines the inductively obtained partial contraction sequences
and extends them further to fulfill the conditions of \Cref{lem:core1p} for the face $\phi$ of~$C$.
Since the decomposition phase is quite similar to the previous proof(s), basically only adding technicalities related
to possible parallel edges in $S^+$ and crossed edge pairs in $E(G)\setminus E(S^+)$, we describe it only briefly here,
and then focus more on the core merging phase.

If $U=\emptyset$, we are immediately done with the empty partial contraction sequence. 
So, we assume $U\not=\emptyset$ and, in particular, $C$ is not a facial cycle of~$H_C\cup S$.

Let $f=\{v_1,v_2\}$ be the lid of $C$ such that $v_1$ belongs to the left wrapping path of~$C$
(and recall that $C$ may be a $2$-cycle, and it may possibly happen $f\not\in E(S)$ -- an added lid).
Consider the bounded face $\sigma$ of $H^+_C=(G\cup S^+)_C$ incident to $f$.
If $\sigma$ is incident to an edge crossing (of~$G$), then, by \Cref{def:4framingsk}, there is an uncrossed $4$-cycle $A\subseteq S^+$
containing~$f$, such that the interior of $A$ contains $\sigma$ and a crossing of the two diagonals of~$A$.
We leave this case for the next paragraph.
Otherwise, $\sigma$ is bounded by a triangle $A\subseteq S^+\cup f$ with the vertices $A=(v_1,v_2,v_3)$.
Let $f_1=\{v_1,v_3\}$ and $f_2=\{v_2,v_3\}$ denote the edges belonging to~$A$ (since there can be parallel edges now).
Let $P_3\subseteq T$ be the vertical path starting in $v_3$ and ending on~$C$.
If either of the edges $f_1$ or $f_2$ belongs to $E(C\cup P_3)$, then we get a $T$-wrapped cycle $C_0:=C\Delta A$,
to which we inductively apply \Cref{lem:core1p}.
This fulfills conditions \eqref{it:all16} and \eqref{it:bound6} also for the cycle $C$, possibly after a final
contraction of $v_3\not\in V(C)$ which is safe by \Cref{lem:core1pmore}\,a) even with red degree at most~$15$ after the contraction of~$v_3$.
Condition \eqref{it:added15} is relevant only when $v_3\not\in V(C)$, and then the previous applies.
The (more complex) case of $f_1,f_2\not\in E(C\cup P_3)$ will be handled below.

Now we return to the case of a $4$-cycle $A\subseteq S^+$, $A=(v_1,v_2,v_3,v_4)$ in this counter-clockwise cyclic order, 
enclosing crossed diagonals $\{v_1,v_3\}$ and $\{v_2,v_4\}$ of~$G$.
Consider the $T$-vertical paths $P_3$ and $P_4$ starting in $v_3$ and $v_4$ and ending on~$C$.
If two of the edges $\{v_2,v_3\}$, $\{v_3,v_4\}$ or $\{v_4,v_1\}$ belong to $E(C\cup P_3\cup P_4)$,
then the third of them is the lid of a $T$-wrapped cycle $C_0:=C\Delta A$, to which we can inductively apply \Cref{lem:core1p}.
As previously, this recursive invocation fulfills conditions \eqref{it:all16} and \eqref{it:bound6} also for the cycle $C$, 
possibly after final contractions of $v_3\not\in V(C)$ and/or $v_4\not\in V(C)$ which is safe by \Cref{lem:core1pmore}\,a).
Condition \eqref{it:added15} is void in this case.
Otherwise, we continue below.

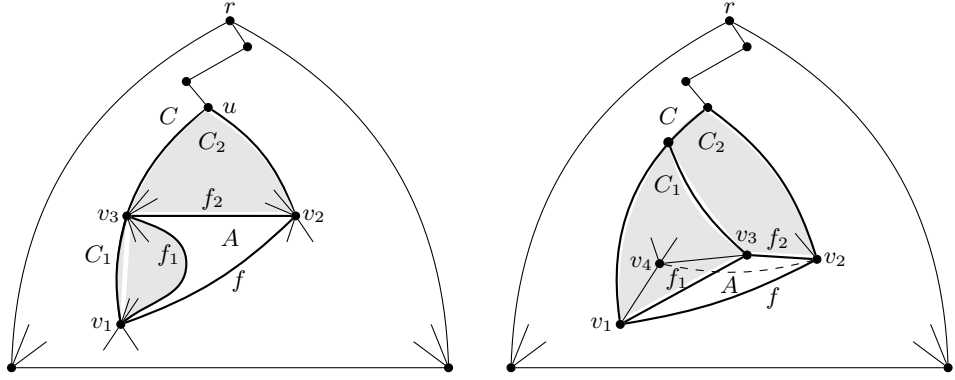
\begin{figure}[tb]
$$
\begin{tikzpicture}[scale=1.15]
\small
\tikzstyle{every node}=[draw, shape=circle, minimum size=3pt,inner sep=0pt, fill=black]
\node at (0.5,0) (x) {};
\node at (5.5,0) (y) {};
\node[label=above:$r$] at (3,4) (z) {};
\draw (x)--(y) to[bend right] (z) (z) to[bend right] (x);
\draw (x)-- +(0.4,0.3); \draw (x)-- +(0.2,0.5);
\draw (y)-- +(-0.4,0.3); \draw (y)-- +(-0.2,0.5);
\node[label=right:$v_2$] at (3.75,1.75) (v2) {};
\node[label=right:$~{u}$] at (2.75,3) (u1) {};
\node[label=left:$v_1$] at (1.75,0.5) (v1) {};
\node[label=left:$v_3$] at (1.82,1.75) (v3) {};
\draw[thick] (v2) to[bend right=18] (u1) (u1) to[bend right] (v1) (v1) to[bend right=12] (v2);
\draw[thick, fill=gray!20] (v1) to[out=40, in=270] (2.5,1.2) to[out=90, in=340] (v3) to[bend right=12] (v1) ;
\draw[draw=none, fill=gray!20] (v2)--(v3) to[bend left=12] (u1) to[bend left=18] (v2);
\draw (z)--(3.2,3.7) node{}--(2.5,3.3) node{}--(u1);
\draw (v1)-- +(-0.2,-0.3); \draw (v1)-- +(0.2,-0.3);
\draw (v1)-- +(0.1,0.3); \draw (v1)-- +(0.2,0.3);
\draw (v2)-- +(-0.1,-0.3); \draw (v2)-- +(0.2,-0.3);
\draw (v2)-- +(-0.25,0.3); \draw (v2)-- +(-0.35,0.15);
\draw (v3)-- +(0.1,-0.3); \draw (v3)-- +(0.25,-0.3);
\draw (v3)-- +(0.35,0.2); \draw (v3)-- +(0.25,0.3);
\draw[thick] (v2)--(v3) ;
\node[draw=none,fill=none] at (2.3,2.9) {$C$};
\node[draw=none,fill=none] at (2.8,2.6) {$C_2$};
\node[draw=none,fill=none] at (1.5,1.3) {$C_1$};
\node[draw=none,fill=none] at (3.1,1) {$f$};
\node[draw=none,fill=none] at (2.3,1.3) {$f_1$};
\node[draw=none,fill=none] at (2.8,1.95) {$f_2$};
\node[draw=none,fill=none] at (3,1.5) {$A$};
\end{tikzpicture}
\qquad
\begin{tikzpicture}[scale=1.15]
\small
\tikzstyle{every node}=[draw, shape=circle, minimum size=3pt,inner sep=0pt, fill=black]
\node at (0.5,0) (x) {};
\node at (5.5,0) (y) {};
\node[label=above:$r$] at (3,4) (z) {};
\draw (x)--(y) to[bend right] (z) (z) to[bend right] (x);
\draw (x)-- +(0.4,0.3); \draw (x)-- +(0.2,0.5);
\draw (y)-- +(-0.4,0.3); \draw (y)-- +(-0.2,0.5);
\node[label=right:$v_2$] at (4,1.25) (v2) {};
\node at (2.75,3) (u1) {};
\node[label=left:$v_1$] at (1.75,0.5) (v1) {};
\node[label=above:$v_3$] at (3.2,1.3) (v3) {};
\draw[draw=none, fill=gray!20] (v2)--(v3) to[bend left=7] (2.3,2.6) to[bend left=12] (u1) to[bend left=18] (v2);
\draw[draw=none, fill=gray!20] (v1)--(v3) to[bend left=21] (2.3,2.6) to[bend right=19] (v1);
\node[label=above:$v_3$] at (3.2,1.3) (v3) {};
\node[label=left:$v_4$] at (2.2,1.2) (v4) {};
\draw[thick] (v2) to[bend right=18] (u1) (u1) to[bend right] (v1) (v1) to[bend right=9] (v2);
\draw (z)--(3.2,3.7) node{}--(2.5,3.3) node{}--(u1);
\draw (v2)-- +(-0.25,0.3);
\draw (v4)-- +(-0.1,0.3); \draw (v4)-- +(0.2,0.3);
\draw[thick] (v2)--(v3)--(v1) ;
\draw (v1)--(v4)--(v3) ;  \draw[dashed] (v4) to[bend right=13] (v2);
\draw[thick] (v3) to[bend left=14] (2.3,2.6) node {} ;
\node[draw=none,fill=none] at (2.3,2.9) {$C$};
\node[draw=none,fill=none] at (2.8,2.6) {$C_2$};
\node[draw=none,fill=none] at (2.3,2.1) {$C_1$};
\node[draw=none,fill=none] at (3,0.95) {$A$};
\node[draw=none,fill=none] at (3.5,0.8) {$f$};
\node[draw=none,fill=none] at (2.4,1.02) {$f_1$};
\node[draw=none,fill=none] at (3.55,1.48) {$f_2$};
\end{tikzpicture}
$$
\caption{Odd cases of the proof of \Cref{lem:core1p} -- ones which could not exist in the proofs 
	presented in \Cref{sec:prooftwwplanar} and \Cref{sec:prooftwwbiplanar},
	and their recursive division into the cycles $C_1$ and $C_2$ (which are shaded in gray).
	Left: the case of triangular $A$, making a $T$-wrapped $2$-cycle $C_1$ with the lid $f_1$ and the sink $v_3$
	 (note that in this case we must have $v_3\not=u$).
	Right: the case of a $4$-cycle $A$ with crossed diagonals; here we obtain the recursive cycles $C_1$ and $C_2$
	by introducing the added lid $f_1=\{v_1,v_3\}$ and ignoring the other diagonal $\{v_2,v_4\}$ in the recursion.
	Observe importance of the condition \eqref{it:added15} of \Cref{lem:core1p} for this construction.}
\label{fig:oddcases1p}
\end{figure}

\vspace*{-2ex}%
\subparagraph{\textcolor{lipicsGray}{Recursive partial contraction sequences~$\pi_1$ and~$\pi_2$.}}
We continue with handling the more complex cases left behind in the previous two paragraphs.

First, we return to the case of a triangular face $\sigma$ bounded by $A=(v_1,v_2,v_3)$ where $E(A)=\{f,f_1,f_2\}$.
If $f_1,f_2\not\in E(C\cup A)$, then, similarly as in the previous proofs of \Cref{lem:core} and \Cref{lem:bicore},
we have $T$-wrapped cycles $C_1$ and $C_2$ in $C\cup A\cup P_3$ such that $C_1\Delta C_2\Delta A=C$.
Note, however, that one of the included possibilities is that $v_3\in V(C)$ forms a nonempty $T$-wrapped $2$-cycle $C_1$ with $v_1$
-- this may happen when $\levll H{v_3}\geq\levll H{v_2}$
(while the symmetric situation with $v_3$ and $v_2$ forming such $2$-cycle is excluded by $T$ being left-aligned).
For clarity, this odd-looking subcase is outlined in \Cref{fig:oddcases1p}.

We inductively invoke \Cref{lem:core1p} on each of $C_1$ and $C_2$ (in any order) with the skeleton $S_1:=S\cup C_1\cup C_2$.
This way we get partial contraction sequences $\pi_1$ and $\pi_2$, respectively.
We can easily see that their concatenation, say $\pi_1.\pi_2$, satisfies the conditions of \Cref{lem:core1p} at every step $H'$ by the induction,
e.g., the vertex of $V(P_3)\cap V(C)$ still has red degree at most $6+0$ (or $0+0$) by the condition on the sink in \eqref{it:bound6} for $C_1,C_2$,
and the vertices of $V(P_3)\cap U'$ have red degree at most $6+6=12$.
In particular, the latter bound of $12<15$ also applies to the vertex $v_3$ considered by \eqref{it:added15}.
We are thus prepared for the merging phase in this case.

In the case of a $4$-cycle $A=(v_1,v_2,v_3,v_4)$ with crossed diagonals, we argue as follows.
If at most one of the edges $\{v_2,v_3\}$, $\{v_3,v_4\}$ or $\{v_4,v_1\}$ belongs to $E(C\cup P_3\cup P_4)$,
then at least one of the paths $P_3$ or $P_4$ starts with an edge not from $E(C\cup A)$.
We assume it~is~$P_3$ starting in $v_3\in U$; while we cannot strictly say that the situation is symmetric with $P_4$ (due to $T$ being left-aligned),
the other case of $P_4$ can indeed be handled in the symmetric way as here.
This subcase is outlined in \Cref{fig:oddcases1p}.
In the subgraph $C\cup P_3\cup\{v_2,v_3\}$, we define two $T$-wrapped cycles -- $C_2$ with the lid $f_2=\{v_2,v_3\}$
(and $P_3$ contained in the left wrapping path of $C_2$) and $C_1$ with the added lid $f_1=\{v_1,v_3\}$ (so $P_3$ contained in the right wrapping path of $C_2$).
We also define $G_1:=G\setminus\{v_2,v_4\}$ (hence $f_1$ is uncrossed in~$G_1$), $H^1:=H\setminus\{v_2,v_4\}$,
and $S_1:=S\cup C_1\cup C_2$.
Then $G_1$ fulfills \Cref{def:4framingsk}, too,
and $H^1$ results from $G_1$ along the same partial contraction sequence which leads to $H$ from~$G$ --
this is since $\{v_2,v_4\}$ is from $H^+_C=(G\cup S^+)_C$ untouched by previous contractions.
Then we may inductively invoke \Cref{lem:core1p} on $C_2$ and $C_1$ (in this order) within the starting graph $G_1$ (not~$G$) and the skeletal trigraph $(H^1,S_1,T)$.

We have inductively obtained partial contraction sequences $\pi_2$ and $\pi_1$, respectively, and now we verify
the conditions of \Cref{lem:core1p} at every step $H'$ of their concatenation $\pi_2.\pi_1$.
Note, however, that now we apply $\pi_2.\pi_1$ back to the trigraph $H$, and so we have to be careful with the vertex $v_2\in V(C)$
and the vertex $v_4'\in U'$ which stems by contractions from $v_4\in U$.

Validity of the condition \eqref{it:bound6} is directly inherited at every vertex of $C$ (including the sink) 
from the inductive invocation of \Cref{lem:core1p}, except for $v_2$.
Regarding $v_2$ (which does not participate in a contraction), we have red degree at most $6$ during $\pi_2$
and at most $4$ right after finishing $\pi_2$ by \Cref{lem:core1pmore}\,a).
The sequence $\pi_1$ then does not influence the red degree of $v_2$ beyond possibly adding one red edge 
to the vertex $v_4'$ that stems from $v_4$ along $\pi_1$, and $4+1<6$.
Condition \eqref{it:all16} holds for every vertex of $U'\cap V(P_3)$ since $6+6<16$, and for every vertex of $U'\setminus V(P_3)$
by the inductive invocations of \Cref{lem:core1p}\,\eqref{it:all16}, 
except that for the vertex $v_4'$ that stems from $v_4$ along $\pi_1$ we have got a stricter upper bound of $15$ by
inductive \Cref{lem:core1p}\,\eqref{it:added15} with $t'=v_4'$, and so the red degree of $v_4'$ in $H'$ is at most $15+1=16$
even when counting a red edge $\{v_4',v_2\}$ in~$H'$.
Condition \eqref{it:added15} is void in this case.
We are thus again ready for the merging phase after $\pi_2.\pi_1$.

\vspace*{-2ex}%
\subparagraph{\textcolor{lipicsGray}{Merging by partial contraction sequence~$\pi_3$.}}
Let $H^2$ be the trigraph resulting from $H$ in the concatenated partial contraction sequence $\pi_2.\pi_1$,
Denote by $u$ the sink of $C$ and by $u_i$, $i=1,2$, the sink of $C_i$, and let $\ell=\levll{H^2}u$ and $\ell_i=\levll{H^2}{u_i}$ be their levels.  
Note that always $u=u_1$ or $u=u_2$ (and both may be true).

We are going to construct a level-preserving partial contraction sequence of $H^2$ which altogether ``merges'' the inductive outcomes.
In the first stage, we define an empty or $1$-step sequence $\pi_3$ as follows.
Assume $u_1\not=u$ (the case of $u_2\not=u$ is handled symmetrically, and this step is simply skipped if $u_1=u_2=u$).
If $V(H^2_{C_1})\setminus V(C_1)$ contains more than one vertex in the level $\ell_1+1$ (two vertices are allowed there by being almost $1$-reduced),
then $\pi_3$ contracts these two vertices.
This adds a red edge into $u_1$, but it is safe with respect to the conditions of \Cref{lem:core1p} by \Cref{lem:core1pmore}\,a)
applied to both $C_1$ and~$C_2$.

Let $\pi_3$ applied to $H^2$ result in the trigraph $H^3$.
We have got that every level of $H^3$ greater than $\ell+1$ contains at most one vertex from each of the sets 
$U_i^3:=V(H^3_{C_i})\setminus V(C_i)$ for $i=1,2$ and at most two vertices from $V(P_3)$.
In the level $\ell+1$ of $H^3$ we may, on the other hand, meet up to $7$ vertices of $U^3:=V(H^3_{C})\setminus V(C)$ ---
up to two from each of $U^3_1$ and $U^3_2$ and up to three from $V(P_3)$ (see the exception of $y$ in \Cref{def:doublelev}).
Let $W\subseteq V(P_3)$ be the set of the $\leq3$ vertices of $P_3$ in the level~$\ell+1$ (note, though, that $W$ may also be empty).

\smallskip
In the second stage, we define the following sequence $\pi_4$ of $H^3$.
For each $x\in W$, ordered by the increasing distance from~$u$, we contract $x$ with the vertex $z\in U^3_1$ of level $\ell+1$
such that $x$ and $z$ have the same adjacency to the sink~$u$.
If such $z$ does not exist, $x$ simply subsumes the position of $z$ for the next step.
Let $H^4$ denote the trigraph resulting from $H^3$ along $\pi_4$.
Then the sink $u$ has no red edge in $H^4_C$, and the set $U^4:=V(H^4_{C})\setminus V(C)$ has no vertices
in levels smaller or equal to $\ell$ and at most $4$ vertices in each level $j\geq\ell+1$.

Before proceeding further, we verify that every step $H'$ along the sequence $\pi_4$ satisfies the conditions of \Cref{lem:core1p}.
Regarding \eqref{it:bound6}, we have created no more than two new red neighbours to a vertex of $V(C)$ (except the sink)
in~$H'$, and this is still fine by \Cref{lem:core1pmore}\,a).
If the condition \eqref{it:added15} concerns one of the contracted vertices in $\pi_4$, then it will be satisfied by the next arguments.
We hence focus on \eqref{it:all16} for $y\in U'=V(H'_C)\setminus V(C)$.
If $y\in V(P_3)\cap V(H')$, 
then $y$ can have red neighbours only in $U^3_1\cup U^3_2$ and the vertices that stem from them by contractions of~$\pi_4$.
So, the red degree of $y$ in this case is not more than the estimated maximum ($6+6$) in $H^3$ which fulfills \eqref{it:all16}.
The same general argument referring to previous degree bounds in~$H^3$ applies if $y\in U^3_1\cap V(H')$ or $y\in U^3_2\cap V(H')$.

Finally, let $y\in V(H')\setminus V(H^3)$ (i.e., $y$ has been created by a contraction in $\pi_4$), and so $\levll{H'}y=\ell+1$.
Let $P_1$ and $P_2$ be the left and right wrapping paths of $C$, respectively.
Then $y$ has (in the levels up to~$\ell+2$) at most $5$ neighbours in $P_1$, and at most $4$ neighbours in $P_2$ 
by additionally using \Cref{lem:leftaligndouble} for the cycle~$C$.
There are $3$ more potential neighbours being the vertices of $U'\setminus V(P_3)$ at the level $\ell+1$.
In a sum, so far, up to~$12$.
Furthermore, we define $k\in\{0,1,2\}$ as the maximum such index that a vertex $x\in W\cap L_{2\ell+k}$
(cf.~\Cref{def:doublelev}) has been contracted into $y$ in~$H'$.
Suppose $k=0$ (this can happen only if $u\in L_{2\ell-1}$).
Then, among the two vertices of $U'\setminus V(P_3)$ in the level $\ell+2$, only the one in $U^3_1$ can be a neighbour of $y$,
and among the vertices of $P_3$ only those three at distances $k+2,k+3,k+4$ from $u$ in $S_1$ can be neighbours of $y$
(i.e., excluding the vertex in the upper half of level $\ell+2$), by applying \Cref{lem:leftaligndouble} to the cycle~$C_1$.
Suppose $k=1$.
Then, compared to the previous case, $y$ may have both neighbours in $U'\setminus V(P_3)$ in the level $\ell+2$,
but now only two neighbours in $P_3$ at distances $k+3,k+4$ from $u$.
Finally, for $k=2$ the difference from the previous is that $y$ can have two neighbours in $P_3$ now at distances $k+4,k+5$ from $u$.
In all three cases the bounds sum up to~$16$.

\smallskip
In the third and final stage, we define a partial contraction sequence $\pi_5$ of $H^4$, which sequentially contracts
each level of $U^4:=V(H^4_C)\setminus V(C)$ down to one vertex except near the sink, in order from the farthest level.
Formally, let $m=\max\{\levll{H^4}x:x\in U^4\}$.
For $i:=m,m-1,\ldots,\ell+2$, as long as there is more then one vertex of level $i$ in $U^4$, we contract a chosen pair of them
in the following order of preference.
We first take (if existing) the vertex in $U^4\cap U^3_2$ of the level $i$, and contract it with the vertex of $P_3$
(again, if existing) in the upper half of level~$i$.
Then we contract one of the previous or their contraction result with the vertex of $P_3$ in the lower half of level~$i$
(again, if both existing).  %, and call this {\em move-2}.
Otherwise, and/or for the rest of the level~$i$, we contract in an arbitrary order.  %, and call these {\em move-3}.
Finally, in the level $i=\ell+1$, we contract the pairs with the same adjacency to the sink $u$ in any order.

The concatenation of the previously defined sequences in order $\pi_2.\pi_1.\pi_3.\pi_4.\pi_5$ produces
a trigraph $H^5=H^*$ in which the face $\phi$ of $C$ is almost $1$-reduced.
What remains to show is that the conditions of \Cref{lem:core1p} hold at each step $H'$ (contracting in the level~$i$)
of the sequence~$\pi_5$.
Consider the condition \eqref{it:bound6} and $x\in V(C)$ such that~$\levll{H'}x=j$.
If $x=u$ is the sink, then we have created no red edge into it.
If $x\in V(P_1)$, then $x$ may have up to $4$ red neighbours in $U'=V(H'_C)\setminus V(C)$
in the level $\ell+1$ if $j\leq\ell+2$, and at most $1$ red neighbour in each of the levels $j-1>\ell+1$, $j>\ell+1$ and $j+1>\ell+1$
since such neighbour of $x$ may only be the one from $U^3_2$ or a vertex that stems by a contraction in $\pi_5$ from it.
This certifies the upper bound of $6$ for this~$x$.
If $x\in V(P_2)$, then $x$ may have only at most $2$ red neighbours in $U'$ in the level $\ell+1$
if $j\leq\ell+2$, which follows from the definition of $\pi_4$, or up to $3$ when the contractions of $\pi_5$ reach the level $\ell+1$.
In each of the levels $j-1>\ell+1$, $j>\ell+1$ and $j+1>\ell+1$, there is at most $1$ red neighbour of $x$ -- the one coming from $U^3_1$,
or up to $2$ red neighbours in the precise level $i\in\{j-1,j,j+1\}$ being contracted now.
These sum to at most $6$ in any case.

Consider the condition \eqref{it:all16} for $y\in U'=V(H'_C)\setminus V(C)$ such that~$\levll{H'}y=j$.
If $y\in U^4$ (not contracted yet by $\pi_5$), then by the definition of $\pi_5$, 
there are no more red neighbours of $y$ in $H'$ than what was the estimated maximum in $H^2$, which fulfills \eqref{it:all16}.
For the rest, assume $j\geq\ell+2$.
If $i<j$ (i.e., the whole level $j$ in $U'$ has already been contracted into~$y$), then $y$ has, by the previous, 
at most $6$ neighbours in $V(P_1)$, at most $5$ in $V(P_2)$ thanks to \Cref{lem:leftaligndouble}, 
and at most $1+4$ in the next and previous levels of~$U'$.
These upper estimates sum to~$16$.
The last case is that $y$ has just been contracted by $\pi_5$ in the level $j$ of~$U'$.
Then $y$ has at most $5$ neighbours in $V(P_2)$ again thanks to \Cref{lem:leftaligndouble},
and at most $4$ neighbours in $U'$ in the level $j-1$ and at most $1$ in the level $j+1$ (no more vertices exist in those level),
together at most~$10$.
It remains to count the neighbours of $y$ in the level $j$ of $U'$ and in $V(P_1)$, which we divide into cases
from the definition of $\pi_5$ as follows.

If $y$ stems from a contraction of $s\in U^4\cap U^3_2$ with $z\in V(P_3)$ in the upper half of level~$j$,
then $y$ has up to $2$ neighbours in the level $j$ of $U'$ and $4$ neighbours in the levels $j$ and $j+1$ of $V(P_1)$.
However, in $H^4$ there is no edge from $z$ to vertices in the level $j-1$ of $V(P_1)$ thanks to \Cref{lem:leftaligndouble},
and no such edges from $s$ either due to the skeleton $S_1$ in~$H^3$.
Hence, there is no edge in $H'$ from $y$ to vertices in the level $j-1$ of $V(P_1)$, and the upper bound is~$10+4+2=16$.
If $y$ stems from a (subsequent) contraction with $z'\in V(P_3)$ in the lower half of level~$j$,
then $y$ has $1$ neighbour in the level $j$ of $U'$ and up to $5$ neighbours in $V(P_1)$, this time excluding
an edge into the vertex in the lower half of level $j-1$ of $V(P_1)$, again due to \Cref{lem:leftaligndouble}.
Finally, if $y$ stems by the last contraction in the level $j$, then the case has been considered above.

The last subcase to consider is $y\in U'\setminus U^4$ in the level $\levll{H'}y=\ell+1$.
Then, with a large margin, $y$ can have $6$ neighbours in $V(P_1)$, plus $5$ in $V(P_2)$, and plus at most
$3+1=4$ in the levels $\ell+1$ and $\ell+2$ of~$U'$, giving at most~$15$.

We are left with the condition \eqref{it:added15}.
It follows directly from the previous estimates (which led to~$\leq16$) by realizing that for such $t'$ as in \eqref{it:added15}
we have no vertex in the upper half of level $\levll{H'}{t'}+1$ of $V(P_1)$, and such vertex was, on the other hand,
included as a neighbour of~$t'$ in all previous estimates.
The whole proof is now finished.
\end{proof}

\section{Proof of \Cref{thm:mapplanar}; the Map-Graph Case}
\label{sec:prooftwwmapplanar}
%%%%%%%%%%%%%%%%%%%%%%%%%%%%%%%%%%%%%%%%%%%%%%%%%%%%%%%%%%%%%%%%%%%%%%%%%%%%%%%%%

We start with an exact characterization of map graphs that we will use in our proof.

\begin{proposition}[Chen, Grigni, and Papadimitriou~\cite{10.1145/506147.506148}]\label{pro:mapgraphs}
A graph $M$ is a map graph, if and only if there exists a simple bipartite planar graph $G$ with a bipartition $(A,B)$
such that~$M\simeq G^{(2)}[A]$ (the subgraph of the square of $G$ induced on the set~$A$).
\end{proposition}

The outline of this section is as follows.
We are going to prove that the squares of planar quadrangulations (which are bipartite) have small twin-width.
We will use an approach very similar to that of \Cref{sec:biplanar} in some aspects, but also significantly
different in other aspects.
On the similar side, we are going to construct a level-preserving partial contraction sequence of the underlying
plane quadrangulation, which thus preserves bipartiteness of the trigraphs resulting from this quadrangulation,
and we will apply the same recursive-decomposition approach as used in the proof of \Cref{lem:bicore}.

First we need to show that it is indeed sufficient to consider squares of planar quadrangulation instead of 
squares of bipartite planar graphs in general, which is not a priori clear.
Let the square of a graph $G$ be denoted by~$G^{(2)}$.

\begin{lemma}\label{lem:toquadrang2}
Let $G_0$ be a simple bipartite planar graph. 
Then there exists a simple $2$-connected plane quadrangulation $G$ such that $G_0^{(2)}$ is an induced subgraph of $G^{(2)}$.
\end{lemma}
\begin{proof}
Fix a plane drawing of $G_0$ and let $\phi$ by a face which is not a quadrangle.
Then $\phi$ is bounded by a walk of an even length~$m\geq6$ which visits the vertices of $\phi$
in the cyclic order $(v_1,v_2,\ldots,v_m)$, possibly with repetition.
We embed into $\phi$ a new cycle $C$ of length $m$ on the vertex set in order $(x_1,x_2,\ldots,x_m)$
and connect $x_i$ to $v_i$ for~$i=1,\ldots,m$.
Then, inside $C$, we add another new vertex $y$ adjacent exactly to $x_2,x_4,\ldots,x_m$.
This we do for every non-quadrangular face of $G_0$ and denote by $G$ the resulting graph.

The result $G$ is by the definition a plane quadrangulation, and it is easy to see the $2$-connectivity property.
Moreover, no vertex of $V(G)\setminus V(G_0)$ is adjacent to two vertices of $V(G_0)$, and hence the square
operation on $G$ adds no new edges between the vertices of $G_0$ in comparison to $G_0^{(2)}$ itself.
\end{proof}

On the other hand, the crucial conceptual difference of the coming proof compared to the proofs 
in the three previous sections is informally explained as follows.
While we are going to construct our contraction sequence recursively from the plane quadrangulation $G$ of \Cref{lem:toquadrang2},
at the same time we have to limit the red degree in the corresponding contractions of the square graph $G^{(2)}$.
We are going to achieve the latter task indirectly -- by imposing additional conditions on the trigraphs
obtained from $G$ (and without an explicit reference to the contractions of $G^{(2)}$).
For that we will use the following:

\begin{lemma}\label{lem:redsquare}
Let $G$ be a simple graph and $\pi$ be a partial contraction sequence of $G$ resulting in a trigraph~$H_1$.
Let the same sequence $\pi$ applied to the square graph $G^{(2)}$ result in a trigraph~$H_2$.
If there is a red edge $f=\{u,w\}\in E(H_2)$, then $H_1$ contains a path of length at most $2$ from $u$ to~$v$,
but no such path with only black edge(s).
\end{lemma}
\begin{proof}
Assume that $u$ stems from vertices $U\subseteq V(G)$ and $w$ stems from vertices $W\subseteq V(G)$, where~$U\cap W=\emptyset$.
Since $f$ is red, up to symmetry, there exist $u_0\in U$ and $w_1,w_2\in W$ such that $\{u_0,w_1\}\in E(G^{(2)})$, 
but $\{u_0,w_2\}\not\in E(G^{(2)})$.
If $\{u_0,w_1\}\in E(G)$, then $f=\{u,w\}$ is a red edge in~$H_1$ (since $\{u_0,w_2\}\not\in E(G)$).
If there is a length-$2$ path $(u_0,t_0,w_1)$ in $G$, then $\{u_0,t_0\},\{t_0,w_1\}\in E(G)$, 
but $\{t_0,w_2\}\not\in E(G)$ since $\{u_0,w_2\}\not\in E(G^{(2)})$.
So, denoting by $t$ the vertex of $H_1$ which is or stems from~$t_1$,
the edge $\{t,w\}\in E(H_1)$ is red and we have a path $(u,t,w)$ in $H_1$ with a red edge.

On the other hand, assume that $f=\{u,w\}$ is a black edge in $H_1$, or there is a path $(u,t,w)$ in $H_1$ with both edges black.
Then every vertex of $U$ is adjacent to every vertex of $W$ in $G^{(2)}$, and this contradicts the assumption that $\{u,w\}$ is red in~$H_2$.
\end{proof}

\begin{corollary}\label{cor:redsquare}
Let $G$ be a simple plane graph, and $C\subseteq G$ be a cycle such that $(G,C)$ is a skeletal graph with the natural face assignment.
Consider a $C$-respecting partial contraction sequence $\pi$ of $G$ which results in a trigraph $H_1$.
Let the same sequence $\pi$ applied to the square~$G^{(2)}$ produce a trigraph~$H_2$.
If there is no red edge incident to a vertex of $C$ in $H_1$, then no red edge in $H_2$ crosses the cycle~$C$
-- in the natural face assignment of~$(H_1,C)$.
\end{corollary}
\begin{proof}
Assume a red edge $f=\{u,w\}\in E(H_2)$ such that $u$ and $v$ are assigned to distinct faces of $C$ in~$(H_1,C)$.
Then $f$ cannot be an edge of $H_1$.
So, by \Cref{lem:redsquare}, there is a path of length~$2$ from $u$ to $v$, which thus has to intersect $C$.
And since this path cannot be only black, we have got a red edge incident to a vertex of~$C$ -- a contradiction.
\end{proof}

We will use two special concepts in the formulation of our restrictions on the contracted trigraphs.
For a set $X\subseteq V(H)$ of a trigraph $H$, we say that two vertices $u,v\in V(H)$ are {\em twins with respect to~$X$}
if their neighbourhoods in $X$ are the same, i.e., $N_H(u)\cap X=N_H(v)\cap X$.
We do not care about the colour of edges in this definition, since we will use it only when there are no red edges incident to~$X$ in~$H$.
Furthermore, for a vertex $v$ of a trigraph $H$, we say that $w$ is a {\em non-black neighbour} of $v$ if $\{v,w\}\in E(H)$ is a red edge,
and that $w$ is a {\em non-black $2$-neighbour} of $v$ if the distance from $v$ to $w$ in $H$ is exactly $2$
and there is no length-$2$ path from $v$ to $w$ in $H$ with both edges black (cf.~\Cref{lem:redsquare}).

Similarly to \Cref{lem:bicore}, we now formulate the core inductive statement:
\begin{lemma}\label{lem:mapcore}
Let $G$ be a simple plane quadrangulation and $T$ a left-aligned BFS tree of~$G$ rooted at a vertex $r\in V(G)$ of the outer face.
Let $S\subseteq G$, $r\in V(S)$, and let a trigraph $H\supseteq S$ be obtained from $G$ in a level-preserving $S$-aware partial contraction sequence,
such that $(H,S,T)$ is a rooted proper skeletal trigraph.
Let $C\subseteq S$ be a $T$-wrapped facial cycle~of~$S$.

Assume that $H_C=G_C$ (the subgraph $H_C$ of $H$ bounded by $C$ has not been touched by any contraction since~$G$), 
and let $U:=V(H_C)\setminus V(C)$ be the vertices of $H$ bounded by~$C$.

Then there exists a level-preserving partial contraction sequence of $H$ which contracts only pairs of vertices that 
are in or stem from $U$ (and so is $S$-aware), ends in a trigraph $H^*$,
and the following conditions are satisfied for every trigraph $H'$ along this sequence from~$H$~to~$H^*$:
\begin{enumerate}[(i)]
\item\label{it:all38}
Every vertex of $U':=V(H'_C)\setminus V(C)$ has at most $25$ non-black neighbours in $H'$,
and at most $38$ non-black $2$-neighbours in $H'$.
\item\label{it:bound28}
No vertex of $C$ is incident to a red edge in $H'_C$, and
every vertex of $C$ has at most $25$ non-black $2$-neighbours in~$V(H'_C)$.
\item\label{it:nomorecontr}
At the end of the sequence, \emph{no} two vertices of $U^*:=V(H^*_C)\setminus V(C)$ in $H^*$ of the same level are twins with respect to~$V(C)$.
\end{enumerate}
\end{lemma}

We again, before diving into a proof of \Cref{lem:mapcore}, show related technical statements.

\begin{lemma}\label{lem:mapcoremore}
Assume the setup of \Cref{lem:mapcore}, and let $H^o$ denote a trigraph which is one of $H$ or $H'$ or $H^*$ of the lemma.
Let $D\subseteq S$ be a $T$-wrapped facial cycle~of~$S$ such that no vertex of $D$ is incident to a red edge in $H^o_D$.
\begin{enumerate}[a)]
\item\label{it:8reduced}
Every vertex of $U^o:=V(H^o_D)\setminus V(D)$ in $H^o$ has at most $3$ neighbours in~$V(D)$, and at most $2$ in each wrapping path of~$D$.
Consequently, if no two vertices of $U^o$ in $H^o$ of the same level are twins with respect to~$V(D)$,
then the face bounded by $D$ is $8$-reduced.
\item\label{it:mapcmsink}
If $x\in V(H^o_D)$ has distance two in $H^o$ from the sink $u^o$ of $D$, then $x$ is formed by contractions of only
such vertices of $G$ that have distance two from~$u^o$ in $G$.
\item\label{it:redn10}
If, again, no two vertices of $U^o$ in $H^o$ of the same level are twins with respect to~$V(D)$,
then every vertex of $V(D)$ has at most $10$ non-black $2$-neighbours in the set~$V(H^o_D)$.
\end{enumerate}
\end{lemma}

\begin{proof}
a) By \Cref{lem:bipartlevels}, a vertex $x\in U^o$ can have neighbours only in the levels $\levll{H^o}x-1$ and $\levll{H^o}x+1$,
which means at most two neighbours in each of the two wrapping paths of~$D$, and the vertex in the level $\levll{H^o}x+1$ of the
right wrapping path of~$D$ is actually excluded by \Cref{lem:leftalign}.
This shows the first part, and the second part then follows easily
since there are at most $2^3=8$ distinct neighbourhoods in the three vertices.

b) Since all vertices of $V(H^o_D)$ other than the sink~$u^o$ are of levels at least $\levll{H^o}u+1$,
the only vertices at distance two from $u^o$, say $t\in V(H^o_D)$, must satisfy $\levll{H^o}t=\levll{H^o}{u^o}+2$ by \Cref{lem:bipartlevels}.
Then every vertex $t_1\in V(G)$ which contracts into $t$ belongs also to the level $\levll{H^o}t=\levll{G}{u^o}+2$ in~$G$,
and by the levels and the face assignment in $(G,S)$, the distance from $t_1$ to $u^o$ in $G$ is two, too.

c) For $y\in V(D)$, let $M_1$, $M_2$, and $M_3$ denote in order the subsets of vertices of~$V(H^o_D)$ in the levels
$\levll{H^o}y-2$, $\levll{H^o}y$, and $\levll{H^o}y+2$.
By \Cref{lem:bipartlevels} and the assumption of no red edge incident to~$V(D)$, 
a non-black $2$-neighbour of $y$ in~$H^o_D$ may only lie in $M_1$, $M_2$, or $M_3$.
Let $y_1$ be the neighbour of $y$ on $D$ closer to the sink~$u^o$ (if $y\not=u^o$), 
and $y_3$ be the neighbour of $y$ on $D$ farther from $u^o$ (if existing on the same wrapping path of $D$).
Observe that if $x\in M_1\cup M_2\cup M_3$ is a neighbour of $y_1$ or $y_3$ in $H^o$
(hence $\{x,y_i\}$ is a black edge by the assumption, and so is $\{y_i,y\}$), 
then $x$ is not a non-black $2$-neighbour of~$y$ by the definition.

If $y$ lies on the left wrapping path of $D$, then, by the same arguments as in a), at most $4$ vertices of $M_1$
are not neighbours of $y_1$, at most $4$ vertices of $M_3$ are not neighbours of $y_3$, and
at most $2$ vertices of $M_2$ are not neighbours of either of $y_1$, $y_3$. 
This case sums to at most~$10$.
If $y$ lies on the right wrapping path of $D$, then, analogously, at most $4$ vertices of $M_3$ are not neighbours of $y_3$,
and at most $4$ of $M_2$ are not neighbours of $y_1$.
Furthermore, by \Cref{lem:leftalign} applied to each of $y$ and $y_1$, no vertex of $M_1$ is at distance two from~$y$ in~$H^o$.
This case sums to at most~$8<10$.
\end{proof}

The next claim is analogous to \Cref{lem:bicoremore}, but due to the special context here its formulation is somehow cumbersome.

\begin{lemma}\label{lem:mapcorefull}
Assume the assumptions and conclusions of \Cref{lem:mapcore}. 
Moreover, assume that in $(H,S)$ every facial cycle $D\subseteq S$, $D\not=C$, such that $H_D\not=G_D$, satisfies the following;
the cycle $D$ is $T$-wrapped, no red edge of $H_D$ is incident to a vertex of~$D$,
and {no} two vertices of $V(H_D)\setminus V(D)$ in $H$ of the same level are twins with respect to~$V(D)$.
For every step $H'$ of the sequence considered in \Cref{lem:mapcore}, let $H''$ denote the trigraph obtained from the square $G^{(2)}$
along the same partial contraction sequence as that of $G$ leading to~$H'$.
Then the following is true for every such~$H''$:
\begin{enumerate}[a)]
\item The maximum red degree of $H''$ is at most~$63$.
\item If $A\subseteq V(H')$ denotes one colour class of $H'$ (which is bipartite by \Cref{lem:bipartlevels}), then
the induced subgraph $H''[A]$ has maximum red degree at most~$38$.
\end{enumerate}
\end{lemma}

\begin{proof} a)
By \Cref{lem:redsquare}, red neighbours of any $x\in V(H'')=V(H')$ can only be among non-black neighbours or $2$-neighbours of $x$ in~$H'$.
Hence, for $x\in U'$ the claim is given in \Cref{lem:mapcore}\eqref{it:all38} as~$25+38=63$.
If $x\in V(H')\setminus V(S)$ is not as in the previous case, then $x$ is assigned to an $S$-face bounded by $D\subseteq S$ in $(H',S)$,
and $D$ bounds the same face in $(H,S)$.
If $H_D=G_D$, then $x$ has no red edge in $H''$ by \Cref{cor:redsquare}.
If $H_D\not=G_D$, then the face of~$D$ is $8$-reduced in $(H',S)$ by \Cref{lem:mapcoremore}\eqref{it:8reduced}, and so we have 
less than $50$ potential red neighbours of $x$ in $H''$ which have to be at distance at most two from $x$ in~$H'_D$
(and less than $30$ at distance exactly two).
Finally, consider $x\in V(S)$ and an $S$-facial cycle $D$ incident to~$x$.
If $x$ is the sink of $D$, then there is no red edge between $x$ and $V(H'_D)$ in $H''$ by \Cref{lem:mapcoremore}\eqref{it:mapcmsink}.
Otherwise, there are at most two facial cycles $D$ incident to $x$ not at their sinks, and they together account
for at most $10+25<38$ non-black $2$-neighbours of $x$ in whole~$H'$ by
\Cref{lem:mapcoremore}\eqref{it:redn10} and \Cref{lem:mapcore}\eqref{it:bound28}.
Since there is no non-black neighbour of $x$ by the assumption (of no red edge incident to~$V(D)$),
the maximum red degree of $x\in V(S)$ in $H''$ is at most~$38$.

b) This follows from the previous arguments if we restrict ourselves only to non-black $2$-neighbours in $H'$
as the potential red neighbours of a vertex of $A$ in $H''[A]$.
\end{proof}

\begin{proof}[Proof of \Cref{thm:mapplanar}]
Let $G_0$ be the given planar bipartite graph, and $G$ be the plane quadrangulation from \Cref{lem:toquadrang2}.
We choose a root $r$ on the outer face of $G$ and, for some left-aligned BFS tree $T$ of $G$ rooted in $r$ which exists by \Cref{clm:existslal},
the graph $H=G$ and the outer face $C=S$ of~$G$, we apply \Cref{lem:mapcore}.

This way we get a level-preserving partial contraction sequence from $G$ to a trigraph $H^*$,
and the corresponding partial contraction sequence of the square $G^{(2)}$ is of maximum red degree at most~$63$ by \Cref{lem:mapcorefull}.
In $H^*$, one face of $C$ is empty and the other one is $8$-reduced by \Cref{lem:mapcoremore}\eqref{it:8reduced}.
We can now easily finish by successively contracting the vertices of $H^*$ in order of their decreasing levels (and arbitrarily within each one level).
The red degree in the corresponding continuation of the contraction sequence of $G^{(2)}$ in this final phase 
clearly does not exceed $5\cdot8-1+4=43<63$.
The restriction of this whole contraction sequence of $G^{(2)}$ to only $V(G_0)$ of the induced subgraph $G_0^{(2)}$ then conforms to the same bounds,
and so the twin-width of $G_0^{(2)}$ is at most~$63$.

Regarding map graph, by \Cref{pro:mapgraphs} they are each an induced subgraph of $G_0^{(2)}[A]$ of the square $G_0^{(2)}$ 
for a suitable choice of a planar bipartite graph $G_0$ and its colour class $A\subseteq V(G_0)$.
Repeating the previous arguments, we now obtain an upper bound on the read degree of $38$ by \Cref{lem:mapcorefull}.
\end{proof}

\subsection{Finishing the proof}\label{sub:finishing38}
%%%%%%%%%%%%%%%%%%%%%%%%%%%%%%%%%%%%%%%%%%%%%%%%%

\begin{proof}[Proof of \Cref{lem:mapcore}]
In this proof we use induction and the recursive decomposition of $H_C=G_C$ with inductive invocation of \Cref{lem:mapcore}
exactly as in the proof of \Cref{lem:bicore} before (cf.~\Cref{fig:divisionbi6}), and so we skip repeating the same details here.
The only interesting case that remains to be resolved differently here is the one of two inductively obtained
partial contraction sequences $\pi_1$ and $\pi_2$ (contracting inside the cycles $C_1$ and $C_2$, respectively, as specified in the previous proof), 
for which their concatenation already satisfies the conditions of \Cref{lem:mapcore} at their every step.

Let $H^2$ denote the trigraph resulting from $G$ in the partial contraction sequence $\pi_1.\pi_2$,
and denote by $U_0:=V(H^2_C)\setminus V(C)$, by $U_j:=V(H^2_{C_j})\setminus V(C_j)$ for~$j=1,2$, and by $\{u_3\}=V(P_3)\cap V(C)$.
Let $k$ be the minimum value of $\levll{H^2}x$ over $x\in V(P_3)\cap U_0$, and $m$ the maximum value of $\levll{H^2}y$ over $y\in U_0$.
In particular, $\levll{H^2}{u_3}=k-1$.

The sequence $\pi_3$ is constructed in stages~$i=k,k+1,\ldots,m$ in this order, where $i$ means the level of $U_0$ in which contractions happen
(notice that we use the opposite order compared to previous proofs, simply since it works better now).
Our construction has an exceptional subcase when $i=k$ and $u_3$ is the sink of~$C_2$.
In this case, all vertices of $U_2$ of the level $i=k$ have the same neighbourhood in $V(C)$ by \Cref{lem:leftalign} and the fact that
they all have to be adjacent to the sink~$u_3$ (actually, there cannot be more than two such vertices in~$U_2$).
Sequence $\pi_3$ then starts with contracting all vertices of $U_2$ of the level $i=k$ into one, and continues as in the generic case below.

In all remaining cases, that is when $i>k$ or $u_3$ is not the sink of~$C_2$,
stage $i$ of $\pi_3$ starts with contracting all pairs of vertices of $U_1$ of the level $i$ which are twins width respect to $V(C)$,
followed by contracting all pairs of $U_2$ of the level $i$ which are twins width respect to $V(C)\cup\{p_{i+1}\}$
where $p_k\in V(P_3)$ denotes the vertex of $P_3$ of the level~$k$.
Observe that after these contractions, we are left with at most $4$ vertices on each side in the level~$i$.
After that, and also in the exceptional case of the previous paragraph, we finish stage $i$ of $\pi_3$ by contracting 
all pairs of vertices that stem from $U_0$ in the level $i$ and which are twins with respect to $V(C)$.
At the end, the trigraph $H^*=H^3$ that results from $H^2$ in~$\pi_3$ satisfies the final condition \eqref{it:nomorecontr} of \Cref{lem:mapcore}.

\smallskip
We are left with proving the conditions \eqref{it:all38} and \eqref{it:bound28} of \Cref{lem:mapcore}
for every trigraph $H'$ along~$\pi_3$ in the stage (i.e., level just being contracted in~$H'$) $i\in\{k,k+1,\ldots,m\}$.

As for the condition \eqref{it:bound28}, pick~$y\in V(C)$.
Recall that non-black $2$-neighbours of $y$ in $H'_C$ belong to the set of levels $J\subseteq\{\levll{H'}y,\levll{H'}y\pm2\}$ by \Cref{lem:bipartlevels},
but the level $\levll{H'}y-2$ is excluded from $J$ if $y$ lies on the right wrapping path by \Cref{lem:leftalign}.
Furthermore, with respect to~$i$, non-black $2$-neighbours of $y$ in $H'_C$ in levels greater than $i+1$ may, by induction, 
only come from the set $U'':=U'\cap U_1$ if $y$ lies on the left wrapping path of $C$ and from $U'':=U'\cap U_2$ otherwise.
Likewise, non-black $2$-neighbours of $y$ in $H'_C$ in levels smaller than $i$ come from the set $U'=U''$ 
which already is $8$-reduced by $\pi_3$ in these levels.
However, we can say more about non-black $2$-neighbours of $y$ in a level $j\in J\setminus\{i,i+1\}$.
By analogous arguments as in the proof of \Cref{lem:mapcoremore}\eqref{it:redn10}, out of the up to $8$ vertices of $U''$
in the level~$j$, only at most $4$ are not adjacent to $y$ by a length-$2$ black path in~$H'$, and only those $4$ can
be non-black $2$-neighbours of $y$ by the definition.

In a summary, the number of possible non-black $2$-neighbours of $y$ in the (at most one) level $j\in J\cap\{i,i+1\}$ of $U'$
is at most $8+1+8=17$, and in each of the remaining levels from $J\setminus\{i,i+1\}$ it is at most $4$, summing to at most~$17+4+4=25$.

\smallskip
As for the condition \eqref{it:all38} of \Cref{lem:mapcore}, we pick~$x\in U'$ and first estimate the number of non-black neighbours of~$x$ in $H'$.
We first observe that non-black neighbours of~$x$ can only come from $U'$ by the assumption of $C$ having no incident red edge and by \Cref{cor:redsquare}.

If we are, within $\pi_3$, in stage $i<\levll{H'}x-1$ and $x\not\in V(P_3)$, 
then for each level $j\in\{\levll{H'}x\pm1\}$ (in which no contraction has happened so far), 
neighbours of $x$ in the level $j$ are only from the set $U_k\cap U'$ where $k\in\{1,2\}$ is such that $x\in U_k$, by induction.
If $x\in V(P_3)$, on the other hand, there is no non-black neighbour of~$x$ in the level $j$.
Altogether, if $i<\levll{H'}x-1$, then we have at most $8+8=16$ non-black neighbours of~$x$ in whole~$H'$.

If $i=\levll{H'}x-1$, then we use the same estimate as previously for neighbours in the level $\levll{H'}x+1$,
and count all up to $8+8+1=17$ potential neighbours of $x$ in $U'$ in the level $\levll{H'}x-1$, giving the upper bound of~$8+17=25$.
Finally, if $i\geq\levll{H'}x$, then we have got $U'$ $8$-reduced in the level $\levll{H'}x-1$, and we use the ``large'' estimate of $17$
as previously for the level $\levll{H'}x+1$, again arriving at the upper bound of~$25$.

We continue with giving an upper bound on the number of non-black $2$-neighbours of~$x\in U'$.
These can, by \Cref{cor:redsquare}, come from the levels $\levll{H'}x$ and $\levll{H'}x\pm2$ of $V(H'_C)$.
First consider stages $i\leq\levll{H'}x-2$ of $\pi_3$, that is, no contractions have happened so far in the levels $\levll{H'}x-1$ and up.
If $x\not\in V(P_3)$, then $x$ has at most $10$ non-black $2$-neighbours in each of the levels $\levll{H'}x$ and $\levll{H'}x+2$
by induction, and if $x\in V(P_3)$, then the bound is at most $8$ in each of $\levll{H'}x$ and $\levll{H'}x+2$ by
the arguments of \Cref{lem:mapcoremore}\eqref{it:redn10} and induction.
Moreover, we can say more, for $x\not\in V(P_3)$ and the level $\levll{H'}x+2$, the upper bound can be decreased to $9$ using \Cref{lem:leftalign}.
The number of non-black $2$-neighbours in the level $\levll{H'}x-2$ of $V(H'_C)$ is simply at most $8+8+3=19$.
Altogether, we have an upper bound of $10+9+19=38$, as desired.

Second, we look at the stage $i=\levll{H'}x-1$ of $\pi_3$.
Then the number of non-black $2$-neighbours of~$x$ in the level $\levll{H'}x-2$ of $V(H'_C)$ is at most $10$ since we have
already contracted $U'$ down to $\leq8$ vertices there, and for the level $\levll{H'}x+2$ the same bound of at most $9$ holds as previously.
In the level $\levll{H'}x$, we have at most $8+8+3=19$ non-black $2$-neighbours again as previously, and together~$10+9+19=38$.
Third, consider stages $i\geq\levll{H'}x+1$ of $\pi_3$.
Then we have got at most $10$ non-black $2$-neighbours in each of the levels $\levll{H'}x-2$ and $\levll{H'}x$,
and at most $19-1=18$ in the level $\levll{H'}x+2$ after applying \Cref{lem:leftalign},
which again sums to~$10+10+18=38$.

Finally, consider the stage $i=\levll{H'}x$ of $\pi_3$.
The number of non-black $2$-neighbours of~$x$ in the level $\levll{H'}x-2$ of $V(H'_C)$ is again at most $10$.
If $x\in V(P_3)$, or contractions in the level $i$ so far happened only within $U_1$ and/or $U_2$
(which includes the exceptional case of $\pi_3$ at the sink of $U_2$),
then the number of potential non-black $2$-neighbours of~$x$ in the level $\levll{H'}x+2$ is at most $10$ by induction.
So, in this subcase, we are fine with an upper bound of $18$ non-black $2$-neighbours of~$x$ in the level $\levll{H'}x$.
Otherwise, by the definition of $\pi_3$, previous steps of $\pi_3$ have already contracted the set $U'$ in the level $i$
down to at most $4+4+1=9$, or $8+1+1=10$ in the exceptional case of $\pi_3$ at the sink of $C_2$, 
and the current step of $\pi_3$ decreased this number one further.
Hence, $x$ may now have only at most $10-1+2-1=10$ other non-black $2$-neighbours in the level $\levll{H'}x$.
Since we again have at most $18$ non-black $2$-neighbours in the level $\levll{H'}x+2$ after applying \Cref{lem:leftalign},
the overall upper bound of $10+10+18=38$ holds, too.

We have verified all conditions of \Cref{lem:mapcore} for the defined sequence~$\pi_3$.
\end{proof}

\section{Concluding Remarks}
%%%%%%%%%%%%%%%%%%%%%%%%%%%%%%%%%%%%%%%%%%%%%%%%%%%%%%%%%%%%%%%%%%%%%%%

The primary contribution of our paper is a significant improvement over previously published
upper bounds on the twin-width of planar graphs, which is now just one off the best known lower bound.
There is, however, a broader new contribution of the paper consisting of a set of tools (a toolbox)
focused on providing refined upper bounds on the twin-width of not only planar graphs, but also
of other graph classes which are related to planarity in some suitable way.
We demonstrate wide usability of the provided toolbox by providing similar refined upper
bounds on the twin-width of bipartite planar graphs, of $1$-planar graphs, and of map graphs,
all of which significantly improve previous knowledge of twin-width on the listed classes, too.

On the other hand, while it would be even nicer to provide one general statement, with a unified proof,
tightly bounding the twin-width of all these classes at once, the numerous small (but important at the end)
differences in the use of our toolbox here between Sections \ref{sec:prooftwwplanar},~\ref{sec:prooftwwbiplanar},~\ref{sec:prooftww1planar},
and~\ref{sec:prooftwwmapplanar} indicate that it is hardly a possible task.
Moreover, right in the cases of planar and bipartite planar graphs, 
the paper's proof method and the used toolbox seem to be at the edge of their possibilities.

Related to the twin-width is the notion of reduced bandwidth~\cite{DBLP:journals/corr/abs-2202-11858} which, informally stating,
requires the subgraph induced by the red edges (along the sequence) to not only have bounded degrees, but also bounded bandwidth.
Strictly speaking, as our construction of the contraction sequence creates arbitrarily large ``red grids'' in some cases, it does not
directly imply any constant upper bound on the reduced bandwidth of planar graphs.
However, a simple modification of the construction (informally, delaying contractions that would create red edges to the vertices of the right wrapping path)
can easily bring a reasonable two-digit upper bound on the reduced bandwidth of planar graphs,
which can possibly be further tightened with a specialized refined argument.

To finally conclude, the problem to determine the exact maximum value of the twin-width over all planar graphs is still open, 
but our continuing research suggests that the value of~$7$ is very likely the right answer.
Likewise, the problem to determine the exact maximum value of the twin-width over bipartite planar graphs is open,
and we cannot now decide or suggest whether $6$ is the right maximum value over bipartite planar graphs,
or whether the upper bound may possibly be~$5$ (while an upper bound lower than $5$ is not likely since a bipartite construction analogous
to \cite{DBLP:journals/corr/abs-2209-11537} seems to exclude it, but we are not aware of this claim being written up as a formal statement).
In the cases of $1$-planar graphs, and especially of map graphs, the provided bounds can very likely be improved down a bit
with a more detailed and longer case analysis of their proofs.
We are not aware of published nontrivial lower bounds in the latter two cases.

\bibliography{tww}

\begin{thebibliography}{10}

\bibitem{DBLP:journals/siamdm/AhnHKO22}
Jungho Ahn, Kevin Hendrey, Donggyu Kim, and Sang{-}il Oum.
\newblock Bounds for the twin-width of graphs.
\newblock {\em {SIAM} J. Discret. Math.}, 36(3):2352--2366, 2022.

\bibitem{DBLP:conf/gd/BadentBBC09}
Melanie Badent, Michael Baur, Ulrik Brandes, and Sabine Cornelsen.
\newblock Leftist canonical ordering.
\newblock In {\em {GD}}, volume 5849 of {\em Lecture Notes in Computer
  Science}, pages 159--170. Springer, 2009.

\bibitem{DBLP:conf/iwpec/BalabanH21}
Jakub Balab{\'{a}}n and Petr Hlin\v{e}n{\'{y}}.
\newblock Twin-width is linear in the poset width.
\newblock In {\em {IPEC}}, volume 214 of {\em LIPIcs}, pages 6:1--6:13. Schloss
  Dagstuhl - Leibniz-Zentrum f{\"{u}}r Informatik, 2021.

\bibitem{DBLP:conf/isaac/BekosLH022}
Michael~A. Bekos, Giordano~Da Lozzo, Petr Hlin\v{e}n{\'{y}}, and Michael
  Kaufmann.
\newblock Graph product structure for h-framed graphs.
\newblock In {\em {ISAAC}}, volume 248 of {\em LIPIcs}, pages 23:1--23:15.
  Schloss Dagstuhl - Leibniz-Zentrum f{\"{u}}r Informatik, 2022.
\newblock (also {\tt{}arXiv:2204.11495}).

\bibitem{DBLP:conf/icalp/BergeBD22}
Pierre Berg{\'{e}}, {\'{E}}douard Bonnet, and Hugues D{\'{e}}pr{\'{e}}s.
\newblock Deciding twin-width at most 4 is {NP}-complete.
\newblock In {\em {ICALP}}, volume 229 of {\em LIPIcs}, pages 18:1--18:20.
  Schloss Dagstuhl - Leibniz-Zentrum f{\"{u}}r Informatik, 2022.

\bibitem{DBLP:conf/iwpec/BonnetC0K0T22}
{\'{E}}douard Bonnet, Dibyayan Chakraborty, Eun~Jung Kim, Noleen K{\"{o}}hler,
  Raul Lopes, and St{\'{e}}phan Thomass{\'{e}}.
\newblock Twin-width {VIII:} delineation and win-wins.
\newblock In {\em {IPEC}}, volume 249 of {\em LIPIcs}, pages 9:1--9:18. Schloss
  Dagstuhl - Leibniz-Zentrum f{\"{u}}r Informatik, 2022.

\bibitem{DBLP:conf/soda/BonnetGKTW21}
{\'{E}}douard Bonnet, Colin Geniet, Eun~Jung Kim, St{\'{e}}phan Thomass{\'{e}},
  and R{\'{e}}mi Watrigant.
\newblock Twin-width {II:} small classes.
\newblock In {\em {SODA}}, pages 1977--1996. {SIAM}, 2021.

\bibitem{DBLP:conf/icalp/BonnetG0TW21}
{\'{E}}douard Bonnet, Colin Geniet, Eun~Jung Kim, St{\'{e}}phan Thomass{\'{e}},
  and R{\'{e}}mi Watrigant.
\newblock Twin-width {III:} max independent set, min dominating set, and
  coloring.
\newblock In {\em {ICALP}}, volume 198 of {\em LIPIcs}, pages 35:1--35:20.
  Schloss Dagstuhl - Leibniz-Zentrum f{\"{u}}r Informatik, 2021.

\bibitem{DBLP:conf/stoc/BonnetGMSTT22}
{\'{E}}douard Bonnet, Ugo Giocanti, Patrice~Ossona de~Mendez, Pierre Simon,
  St{\'{e}}phan Thomass{\'{e}}, and Szymon Torunczyk.
\newblock Twin-width {IV:} ordered graphs and matrices.
\newblock In {\em {STOC}}, pages 924--937. {ACM}, 2022.

\bibitem{DBLP:conf/soda/BonnetKRT22}
{\'{E}}douard Bonnet, Eun~Jung Kim, Amadeus Reinald, and St{\'{e}}phan
  Thomass{\'{e}}.
\newblock Twin-width {VI:} the lens of contraction sequences.
\newblock In {\em {SODA}}, pages 1036--1056. {SIAM}, 2022.

\bibitem{DBLP:conf/focs/Bonnet0TW20}
{\'{E}}douard Bonnet, Eun~Jung Kim, St{\'{e}}phan Thomass{\'{e}}, and
  R{\'{e}}mi Watrigant.
\newblock Twin-width {I:} tractable {FO} model checking.
\newblock In {\em {FOCS}}, pages 601--612. {IEEE}, 2020.

\bibitem{DBLP:journals/jacm/BonnetKTW22}
{\'{E}}douard Bonnet, Eun~Jung Kim, St{\'{e}}phan Thomass{\'{e}}, and
  R{\'{e}}mi Watrigant.
\newblock Twin-width {I:} tractable {FO} model checking.
\newblock {\em J. {ACM}}, 69(1):3:1--3:46, 2022.

\bibitem{DBLP:journals/corr/abs-2202-11858}
{\'{E}}douard Bonnet, O{-}joung Kwon, and David~R. Wood.
\newblock Reduced bandwidth: a qualitative strengthening of twin-width in
  minor-closed classes (and beyond).
\newblock {\em CoRR}, abs/2202.11858, 2022.
\newblock \href {http://arxiv.org/abs/2202.11858} {\path{arXiv:2202.11858}}.

\bibitem{DBLP:journals/corr/abs-2102-06880}
{\'{E}}douard Bonnet, Jaroslav Nesetril, Patrice~Ossona de~Mendez, Sebastian
  Siebertz, and St{\'{e}}phan Thomass{\'{e}}.
\newblock Twin-width and permutations.
\newblock {\em CoRR}, abs/2102.06880, 2021.

\bibitem{10.1145/506147.506148}
Zhi-Zhong Chen, Michelangelo Grigni, and Christos~H. Papadimitriou.
\newblock Map graphs.
\newblock {\em J. ACM}, 49(2):127--138, mar 2002.

\bibitem{DBLP:journals/jacm/DujmovicJMMUW20}
Vida Dujmovic, Gwena{\"{e}}l Joret, Piotr Micek, Pat Morin, Torsten Ueckerdt,
  and David~R. Wood.
\newblock Planar graphs have bounded queue-number.
\newblock {\em J. {ACM}}, 67(4):22:1--22:38, 2020.

\bibitem{DBLP:conf/icalp/GajarskyPPT22}
Jakub Gajarsk{\'{y}}, Michal Pilipczuk, Wojciech Przybyszewski, and Szymon
  Torunczyk.
\newblock Twin-width and types.
\newblock In {\em {ICALP}}, volume 229 of {\em LIPIcs}, pages 123:1--123:21.
  Schloss Dagstuhl - Leibniz-Zentrum f{\"{u}}r Informatik, 2022.

\bibitem{DBLP:journals/corr/abs-2205.05378}
Petr Hlin\v{e}n{\'{y}}.
\newblock Twin-width of planar graphs is at most 9, and at most 6 when
  bipartite planar.
\newblock {\em CoRR}, abs/2205.05378, 2022.
\newblock \href {http://arxiv.org/abs/2205.05378} {\path{arXiv:2205.05378}}.

\bibitem{DBLP:conf/icalp/HlinenyJ23}
Petr Hlin\v{e}n{\'{y}} and Jan Jedelsk{\'{y}}.
\newblock Twin-width of planar graphs is at most 8, and at most 6 when
  bipartite planar.
\newblock In {\em {ICALP}}, volume 261 of {\em LIPIcs}, pages 75:1--75:18.
  Schloss Dagstuhl - Leibniz-Zentrum f{\"{u}}r Informatik, 2023.
\newblock (also {\tt{}arXiv:2210.08620}).

\bibitem{DBLP:journals/jacm/HopcroftT74}
John~E. Hopcroft and Robert~Endre Tarjan.
\newblock Efficient planarity testing.
\newblock {\em J. {ACM}}, 21(4):549--568, 1974.

\bibitem{DBLP:conf/wg/JacobP22}
Hugo Jacob and Marcin Pilipczuk.
\newblock Bounding twin-width for bounded-treewidth graphs, planar graphs, and
  bipartite graphs.
\newblock In {\em {WG}}, volume 13453 of {\em Lecture Notes in Computer
  Science}, pages 287--299. Springer, 2022.
\newblock (also {\tt{}arXiv:2201.09749}).

\bibitem{DBLP:journals/jgt/KorzhikM13}
Vladimir~P. Korzhik and Bojan Mohar.
\newblock Minimal obstructions for 1-immersions and hardness of 1-planarity
  testing.
\newblock {\em J. Graph Theory}, 72(1):30--71, 2013.

\bibitem{DBLP:journals/corr/abs-2209-11537}
Daniel Kr{\'{a}}l and Ander Lamaison.
\newblock Planar graph with twin-width seven.
\newblock {\em CoRR}, abs/2209.11537, 2022.
\newblock \href {http://arxiv.org/abs/2209.11537} {\path{arXiv:2209.11537}}.

\bibitem{DBLP:conf/stacs/PilipczukSZ22}
Michal Pilipczuk, Marek Sokolowski, and Anna Zych{-}Pawlewicz.
\newblock Compact representation for matrices of bounded twin-width.
\newblock In {\em {STACS}}, volume 219 of {\em LIPIcs}, pages 52:1--52:14.
  Schloss Dagstuhl - Leibniz-Zentrum f{\"{u}}r Informatik, 2022.

\bibitem{DBLP:conf/focs/Thorup98}
Mikkel Thorup.
\newblock Map graphs in polynomial time.
\newblock In {\em {FOCS}}, pages 396--405. {IEEE} Computer Society, 1998.

\end{thebibliography}

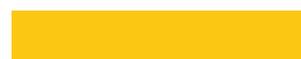
\end{document}